\documentclass[12pt]{amsart}
\pdfoutput=1
\usepackage[left=1in,right=1in,top=1in,bottom=1in]{geometry}
\usepackage[capitalise]{cleveref}

\usepackage{multirow,tikz, blkarray, mathrsfs,scalerel,stackengine,amssymb,amsthm}
\usepackage{setspace}\setdisplayskipstretch{}
\usetikzlibrary{arrows,matrix,positioning}

\stackMath
\newcommand\reallywidehat[1]{\savestack{\tmpbox}{\stretchto{\scaleto{ \scalerel*[\widthof{\ensuremath{#1}}]{\kern.1pt\mathchar"0362\kern.1pt}{\rule{0ex}{\textheight}}}{\textheight}}{2.4ex}}\stackon[-6.9pt]{#1}{\tmpbox}}

\newcommand{\Q}{\mathbb{Q}}
\newcommand{\Z}{\mathbb{Z}}

\newcommand{\K}{\mathbb{K}}

\newcommand{\Lq}{\mathbb{K}}
\newcommand{\ZZ}{\widetilde{Z}}

\newcommand{\Stab}{\operatorname{Stab}}
\newcommand{\Spann}{\text{-}\operatorname{span}}

\newcommand{\en}{e^+_{[1,n-k]}}

\newcommand{\ep}{e^+}
\newcommand{\eplusB}{e^+_{[k+1,n]}}
\newcommand{\Sn}{S_{[1,n-k]}}

\newcommand{\Stabn}{\Stab_{[1,n-k]}}

\newcommand{\Symm}[1]{\sum\limits_{w\in \Sn} T_w \left( #1 \right)}

\newcommand{\shatproduct}{s_1 \dotsm \widehat{s}_{t_1} \dotsm \widehat{s}_{t_r} \dotsm s_k}

\newcommand{\EB}[1]{E_{#1}^\dagger}
\newcommand{\EBlong}[1]{E_{#1}^\dagger(z;q^{-1},t^{-1})}
\newcommand{\EI}[1]{E_{#1}}
\newcommand{\EIlong}[1]{E_{#1}(x;q,t)}

\newcommand{\Estar}[1]{E^*_{#1}}
\newcommand{\Estarlong}[1]{E^*_{#1}(z;q,t)}
\newcommand{\PB}[1]{P_{#1}^\dagger}
\newcommand{\PBlong}[1]{P_{#1}^\dagger(z;q^{-1},t^{-1})}
\newcommand{\PP}{P_{(\lambda \mid \gamma)}}

\newcommand{\Pn}{P_{(\lambda \mid \gamma)}}

\newcommand{\JJ}{\mathcal{J}_{(\lambda \mid \gamma)}}
\newcommand{\MM}{\mathbb{M}_{\gammalambdamin}}
\newcommand{\MMM}{\mathbb{M}_{\lambdagamma}}
\newcommand{\CC}{\mathcal{A}^\lambdagamma_\mueta}

\newcommand{\munu}{{(\mu \mid \nu)}}
\newcommand{\numin}{{\nu^-}}
\newcommand{\nudom}{{( \lambda \mid \gamma )}}
\newcommand{\lambdagamma}{{\nudom}}
\newcommand{\gammalambdamin}{{(\gamma \mid \lambda^-)}}
\newcommand{\gammanu}{{(\gamma \mid \nu)}}
\newcommand{\nuantidom}{{(\lambda^- \mid \gamma)}}

\newcommand{\etamu}{{(\eta \mid \mu^-)}}
\newcommand{\mueta}{{(\mu \mid \eta)}}
\newcommand{\etamuprime}{{(\eta \mid \mu')}}
\newcommand{\mutilde}{{\widetilde{\mu}}}
\newcommand{\lambdatilde}{{\widetilde{\lambda}}}
\newcommand{\gammalambdatilde}{{(\gamma \mid \lambdatilde)}}
\newcommand{\lambdatildegamma}{{(\lambdatilde \mid \gamma)}}
\newcommand{\lambdamingamma}{(\lambda^- \mid \gamma)}
\newcommand{\mumineta}{(\mu^- \mid \eta)}
\newcommand{\flambdatilde}{f_{\lambdatilde , \lambdagamma}}
\newcommand{\etamutilde}{{(\eta \mid \widetilde{\mu})}}
\newcommand{\nugamma}{{(\nu \mid \gamma)}}
\newcommand{\mugamma}{{(\mu \mid \gamma)}}
\newcommand{\mutildeeta}{{(\mutilde \mid \eta)}}
\newcommand{\etamumin}{{(\eta \mid \mu^-)}}
\newcommand{\fmuu}{f_{\mu, \lambdagamma}}

\newcommand{\jj}{j_{(\lambda \mid \gamma)}}
\newcommand{\pIone}{p_{I_1}(\overline{z})}

\newcommand{\pIoneevaluated}{p_{I_1}\left(\overline{\etamutilde}\right)}
\newcommand{\pItwo}{p_{2}(\overline{z})}
\newcommand{\pItwoevaluated}{p_{2}\left(\overline{\etamutilde}\right)}
\newcommand{\mIone}{m_{\etamu}^{\gammalambdamin}}

\newcommand{\hIone}{h}
\newcommand{\dI}{d_I}
\newcommand{\fnu}{f_{\nu, \gammalambdamin}}
\newcommand{\fnuu}{f_{\nu,\lambdagamma}}
\newcommand{\fmu}{f_{\mu,\lambdagamma}}
\newcommand{\fsimu}{{f_{s_i(\mu),\lambdagamma}}}
\newcommand{\gmutilde}{f_{\mutilde,\etamu}}
\newcommand{\fmutilde}{f_{\mutilde, \mueta}}
\newcommand{\gmuprime}{f_{\mu',\etamu}}
\newcommand{\betaa}{{q^{\ell_{s_i(\mu)}(u)+1} t^{a_{s_i(\mu)}(u)}}}

\newcommand{\leg}[1]{{\ell_{#1}(\square)}}
\newcommand{\legnu}{{\ell_\nu(\square)}}
\newcommand{\legnumin}{{\ell_{\nu^-}(\square)}}
\newcommand{\armone}[1]{{\widetilde{a}_{#1}(\square)}}
\newcommand{\armonenu}{{\widetilde{a}_\nu(\square)}}
\newcommand{\armonenumin}{{\widetilde{a}_{\nu^-}(\square)}}
\newcommand{\armtwo}[1]{{a_{#1}(\square)}}
\newcommand{\armtwonu}{{a_\nu(\square)}}
\newcommand{\armtwonumin}{{a_{\nu^-}(\square)}}

\newcommand{\coarm}{{a'_\mu(\square)}}
\newcommand{\symmetricprod}{\prod\limits_{\square \in \lambda^-}}
\newcommand{\nonsymmetricprod}{\prod\limits_{\square \in \gamma}}

\newcommand{\group}[1]{\operatorname{Group}\,\operatorname{#1}}

\newcommand{\jone}{j_1}
\newcommand{\jtwo}{j_2}
\newcommand{\jthree}{j_3}
\newcommand{\jfour}{j_4}
\newcommand{\jfive}{j_5}
\newcommand{\Estarone}{E^*_1}

\newcommand{\Estarthree}{E^*_3}
\newcommand{\Estarfour}{E^*_4}
\newcommand{\Estarfive}{E^*_5}
\newcommand{\ltm}{\ell'_{t_u}}
\newcommand{\lj}{\ell'_j}
\newcommand{\lh}{\ell'_h}

\newcommand{\DBaratta}{D_{\etamumin}^\gammalambdamin}

\newcommand{\dtwo}{d'_I}
\newcommand{\nutilde}{\widetilde{\nu}}
\newcommand{\gammanutilde}{{(\gamma \mid \nutilde)}}
\newcommand{\pstarone}{p^*_{I'_1}}
\newcommand{\pstartwo}{p^*_2}
\newcommand{\pstarthree}{p^*_{I'_3}}

\newcommand{\comment}[1]{}

\newcommand{\dg}{\mathrm{dg}}
\newcommand{\dga}{\widehat{\dg}}

\theoremstyle{plain}
\newtheorem{thm}{Theorem}[section]
\newtheorem{cor}[thm]{Corollary}
\newtheorem{lem}[thm]{Lemma}
\newtheorem{prop}[thm]{Proposition}

\theoremstyle{definition}
\newtheorem{dfn}[thm]{Definition}
\newtheorem{ex}[thm]{Example}
\newtheorem{rem}[thm]{Remark}

\title{Type A Partially-Symmetric Macdonald Polynomials}
\author{Ben Goodberry}
\date{November 20, 2023}

\begin{document}

\begin{abstract}

We construct type A partially-symmetric Macdonald polynomials $\PP$, where $\lambda \in \mathbb{Z}_{\geq 0}^{n-k}$ is a partition and $\gamma \in \mathbb{Z}_{\geq 0}^k$ is a composition. These are polynomials which are symmetric in the first $n-k$ variables, but not necessarily in the final $k$ variables. We establish their stability and an integral form defined using Young diagram statistics. Finally, we build Pieri-type rules for degree 1 products $x_j \PP$ for $j > n-k$ and $e_1[x_1, \dotsc, x_{n-k}] \PP$, along with substantial combinatorial simplification of the $e_1$ multiplication. The $\PP$ are the same as the $m$-symmetric Macdonald polynomials defined by Lapointe in \cite{Lapointe} up to a change of variables.
\end{abstract}

\maketitle

\section{Introduction}
The type $GL_n$ nonsymmetric Macdonald polynomials can be symmetrized, for example by acting on $E_\mu$ by a symmetrization operator over the finite Weyl group, to obtain the classic symmetric Macdonald polynomials. In recent work \cite{Schlosser}, Schlösser constructs a symmetrization over a parabolic subgroup $W_J$ for nonsymmetric Macdonald polynomials in general type. This was done in type $GL_n$ by Lapointe \cite{Lapointe} for a particular choice of parabolic subgroup, with the goal of building a basis for partially-symmetric polynomials which are Schur-positive for some $m$-symmetric form of Schur polynomials. We will use the same type of construction, with the symmetric group $S_{[1,n-k]}$ acting on the first $n-k$ variables as our parabolic subgroup, which gives us polynomials that are stable when adding symmetrized variables. This stability is special to type $A$, and to this parabolic subgroup.

In order to build an integral form $\JJ$ as a normalization of $\PP$, in \cref{diagrams} we introduce Young diagram statistics associated with the split diagram $\lambdagamma$. This includes a different way to count arms depending on whether the column is in the symmetric part of the diagram $\lambda$, or the nonsymmetric part $\gamma$. These coincide with the integral form of the $m$-symmetric Macdonald polynomials in \cite{Lapointe}. We use the combinatorial formula for the nonsymmetric Macdonald polynomials found in \cite{HHL} in order to prove that $\JJ \in \mathbb{Z}[q,t][x_1, \dotsc, x_n]$ in \cref{integralformsection}. This is mentioned but not proven in \cite{Lapointe}.

The major new identities we explore that do not seem to appear elsewhere are rules for multiplication of $\PP$ by either $e_1[x_1, \dotsc, x_{n-k}]$, the elementary symmetric polynomial in all of the symmetrized variables, or by $x_j$ where $j>n-k$, a nonsymmetric variable. We derive these formulas in \cref{ElementaryMultiplication}. Our approach for finding formulas for these products, written in the basis of partially-symmetric Macdonald polynomials, uses a similar approach by Baratta in \cite{Ba} for the product $x_i E_\mu$. Namely, this makes heavy use of the interpolation nonsymmetric Macdonald polynomials of Knop \cite{KnopSymmetric} and Sahi \cite{Sahi}. Halverson and Ram \cite{Monk} find Monk type formulas for the same products, but instead using Cherednik's intertwiners. While these formulas could be symmetrized in a naive way, combining the resulting expansion to get it into the basis of partially-symmetric Macdonald polynomials is quite difficult, as many terms combine to the same $P_{\mueta}$. We get around this using properties of the Hecke operators and by evaluating the interpolation polynomial expressions at a special point which forces uniqueness and avoids any overlapping of terms.

Our primary goal is to build the polynomials which appear in \cite{CGM} as isomorphic images of fixed points of a torus action in the Parabolic Flag Hilbert scheme, which will be established in upcoming works with D. Orr and M. Bechtloff Weising. Toward this end, after finding a formula for the expansion of $e_1[x_1, \dotsc, x_{n-k}] \PP$, in \cref{Cancellation} we categorize many combinatorial simplifications that can be made to that formula.

This article includes the majority of the results submitted for a doctoral dissertation in \cite{Dissertation}, though the final simplifications and $x_j \PP$ formulas do not appear there.

\section*{Acknowledgements} I would like to thank Daniel Orr for his extensive guidance over the course of this project. I would also like to thank Luc Lapointe for helpful discussion on the $m$-symmetric Macdonald polynomials.

\section{Partially symmetric Macdonald polynomials}

\subsection{Nonsymmetric Macdonald polynomials}

Let $\K=\Q(q,t)$ and $n>0$. For any composition $\nu\in(\Z_{\ge 0})^n$, denote by $E_\nu =E_\nu(x;q,t)\in \K[x_1,\dotsc,x_n]$ the nonsymmetric Macdonald polynomial of type $GL_n$ as defined in  \cite{HHL}. These satisfy the triangularity property,
\begin{equation}\label{Etriangularity} E_\mu \in x^\mu + \sum_{\nu <_{\text{(Bru)}} \mu} c_\nu x^\nu,\end{equation}
where $x^\mu = x_1^{\mu_1} \dotsm x_n^{\mu_n}$ and $<_{\text{(Bru)}}$ is the Bruhat order. In this ordering, we consider a weakly-decreasing weight to be dominant, and a weakly-increasing weight antidominant.

The symmetric group $S_n$ acts naturally on $\Z^n$ and $\K[x_1,\dotsc,x_n]$. For $1\le i<n$, let $s_i$ be the simple transposition $(i,i+1)\in S_n$. Define the Demazure-Lusztig operators,
\begin{align*}
T_i = ts_i + \frac{(t-1)x_{i+1}}{x_{i+1}-x_i}(1-s_i), \qquad 1\le i<n.
\end{align*}
These satisfy the braid relations,
\begin{equation}\label{braidrelations}T_i T_{i+1} T_i = T_{i+1} T_i T_{i+1}, \qquad 1\leq i < n\end{equation}
\begin{equation} T_i T_j = T_j T_i \qquad \text{ if } |i-j|\geq 2\end{equation}
and the quadratic relations,
\begin{equation}\label{quadratic}
(T_i-t)(T_i+1)=0.
\end{equation}
For any reduced expression $w=s_{i_1}\dotsm s_{i_\ell}\in S_n$, we define the operator $T_w=T_{i_1}\dotsm T_{i_\ell}$. 

\subsection{Partially symmetric Macdonald polynomials}
Suppose $n>0$ and $k\geq 0$. We can regard $\Sn$ as the subgroup of $S_n$ fixing the elements $\{n-k+1,\dotsc,n\}$. Define the partial Hecke symmetrizer
$$
\en=\sum_{w\in \Sn} T_w.
$$
For $\lambda\in(\Z_{\ge 0})^{n-k}$, let $(\Sn)_\lambda\subseteq \Sn$ be its stabilizer and
\begin{align*}
W_\lambda(t)=\sum_{w\in (\Sn)_\lambda} t^{\ell(w)}.
\end{align*}

We write $\nu = (\lambda \mid \gamma)$ to indicate a splitting of a composition $\nu$ into two compositions, $\lambda$ and $\gamma$, i.e., $\nu=(\nu_1,\dotsc,\nu_n)=(\lambda_1,\dotsc,\lambda_{n-k},\gamma_1,\dotsc,\gamma_k)$ where $\lambda=(\lambda_1,\dotsc,\lambda_{n-k})\in(\Z_{\ge 0})^{n-k}$ and $\gamma=(\gamma_1,\dotsc,\gamma_k)\in(\Z_{\ge 0})^k$.

\begin{dfn}
For a weakly-decreasing partition $\lambda \in (\mathbb{Z}_{\geq 0})^{n-k}$ and a composition $\gamma\in (\Z_{\ge 0})^k$, the partially symmetric Macdonald polynomial $P_{(\lambda|\gamma)}=P_{(\lambda|\gamma)}(x;q,t)$ is defined by
\[
P_{(\lambda|\gamma)} = \frac{\en  E_{(\lambda|\gamma)}}{W_\lambda(t)}.
\]
\end{dfn}

\subsection{Diagrams and their statistics}\label{diagrams}
In constructions related to compositions and their diagrams, we mostly follow the conventions of \cite{HHL}. Compositions are tuples $\nu=(\nu_1,\dotsc,\nu_n)\in(\Z_{\ge 0})^n$. A composition $\nu$ is a partition if its entries are weakly decreasing. 

The diagram of a composition $\nu\in(\Z_{\ge 0})^n$ is the subset $\dg(\nu)\subset (\Z_{\ge 0})^2$ given by
$$ \dg(\nu)=\{(i,j) \mid 1\le i\le n, 1\le j\le\nu_i\}. $$
We view the parts $\nu_i$ of $\nu$ as columns in $\dg(\nu)$. Elements of $\dg(\nu)$ are called boxes of $\nu$. We also define the following subsets of $\dg(\nu)$:
\begin{align*} 
\dg_r(\nu) &= \{(i,j) \in \nu \mid j=r \} \\
\dg_{>r}(\nu) &= \{(i,j) \in \nu \mid j>r\} \\
\dg_{top}(\nu) &= \{(i,\nu_i) \mid 1 \leq i \leq n\}.
\end{align*}

The standard leg and arm lengths of a box $\square=(i,j)\in\nu$ are defined by
\begin{align*}
\leg{\nu} &= \nu_i -j,\\
\armtwonu &= \# \{ 1 \leq r < i \mid j\leq \nu_r \leq \nu_i\} + \# \{ i<r \leq n \mid j-1\leq \nu_r < \nu_i\}.
\end{align*}
We will also use the following alternate versions of arm length (see Example~\ref{DiagramExample}):
\begin{align*}
\armonenu &= \# \{1\leq r < i \mid j \leq \nu_r \leq \nu_i \} + \# \{i<r \leq n \mid j \leq \nu_r < \nu_i \}, \\
\coarm &= \# \{1\leq r<i \mid \nu_i < \nu_r \} + \# \{i<r \leq n \mid \nu_i \leq \nu_r \}.\end{align*}
The quantity $\coarm$ will be called the coarm of the box.
Starting with $\nu=(\lambda \mid \gamma)$ as above and $\lambda$ a partition, we construct an augmented diagram $\dga(\nu)$ as follows. First, we form $\nu^- := (\lambda^- \mid \gamma)$, where $\lambda^-$ is the weakly increasing rearrangement of $\lambda$. The augmented diagram associated with $\nu$ is then defined as follows
\[ \dga(\nu)=\dg(\nu^-)\cup \{(n-k+1,0),\dotsc,(n,0)\}. \]

We call the subsets of $\dg(\nu^-)$ corresponding to $\lambda^-$ and $\gamma$ the symmetric and nonsymmetric parts of the diagram, respectively. We will use different arm functions for boxes in these two parts of the diagram- for boxes in the symmetric part of $\dg(\nu^-)$, we will use $\widetilde{a}_{\nu^-}$, and for boxes in the nonsymmetric part, we will use the arm function $a_{\nu^-}$. These arm functions have interpretations as counting certain boxes in $\dga(\nu^-)$, as we illustrate in the example below.

\begin{ex}\label{DiagramExample} For $\nu=(3,1  \mid  2,1,3,0,1)$, the following diagrams illustrate the boxes counted as arms and legs for the nonsymmetric box $u=(3,1)$ and symmetric box $v=(2,1)$ respectively.

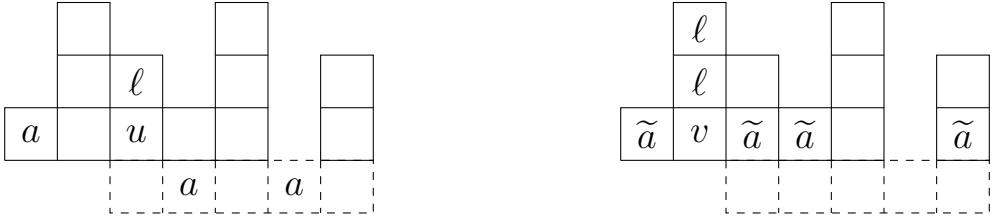
\begin{center}
\begin{tikzpicture}[scale=.7]
\draw (0,0) -- (0,1) -- (5,1) -- (5,0) -- (0,0);
\draw (1,0) -- (1,3) -- (2,3) -- (2,0);
\draw (1,2) -- (3,2) -- (3,0);
\draw (4,0) -- (4,3) -- (5,3) -- (5,1);
\draw (4,2) -- (5,2);
\draw (6,0) -- (6,2) -- (7,2) -- (7,0) -- (6,0);
\draw (6,1) -- (7,1);
\draw[dashed] (2,0) -- (2,-1) -- (7,-1) -- (7,0) -- (2,0);
\draw[dashed] (3,0) -- (3,-1);
\draw[dashed] (4,0) -- (4,-1);
\draw[dashed] (5,0) -- (5,-1);
\draw[dashed] (6,0) -- (6,-1);
\node at (.5,.5) {\large $ a$};
\node at (2.5,.5) {\large $ u$};
\node at (2.5,1.5) {\large $ \ell$};
\node at (3.5,-.5) {\large $ a$};
\node at (5.5,-.5) {\large $ a$};
\node at (6.5,-.5) {\large $ $};
\end{tikzpicture}
\hspace*{3cm}
\begin{tikzpicture}[scale=.7]
\draw (0,0) -- (0,1) -- (5,1) -- (5,0) -- (0,0);
\draw (1,0) -- (1,3) -- (2,3) -- (2,0);
\draw (1,2) -- (3,2) -- (3,0);
\draw (4,0) -- (4,3) -- (5,3) -- (5,1);
\draw (4,2) -- (5,2);
\draw (6,0) -- (6,2) -- (7,2) -- (7,0) -- (6,0);
\draw (6,1) -- (7,1);
\draw[dashed] (2,0) -- (2,-1) -- (7,-1) -- (7,0) -- (2,0);
\draw[dashed] (3,0) -- (3,-1);
\draw[dashed] (4,0) -- (4,-1);
\draw[dashed] (5,0) -- (5,-1);
\draw[dashed] (6,0) -- (6,-1);
\node at (1.5,.5) {\large $v$};
\node at (1.5,1.5) {\large $\ell$};
\node at (1.5,2.5) {\large $\ell$};
\node at (.5,.5) {\large $\widetilde{a}$};
\node at (2.5,.5) {\large $\widetilde{a}$};
\node at (3.5,.5) {\large $\widetilde{a}$};
\node at (6.5,.5) {\large $\widetilde{a}$};
\end{tikzpicture}
\end{center}

\end{ex}

It will prove useful to have formulas for the action of $T_i$ on $E_\mu$, which is most simply expressed using diagram statistics.

\begin{lem}\label{TiFormulas} We have the following formulas for the action of $T_i$ on the nonsymmetric Macdonald polynomials.
\begin{itemize} \item\hspace*{0mm}{\cite[(17)]{HHL}} If $\mu_{i}>\mu_{i+1}$ for some $1\le i<n$, then
\begin{align}
\label{E:T-on-E-up}
T_i E_\mu = E_{s_i(\mu)}-\frac{1-t}{1-q^{\ell_\mu(\square)+1}t^{a_\mu(\square)}}E_\mu
\end{align}
where $\square=(i,\mu_{i+1}+1)\in\dg(\mu)$.\\

\item\hspace*{0mm}[Proof in Appendices] If $s_i(\mu)<\mu$, and $u=(i,\mu_{i}+1)$, then
\begin{equation} \label{decreasingT} T_iE_\mu = \frac{q^{\ell(u)+1}t^{a(u)}(1-t)}{1-q^{\ell(u)+1}t^{a(u)}}E_{\mu} - \frac{(t-q^{\ell(u)+1}t^{a(u)})(1-q^{\ell(u)+1}t^{a(u)+1})}{(1-q^{\ell(u)+1}t^{a(u)})^2} E_{s_i(\mu)},\end{equation}
where all arms and legs are in terms of the diagram of $s_i(\mu)$.\\

\item If $\mu_i = \mu_{i+1}$, then
\begin{equation}\label{simpledz} T_i E_\mu = t E_\mu.\end{equation}

\item Applying the above formulas, for any $w\in \Sn$,
\begin{equation}\label{span}
    T_w(E_\mu) \in \Lq \Spann \{ E_{w'(\mu)} \}_{w' \in \Sn}.
\end{equation}

\end{itemize}
\end{lem}

\section{Properties of Partially-Symmetric Macdonald Polynomials}
\subsection{Stability}
The polynomials $\PP$ exhibit stability in the following sense. Considering the projection of polynomials in $\mathbb{K}[x_1,\dotsc, x_n]$,
\[ \pi_1(x_1^{\lambda_1} \dotsm x_n^{\lambda_n}) := \begin{cases} x_1^{\lambda_2} \dotsm x_{n-1}^{\lambda_n} &\text{ if } \lambda_1=0\\
0 &\text{ if } \lambda_1 > 0\end{cases}. \]

This is equivalent to setting $x_1 \mapsto 0$ and shifting the remaining indices down by 1.

\begin{prop} If $\lambda \in \Z_{\geq 0}^{n-k}$ is an $S_{[1,n-k]}$-dominant weight and $\gamma \in \Z_{\geq 0}^k$, then
\begin{align}\label{stability}\pi_1 P_{(\lambda,0\mid\gamma)} =  P_{(\lambda\mid\gamma)}.\end{align}\end{prop}

This is proved in \cite{Lapointe} for the $m$-symmetric Macdonald Polynomials. An explicit computation in our notation can be found in \cite{Dissertation}.

\begin{rem}
    The partially-symmetric Macdonald polynomials form a basis for partially-symmetric polynomials in $\Lq[x_1, \dotsc, x_n]^{\Sn}$, which is a direct result of the fact that the $E_\mu$ form a basis for $\Lq[x_1,\dotsc,x_n]$.
\end{rem}

\subsection{Integral Form}\label{integralformsection}
The integral form partially symmetric Macdonald polynomial $J_{(\lambda|\gamma)}=J_{(\lambda|\gamma)}(x;q,t)$ is the scalar multiple
\[
J_{(\lambda|\gamma)}= j_{(\lambda|\gamma)}P_{(\lambda|\gamma)}
\]
where
\[
\jj := \symmetricprod (1-q^{\ell_{\nu^-}(\square)}t^{\tilde{a}_{\nu^-}(\square) + 1}) \nonsymmetricprod(1-q^{\ell_{\nu^-}(\square) + 1}t^{a_{\nu^-}(\square) + 1}).
\]

In the definition of $j_{(\lambda|\gamma)}$, the products are taken over the symmetric and nonsymmetric parts of $\dg(\nu^-)$, respectively. That is, $\displaystyle{\symmetricprod}=\displaystyle{\prod_{\substack{(i,j) \in \nu^-\\ 1\leq i \leq n-k}}}$ and $\displaystyle{\nonsymmetricprod}=\displaystyle{\prod_{\substack{(i,j) \in \nu^-\\ n-k+1 \leq i \leq n}}}$.

\begin{rem}
    The integral form normalization agrees with the one in \cite{Lapointe}.
\end{rem}

Before proving integrality of $\JJ$, we will break it down to pieces which we will prove individually. Fist, converting to the integral form of the nonsymmetric Macdonald polynomials found in \cite[Eqn. (6.11)]{KnopComposition}, denoted $\mathcal{E}_\nu$, and reorganizing the denominator, we get:

\begin{align*}\JJ &= \frac{\symmetricprod (1-q^{\legnumin}t^{\armonenumin + 1}) \nonsymmetricprod(1-q^{\legnumin + 1}t^{\armtwonumin + 1})}{  W^{n-k}_\lambda(t)} \frac{\Symm{\mathcal{E}_\nudom}}{\prod\limits_{\square \in \nu} \left(1-q^{\legnu+1}t^{\armtwonu +1} \right)} \\
&= \frac{\symmetricprod (1-q^{\legnumin}t^{\armonenumin + 1}) }{  \prod\limits_{\square \in d_{>1}(\lambda)} \left(1-q^{\legnu+1}t^{\armtwonu +1} \right)  } \cdot \frac{\nonsymmetricprod(1-q^{\legnumin + 1}t^{\armtwonumin + 1})}{\prod\limits_{\square \in \gamma} \left(1-q^{\legnu+1}t^{\armtwonu +1} \right)}   \\
&\hspace*{3cm} \cdot  \frac{\Symm{\mathcal{E}_{(\lambda \mid \gamma)}}}{W^{n-k}_\lambda(t)\prod\limits_{\square \in d_1(\lambda)} \left(1-q^{\legnu+1}t^{\armtwonu +1} \right)} \end{align*}

The middle products cancel immediately, since the ordering of the columns of $\lambda$ does not change the arm length of a box in $\gamma$. So this can be simplified to
\begin{equation}\label{simplifiedJ} \JJ = \frac{\symmetricprod (1-q^{\legnumin}t^{\armonenumin + 1}) }{  \prod\limits_{\square \in d_{>1}(\lambda)} \left(1-q^{\legnu+1}t^{\armtwonu +1} \right)  } \cdot  \frac{\Symm{\mathcal{E}_{(\lambda \mid \gamma)}}}{W^{n-k}_\lambda(t)\cdot \prod\limits_{\square \in d_1(\lambda)} \left(1-q^{\legnu+1}t^{\armtwonu +1} \right)}. \end{equation}

Then to complete the proof that $\JJ$ is integral, we will show that
\begin{align}\label{HHLpiece} \frac{\mathcal{E}_\nu}{\prod\limits_{ \square \in d_1(\lambda) }(1-q^{\legnu+1}t^{\armtwonu +1})} \in \Z[q,t][x_1, \dotsc, x_n],\end{align}
and also,
\begin{align}\label{coefficientpiece} \frac{\symmetricprod (1-q^{\legnumin}t^{\armonenumin + 1}) }{  \prod\limits_{\square \in d_{>1}(\lambda)} \left(1-q^{\legnu+1}t^{\armtwonu +1} \right)  } = \prod\limits_{\square \in d_{top}(\lambda^-)} (1-q^\legnumin t^{\armonenumin + 1}). \end{align}

We will make use of the combinatorial formula for the nonsymmetric Macdonald polynomials from \cite{HHL} in order to prove \cref{HHLpiece}.
\begin{lem}\label{redundancy} For any $\nu = \nudom$ where $\lambda$ has weakly decreasing entries 
\[ \frac{\mathcal{E}_\nu}{\prod\limits_{ \square \in d_1(\lambda) }(1-q^{\legnu+1}t^{\armtwonu +1})} \in \Z[q,t][x_1, \dotsc, x_n]. \]
\end{lem}
\begin{proof}
First, we will convert $\mathcal{E}_\nu$ back into its monic form $E_\nu$,
\[ \frac{\mathcal{E}_\nu}{\prod\limits_{ \square \in d_1(\lambda) }(1-q^{\legnu+1}t^{\armtwonu +1})} = \prod\limits_{\square \in d_{>1}(\lambda) }(1-q^{\legnu+1}t^{\armtwonu + 1}) \nonsymmetricprod(1-q^{\legnu+1}t^{\armtwonu + 1}) E_\nu.\]
Consider the combinatorial formula for $E_\nu$ given in ~\cite{HHL},
\[E_\nu = \sum\limits_{\substack{\sigma :\, \nu \to [n] \\ \text{nonattacking}}} x^\sigma q^{maj(\widehat{\sigma})}t^{coinv(\widehat{\sigma})} \prod\limits_{\substack{\square \in  dg(\nu) \\ \widehat{\sigma}(\square) \neq \widehat{\sigma}(d(\square))}} \frac{1-t}{1-q^{\legnu+1} t^{\armtwonu+1}}.\]

The following definitions are taken directly from \cite{HHL}. A filling $\sigma$ of the diagram $\nu$ is a function assigning a number $i\in \{1, \dotsc, n \} $ to each box in the diagram, which we will call the box's label.\\

Two boxes $(i,j)$ and $(i',j')$ are said to be attacking if either,
\begin{itemize}
    \item The boxes are in the same row, i.e. $j = j'$, or
    \item The boxes are in adjacent rows and different columns, and the box in the lower row is to the right of the box in the higher row, i.e. the boxes are of the form $(i,j)$ and $(i-1,j')$ where $j' > j$.
\end{itemize}
In this formula, $dg(\nu)$ is the diagram of $\nu$ with no basement row. The term $d(\square)$ is the box immediately below $\square$, so the product is over every box in the diagram $dg(\nu)$ which does not have a box below it with the same label. Then for a filling $\sigma$, the augmented filling $\widehat{\sigma}$ is a filling of the diagram $dg(\nu) \cup \{ (i,0) \mid 1\leq i \leq n\}$ where $\widehat{\sigma}(i,0)=i$.

A filling $\widehat{\sigma}$ of the augmented diagram is attacking if there is a pair of boxes in the augmented diagram which are attacking and which have the same label. And $\sigma$ is said to be a nonattacking filling of $\nu$ if the corresponding filling $\widehat{\sigma}$ of the augmented diagram is nonattacking. This means for $\sigma$ to be nonattacking, every pair of attacking boxes in the augmented diagram must have distinct values.

Now suppose $\sigma$ is a nonattacking filling of $\nu$. Since $\lambda$ is decreasing, every column $i$ in the $\lambda$ part of the augmented diagram where $\lambda_i \neq 0$ has boxes in the basement row and row 1. By assumption, each box $(i,0)$ contains the value $i$. So if we start filling in the boxes $(j,1)$ from left to right, the box $(1,1)$ is attacking every box $(j,0)$ where $j>1$, all the boxes to the right of column 1 in row 0. Continuing to the right, the box $(i,1)$ is attacking the boxes $(k,1)$ where $k<i$, and the boxes $(j,0)$ where $j>i$. Inductively, the boxes $(k,1)$ for $k<i$ are labeled $k$, and the boxes $(j,0)$ for $j>i$ are labeled $j$. So for the box to be nonattacking, $(i,1)$ must have label $i$. Then altogether, $(i,1)$ has label $i$ where $\lambda_i \neq 0$.

Finally, since $\sigma$ is nonattacking and must have the filling described above, for any $\square = (i,1)$ where $1\leq i \leq n-k$, we have $\widehat{\sigma}((i,j))=\widehat{\sigma}((i,j-1))=i$. So no box in the first row of $\lambda$ contributes to the product $\displaystyle{\prod\limits_{\substack{\square \in  dg(\nu) \\ \widehat{\sigma}(\square) \neq \widehat{\sigma}(d(\square))}}} \frac{1-t}{1-q^{\legnu+1} t^{\armtwonu(\square)+1}}$. That means that the product, \[ \prod\limits_{\substack{\square \in \nu, \\ \square \notin d_1(\lambda) }}(1-q^{\legnu+1}t^{\armtwonu + 1}) = \prod\limits_{\square \in d_{>1}(\lambda) }(1-q^{\legnu+1}t^{\armtwonu + 1}) \nonsymmetricprod(1-q^{\legnu+1}t^{\armtwonu + 1}),\] is sufficient to cancel all the denominators in the combinatorial formula for $E_\numin$. Therefore for this normalized form of the nonsymmetric Macdonald polynomials, \[\prod\limits_{\substack{\square \in \numin, \\ \square \notin d_1(\lambda^-) }}(1-q^{\legnumin+1}t^{\armtwonumin + 1}) E_\numin \in \mathbb{Z}[z,t][x_1,\dotsc,x_n].\qedhere\]
\end{proof}
\begin{lem}\label{lambdalemma} If $\lambda$ is weakly decreasing,

\begin{equation}\label{lambdaequality}\frac{\symmetricprod (1-q^{\legnumin}t^{\armonenumin + 1}) }{  \prod\limits_{\square \in d_{>1}(\lambda)} \left(1-q^{\legnu+1}t^{\armtwonu +1} \right)  } = \prod\limits_{\square \in d_{top}(\lambda^-)} (1-q^\legnumin t^{\armonenumin + 1}).\end{equation}\end{lem}

\begin{proof}
Let $w_{\lambda^-}\in \Sn$ be the shortest element such that $w_{\lambda^-}(\lambda)=\lambda^-$. Note that $w_{\lambda^-}$ does not change the ordering of columns of the same height relative to each other. Consider the boxes $u=(i,j)$ in $d_{>1}(\lambda)$ inside the diagram $\nudom$, and $u'=(w_{\lambda^-}(i),j-1)$, the box in $\nuantidom$ immediately below the box corresponding to $u$ after reversing $\lambda$. Since $u$ is not in the first row of $\lambda$, the box $u'$ appears in $\lambda^-$ and not in $d_{top}(\lambda)$.

We wish to show that $\widetilde{a}_\numin(u')=a_\nu(u)$ and $\ell_{\numin}(u')=\ell_\nu(u)+1$. The leg equality is immediate, since $u$ and $u'$ are in columns of the same height and $u'$ is one row below $u$, and thus has one more leg than $u$.

To count the arms in $\widetilde{a}_\numin(u')$ and $a_\nu(u)$, consider the arms which are counted in the following cases:

\begin{itemize} \item Consider a column in $\gamma$ which could be counted in the right arms of $u$ and $u'$. The same box is counted in the right arms, using the definitions of $a$ and $\widetilde{a}$, since $a_\nu(u)$ counts boxes in the row below $u$, and $\widetilde{a}_\numin(u')$ counts boxes in the same row as $u'$, which are the same row, $j-1$. And the criterion for the box to be counted is the same in either case.

\item In $\lambda$ and $\lambda^-$, columns of height $\lambda_i$ are only counted if they are to the left of the box $u$ or $u'$ in their respective diagrams. Because those columns maintain the same order relative to each other, and left arms in $a$ and $\widetilde{a}$ are counted the same, $u$ and $u'$ have the same number of arms in those columns. 

\item Columns of height $\lambda_i-1$ appear to the right of $u$ in $\nu$ since $\lambda$ is weakly decreasing, and to the left of $u'$ in $\numin$ since $\lambda^-$ is weakly increasing. Then since the columns have fewer boxes than $\lambda_i$, these columns contribute to right right arm in $a_\nu(u)$. Similarly, the columns contribute to the left arm of in $\widetilde{a}_\numin(u')$.

\item Columns of height greater than $\lambda_1$ or less than $\lambda_1-1$ do not contribute to either $\widetilde{a}_\numin(u')$ or $a_\nu(u)$.

\end{itemize}
Altogether, this implies that $\widetilde{a}_\numin(u') = a_\nu(u)$, and from that we can see that, \[1-q^{\ell_\numin(u')}t^{\widetilde{a}_\numin(u') + 1} = 1-q^{\ell_\nu(u)+1}t^{a_\nu(u) +1}.\]
This allows us to cancel most of the terms in the quotient on the left side of \cref{lambdaequality}. The terms coming from the boxes in $d_{>1}(\lambda)$ correspond to the terms from the boxes below them in $\lambda^-$. What's left in the numerator are terms from the top horizontal strip, $d_{top}(\lambda^-)$.
\end{proof}

Now to complete the proof that $\JJ$ is integral, it remains to show that the Poincar\'e polynomial $W^{n-k}_\lambda(t)$ is cancelled by the symmetrization of $E_\nu$.

\begin{thm}\label{Integrality}
If we expand $\displaystyle{\JJ = \sum_\mu c_{ \mu} x^\mu}$, with $\lambda$ weakly decreasing, then $c_{ \mu} \in \Z[q,t] $ for any $\mu$.
\end{thm}

\begin{proof} 
Applying \cref{lambdalemma} to \cref{simplifiedJ}, we can start by writing,
\[  \JJ =  \prod\limits_{\square \in d_{top}(\lambda^-)} (1-q^{\legnumin} t^{\armonenumin + 1}) \cdot \frac{1}{W^{n-k}_\lambda(t)}  \Symm{\frac{\mathcal{E}_{(\lambda \mid \gamma)}}{\prod\limits_{\square \in d_1(\lambda)} \left(1-q^{\legnu+1}t^{\armtwonu +1} \right)}} . \]
 Looking at the symmetrizer and Poincar\'e polynomial, and decomposing each $w$ into $w=w'w''$ where $w'\in \Stabn$ and $w''$ is the minimal length coset representative of $w\Stabn$ in $W^\lambda: = \Sn / \Stabn$, we get,
\[ \frac{1}{W^{n-k}_\lambda(t)} \sum\limits_{w\in \Sn} T_w = \frac{1}{\sum\limits_{w\in \Stabn(\lambda)} t^{\ell(w)}} \sum\limits_{w''\in W^\lambda}T_{w''} \sum\limits_{w'\in \Stabn}T_{w'}. \]
Then using \cref{simpledz}, the action of $T_{w'}$ is just multiplication by $t^{\ell(w')}$, and so,
\[ \frac{1}{W^{n-k}_\lambda(t)} \sum\limits_{w\in \Sn} T_w = \sum\limits_{w''\in W^\lambda}T_{w''}. \]
Finally, we can consider $\JJ$ as,
\[ \JJ= \prod\limits_{\square \in d_{top}(\lambda^-)} (1-q^{\legnumin} t^{\armonenumin + 1}) \cdot \sum\limits_{w''\in W^\lambda}T_{w''}  \left(\frac{\mathcal{E}_{(\lambda \mid \gamma)}}{\prod\limits_{\square \in d_1(\lambda)} \left(1-q^{\legnu+1}t^{\armtwonu +1} \right)} \right).\]

Since the operators $T_i$ preserve $\Z[q,t][x_1,\dotsc,x_n]$, and from \cref{HHLpiece}, the function inside the symmetrizer is in $\Z[q,t][x_1,\dotsc,x_n]$, we conclude that $\JJ \in \Z[q,t][x_1,\dotsc,x_n]$.
\end{proof}

\vspace*{.5cm} By definition, $P_{\lambdagamma}$ specializes to the symmetric and nonsymmetric Macdonald polynomials if we take $P_{(\lambda \mid \emptyset)}$ and $P_{(\emptyset \mid \gamma)}$ respectively. The formula for $\jj$ also specializes to the integrality constants of nonsymmetric Macdonald polynomials in \cite{MR3443860}, and in the symmetric Macdonald polynomials in \cite[section VI.8]{MR1427661}, so that
\[ j_{(\lambda \mid \emptyset)} P_{(\lambda \mid \emptyset)} = J_\lambda \hspace{.5cm} \text{ and } \hspace{.5cm}  j_{(\emptyset \mid \gamma)} P_{(\emptyset \mid \gamma)} = \mathcal{E}_\gamma,\]
meaning $\JJ$ specializes to the integral forms of the symmetric and nonsymmetric Macdonald polynomials respectively. Additionally, because the diagram for $\jj$ is constructed with $\lambda$ rewritten in weakly increasing order, any extra 0 entries appended to $\lambda$ will not change the diagram, and thus $j_{(\lambda \mid \gamma)} = j_{(\lambda, 0 \mid \gamma)}$ for any $\lambda$ and $\gamma$. Combining this with the stability of $\PP$, we have stability of $\JJ$ as well.

	\subsection{Expansion into nonsymmetric Macdonald basis}
	The nonsymmetric Macdonald polynomials form a basis for $\Lq[x_1, \dotsc, x_n]$, so every partially-symmetric Macdonald polynomial can be written in that basis. Additionally, we know from \cref{TiFormulas} that for weakly decreasing partition $\lambda$ and composition $\gamma$,
\[ \Pn \in \Lq \Spann \{ E_{(w(\lambda) \mid \gamma)}\}_{w\in \Sn }.\]
It will prove useful in later computations to have formulas for the coefficients of that expansion in the nonsymmetric Macdonald basis.

\begin{prop}\label{Pexpansion}
In the expansion of $\PP$ into the nonsymmetric Macdonald basis,
\[ \PP = \sum_{\mu \in S_{[1,n-k]}(\lambda)} \fmu E_{\mugamma},\]
the coefficients satisfy the recurrence relation,
\begin{align} 
f_{\lambda^-,\lambdagamma} &= 1 \label{identity}\\
\fsimu &= \fmu \cdot \frac{t-\betaa}{1-\betaa},\label{recursion}\end{align}
if $s_i(\mu) < \mu$ and $u=(i, \mu_{i}+1)$.
\end{prop}
\begin{proof}
The identity \cref{identity} can be found using the formulas in \cref{TiFormulas} and the triangularity of the nonsymmetric Macdonald polynomials. Let $\mu \in S_{[1,n-k]}(\lambda)$ be some weight such that $s_i(\mu)<\mu$ and define $u=(i,\mu_{i}+1)$. Since $\PP$ is invariant under the permutation $s_i$,

\[ T_i \PP = t \PP.\]
In the expansion into the nonsymmetric Macdonald basis, writing $\nu \sim \lambda$ to mean $\nu$ is a rearrangement of $\lambda$ (equivalently, $\nu \in S_{[1,n-k]}(\lambda)$),
\[ \sum_{\nu \sim \lambda} \fnuu T_i E_{\nugamma} = \sum_{\nu \sim \lambda} \fnuu t E_{\nugamma}.   \]
The recursive identity can be recovered by comparing the coefficients of $E_{(s_i(\mu) \mid \gamma)}$ on each side. On the right, we simply have $t\cdot f_{s_i(\mu),\lambdagamma}$. On the left, because $T_i E_{(\mu \mid \gamma)}$ is a linear combination of $E_{(\mu \mid \gamma)}$ and $E_{(s_i(\mu) \mid \gamma)}$, the only possible values of $\nu$ which contribute to the coefficient of $E_{(s_i(\mu) \mid \gamma)}$ are $E_\mugamma$ and $E_{(s_i(\mu) \mid \gamma)}$. Once again considering all arms and legs to be with respect to the diagram of $s_i(\mu)$, computing both possible contributions gives,
\begin{align*}
     T_i E_{\mugamma} &=  \frac{q^{\ell(u)+1}t^{a(u)}(1-t)}{1-q^{\ell(u)+1}t^{a(u)}} E_{\mugamma} +  \frac{(t-q^{\ell(u)+1}t^{a(u)})(1-q^{\ell(u)+1}t^{a(u)+1})}{(1-q^{\ell(u)+1}t^{a(u)})^2} E_{(s_i(\mu) \mid \gamma)}\\
    T_i E_{(s_i(\mu)\mid \gamma)} &=  E_{\mugamma} +  \frac{t-1}{1-q^{\ell(u)+1}t^{a(u)}} E_{(s_i(\mu)\mid \gamma)}
\end{align*}
Thus the coefficient of $E_{(s_i(\mu)\mid \gamma)}$ on the left side is,
\[ f_{s_i(\mu),\lambdagamma} \cdot \frac{t-1}{1-q^{\ell(u)+1}t^{a(u)}} + f_\mu  \cdot \frac{(t-q^{\ell(u)+1}t^{a(u)})(1-q^{\ell(u)+1}t^{a(u)+1})}{(1-q^{\ell(u)+1}t^{a(u)})^2}.\]
This must be equal to $t\cdot f_{s_i(\mu),\lambdagamma}$, which implies,
\[ f_{s_i(\mu),\lambdagamma} \left( t-\frac{t-1}{1-q^{\ell(u)+1}t^{a(u)}} \right) = f_\mu  \cdot \frac{(t-q^{\ell(u)+1}t^{a(u)})(1-q^{\ell(u)+1}t^{a(u)+1})}{(1-q^{\ell(u)+1}t^{a(u)})^2}.\]
After simplification, this is exactly \cref{recursion}.
\end{proof}

\begin{prop}
There is a closed formula for $\fmuu$,
\begin{equation}\label{closedform} \fmuu = t^{\ell(w_m)} \cdot \frac{\prod\limits_{\square \in (\lambda^- \mid \gamma)} \left( 1-q^{\ell_{(\lambda^- \mid \gamma)} (\square)+1} t^{a_{(\lambda^- \mid \gamma)} (\square)} \right)}{\prod\limits_{\square \in (\mu \mid \gamma)} \left( 1-q^{\ell_{(\mu \mid \gamma)} (\square)+1} t^{a_{(\mu \mid \gamma)} (\square)} \right) }, \end{equation}
where $w_m$ is the minimal length element in $S_{[1,n-k]}$ such that $w_m(\lambda^-)=\mu$.
\end{prop}
\begin{proof}
We will prove this by showing that the given formula for $\fmuu$ satisfies the recurrence relations in \cref{Pexpansion}. The identity \cref{identity} is trivial. To show \cref{recursion}, suppose $s_i(\mu)<\mu$ for some $\mu \in S_{[1,n-k]}(\lambda)$. Let $w_m$ and $w'_m$ be the elements in $S_{[1,n-k]}$ of minimal length such that $w_m(\lambda^-)=s_i(\mu)$ and $w'_m(\lambda^-)=\mu$. We must now simplify
\begin{equation}\label{extrabox}\frac{f_{s_i(\mu),\lambdagamma}}{\fmuu} = t^{\ell(w_m)-\ell(w'_m)} \cdot \frac{\prod\limits_{\square \in (\mu \mid \gamma)} \left( 1-q^{\ell_{(\mu\mid \gamma)} (\square)+1} t^{a_{(\mu \mid \gamma)} (\square)} \right)}{\prod\limits_{\square \in (s_i(\mu) \mid \gamma)} \left( 1-q^{\ell_{(s_i(\mu) \mid \gamma)} (\square)+1} t^{a_{(s_i(\mu) \mid \gamma)} (\square)} \right) }. \end{equation}
Because $s_i(\mu)<\mu$, we have $\ell(w_m)-\ell(w'_m)=1$. Note that every box in the diagrams $\mugamma$ and $(s_i(\mu)\mid \gamma)$ has the same arms and legs in its corresponding diagram except one. The exception is the box $u$ whose coordinates in $s_i(\mu)$ are $(i,\mu_i+1)$, and the corresponding box $u'$ in $\mu$ with coordinates $(i+1,\mu_i+1)$. The boxes $u'$ and $u$ have the same number of legs. And $u$ has one more box in $s_i(\mu)$ in its arm than $u'$ in $\mu$, namely the box $(i+1,s_i(\mu)_{i+1})$. Thus we can write
\[ 1-q^{\ell_{(\mu\mid \gamma)} (u')+1} t^{a_{(\mu \mid \gamma)} (u')}=1-q^{\ell_{(s_i(\mu)\mid \gamma)} (u)+1} t^{a_{(s_i(\mu) \mid \gamma)-1} (u)}.\]
Now since every term from a box other than $u$ and $u'$ in \cref{extrabox} cancels, we are left with
\begin{align*} \frac{f_{s_i(\mu),\lambdagamma}}{\fmuu} &= t \cdot \frac{ 1-q^{\ell_{(\mu\mid \gamma)} (u')+1} t^{a_{(\mu \mid \gamma)} (u')}}{ 1-q^{\ell_{(s_i(\mu) \mid \gamma)} (u)+1} t^{a_{(s_i(\mu) \mid \gamma)} (u)} }\\
&= \frac{t-q^{\ell_{(s_i(\mu) \mid \gamma)} (u)+1} t^{a_{(s_i(\mu) \mid \gamma)} (u)}}{1-q^{\ell_{(s_i(\mu) \mid \gamma)} (u)+1} t^{a_{(s_i(\mu) \mid \gamma)} (u)}}.\end{align*}
Therefore the formula \cref{closedform} satisfies the recursive condition \cref{recursion}, and so the closed form we found is an equation for the coefficients $\fmuu$. \end{proof}

\section{Multiplication Rules}\label{ElementaryMultiplication}

Our next goal with the partially-symmetric Macdonald polynomials will be to find multiplication rules for $e_1[x_1,\dotsc, x_{n-k}] P_{\lambdagamma}$ in the basis $P_{\munu}$ of partially-symmetric Macdonald polynomials, where $e_1[x_1,\dotsc, x_{n-k}]$ is the elementary symmetric polynomial $x_1+ \dotsm + x_{n-k}$. After this is done, we will find a similar formula for $x_j \PP$ with $j>n-k$, so $x_j$ is a nonsymmetric variable, in a similar way. We refer to those coefficient formulas as Pieri rules.

	\subsection{Setup and notation}
	From this point forward, we make the assumption that $0<k<n$. Because $e_1[x_1, \dotsc, x_{n-k}]$ and $\PP{\lambdagamma}$ are both symmetric in the first $n-k$ variables, so is their product. Additionally, since the two polynomials are homogeneous with degree 1 and $|\lambdagamma|$ respectively, their product will be a homogeneous polynomial with degree $|\lambdagamma|+1$. There is therefore an expansion,
\begin{equation}\label{finalexpansion} (x_{1} + \dots + x_{n-k}) \PP{\lambdagamma} =  \sum_{\substack{|\mueta| = |\lambdagamma| + 1 \\ \mu_1 \geq \, \dotsm \, \geq \mu_{n-k}}}  C^{\lambdagamma}_{\mueta} \PP{\mueta}.  \end{equation}

In order to compute the coefficients, we will be using similar methods to those used by Baratta in \cite{Ba} to compute Pieri rules for the nonsymmetric Macdonald polynomials, in their case for $x_i E_\mu$. Because we will be making use of several of the formulas in that paper, we start by converting our notation into theirs.

The first difference in notation in our constructions and Baratta's is the indexing of the variables. In \cite{Ba}, the variables are $z_j$ and are written in reverse order, so our $x_i$ corresponds to $z_{n+1-i}$. As a result, the vectors we have considered to this point, written as $\lambdagamma$, will be reversed, so our `symmetrized' variables appear at the end. For the element of maximum length $w_0 \in S_n$, the vector $\lambdagamma$ in our notation then corresponds to $w_0\lambdagamma$ in Baratta's.

The Hecke algebra operators $H_j$ are defined with identical relations to ours, with $H_j$ replacing $T_j$ in the braid relations and quadratic relations \cref{quadratic}. The action of $H_j$ is given by the equation,
\[ H_j = \frac{(t-1)z_j}{z_j-z_{j+1}} + \frac{z_j-tz_{j+1}}{z_j-z_{j+1}} s_j. \label{DLbaratta}\]
Then using the fact that $x^\mu = z^{w_0(\mu)}$, we find,
\[ T_i f(x_1, \dotsc, x_n) = H_{n-i}f(z_n, \dotsc, z_1).\]

Under these conventions, the nonsymmetric Macdonald polynomials $\EI{\lambdagamma}$ correspond to Baratta's $E_{w_0 \lambdagamma}(z;q^{-1},t^{-1})$. The inverted powers of $q$ and $t$ are a formality, since we only use that form of the nonsymmetric Macdonald polynomials as an intermediate step. For clarity, we will also denote the polynomials in Baratta's conventions with the dagger, so $\EBlong{\gammalambdamin}$. Specifically, they correspond by the equation,
\[ \EIlong{\lambdagamma}=\left[\EBlong{w_0\lambdagamma}\right]_{z_i \mapsto x_{n+1-i}}. \]
Where it is unambiguous, the argument $(z;q^{-1},t^{-1})$ will be suppressed. Using the Hecke algebra isomorphism mapping $T_i$ to $H_{n-i}$, we get the definition of the partially-symmetric Macdonald polynomials,
\[\PBlong{\gammalambdamin} := \frac{1}{W^{n-k}_{\lambda}(t)} \sum_{w\in S_{[k+1, n]}} H_w \EBlong{\gammalambdamin}.\]

Because the $H_w$ act the same as the corresponding $T_w$ in our symmetrizers, we will also write this as,
\[ \PB{\gammalambdamin} = \eplusB  \EB{\gammalambdamin} \hspace*{.5cm} \text{ or } \hspace*{.5cm} \PB{\gammalambdamin} = \ep \EB{\gammalambdamin}. \]

In cases where we need the expansion of $\PB{\gammalambdamin}$ as a $\mathbb{Q}(q,t)$-linear combination of the nonsymmetric Macdonald polynomials, we will write,
\[ \PB{\gammalambdamin}=\sum_{\nu \sim \lambda} \fnu \EB{\gammanu},\]
where the $\fnu$ is equal to $f_{\mu, \eta}$ in \cref{Pexpansion} where $\nu$ is $\mu$ written in reverse order, and $\eta$ is $\gammalambdamin$ in reverse order.
    \subsection{Interpolation Polynomials}
	The following definitions can be found in \cite{Ba}. In order to derive Pieri-type formulas, we will be working with the interpolation nonsymmetric Macdonald polynomials. Begin with the eigenvalues associated with $E_\nu^\dagger$,
\begin{equation}\label{eigenvalueformula} \overline{\nu}_i = q^{\nu_i} t^{-l'_\nu(i)}, \hspace*{.5cm} 1\leq i \leq n,\end{equation}
\[ l'_\nu(i) = \# \{j<i \mid \nu_j \geq \nu_i\} + \# \{j>i \mid \nu_j > \nu_i\}.\]
Let $\overline{\nu} = (\overline{\nu}_1, \dotsc, \overline{\nu}_n)$ be the vector whose entries are the corresponding eigenvalues. These eigenvalues can be used to define the interpolation nonsymmetric Macdonald polynomials, which are the unique polynomials (up to normalization) $\Estarlong{\nu}$ of degree $| \nu |$ which satisfy evaluation criteria,
\begin{align}\label{vanishing} \begin{cases}\Estar{\nu}(\overline{\mu}) = 0  \hspace*{.5cm}\text{if } \mu \neq \nu \text{ and } |\mu| \leq |\nu|\\
\Estar{\nu}(\overline{\nu}) \neq 0
\end{cases}\end{align}

In fact, the latter evaluation has an explicit formula, given in Proposition 6 of \cite{Ba}:
 \begin{equation}\label{Estarformula} E^*_{\nu}(\overline{\nu}) = \left( \prod_{i=1}^n \overline{\nu}_i^{\nu_i} \right) \prod_{\square \in \nu} (1-q^{-(\ell_\nu(\square)+1)} t^{-a_\nu(\square)}).\end{equation}
 
 By Theorem 3.9 in \cite{KnopSymmetric}, the top homogeneous component of $\Estar{\nu}$ is the corresponding nonsymmetric Macdonald polynomial $\EB{\nu}$. In other words,
\[\Estarlong{\nu} = \EBlong{\nu} + \sum_{|\mu| < |\nu|} h_{\mu,\nu} \EBlong{\mu}.  \]

Thus define a $\mathbb{Q}(q,t)$-linear isomorphism $\Psi:\K[x_1, \dotsc, x_n] \to \K[x_1, \dotsc, x_k]$ which sends $\EBlong{\nu}$ to $\Estarlong{\nu}$. Note that the powers of $q$ and $t$ in the argument of the interpolation polynomials are again positive, so in our notation, $\EIlong{\nu}$ corresponds to $\Estarlong{w_0(\nu)}$.

The following lemmas will be useful when we look at the actions of the Hecke algebra on the interpolation polynomials.

\begin{lem}\label{eigenvaluepermutations}
Let $\mu$ be a composition with its corresponding vector of eigenvalues $\overline{\mu}$. Then $s_i \overline{\mu} = \overline{s_i\mu}$ if and only if $\mu_i \neq \mu_{i+1}$.
\end{lem}
\begin{proof}
This comes by comparing directly the $i$th and $(i+1)$st column eigenvalues.
 \end{proof}

\begin{lem}\label{eigentwo}
Let $\mu$ be a composition with $\mu_n>0$. Then
\[\overline{(\mu_n-1, \mu_1, \dotsc, \mu_{n-1})} = (q^{-1} \overline{\mu}_n\, , \, \overline{\mu}_1\, , \, \dotsc\, , \overline{\mu}_{n-1}).\]
\end{lem}
\begin{proof}
This is once again immediate from the eigenvalue definition.
\end{proof}

In Theorem 3.6 of \cite{KnopSymmetric}, the interpolation nonsymmetric Macdonald polynomials are proven to be simultaneous eigenfunctions of the commuting operators,
\[ \Xi_i := z_i^{-1} + z_i^{-1}H_i \dotsm H_{n-1} \Phi H_1 \dotsm H_{i-1},\]
where $\Phi$ is defined by
\[ \Phi = (z_n-t^{-n+1})\Delta,\]
\[ \Delta f(z_1, \dotsc, z_n) = f\left(\frac{z_n}{q}, z_1, \dotsc, z_{n-1}\right).\]
With these operators, we have
\[ \Xi_i \Estar{\nu}=(\overline{\nu}_i)^{-1} \Estar{\nu}.\]

Then define the following operators on the interpolation polynomials,
\[ Z_i := t^{-\binom{n}{2}} (z_i \Xi_i-1) \Xi_1 \dotsm \reallywidehat{\Xi_i} \dotsm \Xi_n,\]
where $\reallywidehat{\Xi_i}$ means $\Xi_i$ is not included in the list. 
The significance of $Z_i$ stems from a useful corollary in \cite{Ba},
\begin{equation}\label{isomorphism} z_i \EBlong{\nu} = q^{|\nu|} \Psi^{-1} Z_i \Estarlong{\nu}.\end{equation}
    \subsection{Interpolation polynomial expansion}
    We now begin with the computations that will lead to the desired Pieri formulas. First, following definitions from Baratta, define a modified version of the operators $Z_i$,
\[ \ZZ_i := H_i \dotsm H_{n-1} \Phi H_1 \dotsm H_{i-1}=z_i \Xi_i - 1. \]

With this definition, the definition of $Z_i$, and using commutativity of the $\Xi_j$, we have an identity,
\[ Z_i = t^{-\binom{n}{2}} \ZZ_i\, \Xi_i^{-1} \prod_{j=1}^n \Xi_j. \]

This form with the full product $\displaystyle{\prod_{j=1}^n \Xi_j}$ is useful since the product acts identically on all interpolation nonsymmetric Macdonald polynomials of the same degree.

\begin{lem}
Let $\mu$ be a composition. Then $\displaystyle{\prod_{j=1}^n \Xi_j \Estar{\mu} = q^{-|\mu|} t^{\binom{n}{2}} \Estar{\mu}}$.
\end{lem} \begin{proof} The eigenvalue of $\displaystyle{\prod_{j=1}^n \Xi_j}$ acting on $E^\dagger_{\mu}$ has a power of $q$ equal to $\displaystyle{-\sum_{i=1}^n \mu_i} = -|\mu|$, and a $t$ power equal to $\displaystyle{\sum_{j=1}^n l'_\nu(i) = \binom{n}{2}}$. Hence the eigenvalue associated with $\displaystyle{\prod_{j=1}^n \Xi_j}$ is $q^{-|\mu|} t^{\binom{n}{2}}$.
\qedhere
\end{proof}

Every nonsymmetric Macdonald polynomial in the expansion of $\PB{\gammalambdamin}$ has the same degree, so $\displaystyle{\prod_{j=1}^n \Xi_j}$ acts identically on each of those terms, and thus,

\[ \prod_{j=1}^n \Xi_j \PB{\gammalambdamin} = q^{-|\mu|}t^{\binom{n}{2}} \PB{\gammalambdamin}.\]

We also need the fact that $H_i$ acts the same on nonsymmetric Macdonald polynomials as their interpolation counterparts. This can be seen in \cite[page 6]{LRW}. We have the relation,
\begin{equation}\label{Tiactthesame} H_i \EB{\nu} = \Psi^{-1} H_i \Estar{\nu} .\end{equation}

Now we have all the required pieces to set up the desired computation, beginning by using \cref{Tiactthesame} and the $\Lq$-linearity of $\Psi$ to pass $\Psi^{-1}$ through the symmetrizer, and then \cref{isomorphism} to move it through the sum of monomials.
\begin{align*}
    (z_{k+1} + \dots + z_n) \PB{\gammalambdamin} &=  \left(\sum_{j=k+1}^n z_j \right) \eplusB \Psi^{-1} \Estar{\gammalambdamin}\\
     &=  \left(\sum_{j=k+1}^n z_j \right) \Psi^{-1} \eplusB  \Estar{\gammalambdamin}\\
    &= q^{|\gammalambdamin|} \Psi^{-1} \left( \sum_{j=k+1}^n Z_j \eplusB \Estar{\gammalambdamin}\right)\\
    &= q^{|\gammalambdamin|}\Psi^{-1}  \left( \sum_{j=k+1}^n \left( t^{-\binom{n}{2}} \ZZ_j\, \Xi_j^{-1} \prod_{\ell=1}^n \Xi_\ell \right) \ep \Estar{\gammalambdamin}\right)\\
    &= \Psi^{-1} \left( \sum_{j=k+1}^n \left( \ZZ_j\, \Xi_j^{-1} \right) \ep \Estar{\gammalambdamin}\right)\\
    &= \Psi^{-1} \left( \sum_{j=k+1}^n \left( H_j \dotsm H_{n-1} \Phi H_1 \dotsm H_{j-1} \, \Xi_j^{-1} \right) \ep \Estar{\gammalambdamin}\right)   \end{align*}
Combining the above with \cref{finalexpansion}, we have
\[ \sum_{\substack{|\etamu| = |\lambdagamma| + 1 \\ \mu_1 \geq \, \dotsm \, \geq \mu_{n-k}}}  C^{\gammalambdamin}_{\etamu} \PP{\etamu} = \Psi^{-1} \left( \sum_{j=k+1}^n \left( H_j \dotsm H_{n-1} \Phi H_1 \dotsm H_{j-1} \, \Xi_j^{-1} \right) \ep \Estar{\gammalambdamin}\right). \]

Then taking $\Psi$ of both sides and once again using \cref{Tiactthesame} on the left side,
\begin{equation}\label{bothsides}  \sum_{\substack{|\etamu| = |\lambdagamma| + 1 \\ \mu_1 \geq \, \dotsm \, \geq \mu_{n-k}}}  C^{\gammalambdamin}_{\etamu} e^+ \Estar{\etamu} = \sum_{j=k+1}^n \left( H_j \dotsm H_{n-1} \Phi H_1 \dotsm H_{j-1} \, \Xi_j^{-1} \right) \ep \Estar{\gammalambdamin}.\end{equation}
This is the equation we will use to compute $C_{\etamu}^{\gammalambdamin}$. The first simplification will be using the relation (see  \cite{KnopSymmetric}),
\[ H_i \Xi_{i+1}^{-1} = \Xi^{-1}_i \overline{H_i},\]
where 
\[\overline{H_i} := \frac{(t-1)z_{i+1}}{z_i-z_{i+1}} + \frac{z_i-tz_{i+1}}{z_i-z_{i+1}} s_i.\]
This $\overline{H_i}$ is almost the inverse of $H_i$, but instead their multiplication is given by,
\[H_i \overline{H_i} = t.\] Importantly, if $f=s_i f$, then $\overline{H_i} f = f$. In particular, this means for $i\geq k+1$,
\[ \overline{H_i} \PB{\gammalambdamin} = \PB{\gammalambdamin}.\]
Now continuing with our calculations,
\begin{align*}
   \sum_{\substack{|\etamu| = |\lambdagamma| + 1 \\ \mu_1 \geq \, \dotsm \, \geq \mu_{n-k}}}  C^{\gammalambdamin}_{\etamu} e^+ \Estar{\etamu}  &=  \sum_{j=k+1}^n  H_j \dotsm H_{n-1} \Phi H_1 \dotsm H_{j-1} \, \Xi_j^{-1} \ep \Estar{\gammanu}\\  
    &=  \sum_{j=k+1}^n  H_j \dotsm H_{n-1} \Phi H_1 \dotsm H_{k} \, \Xi_{k+1}^{-1}  \ep \Estar{\gammanu}\\
    &= \sum_{j=k+1}^n H_j \dotsm H_{n-1} \Phi H_1 \dotsm H_{k} \, \Xi_{k+1}^{-1} \sum_{\nu \sim \lambda} \fnu \Estar{\gammanu} \\
    &=  \sum_{\nu \sim \lambda} \fnu  \sum_{j=k+1}^n H_j \dotsm H_{n-1} \Phi H_1 \dotsm H_{k} \, \Xi_{k+1}^{-1} \Estar{\gammanu}
\end{align*}

Note that the choice of $\mu^-$ to be weakly increasing guarantees no two nonsymmetric Macdonald polynomials in the left sum symmetrize to the same partially-symmetric Macdonald polynomial, so the coefficients $C_{\etamu}^{\gammalambdamin}$ are well-defined.
    \subsection{Expansion Support}
    In this section, we will determine the weights which appear with nonzero coefficients in the Pieri expansion. This process serves the dual purpose of setting up the final computation for the coefficients. The following result identifies all possible coefficients which could appear, and we will show in \cref{support} that every such coefficient is nonzero.
\begin{prop}\label{expansion}
In the Pieri expansion,
\[ (z_{k+1} + \dots + z_n)\PB{\gammalambdamin} = \sum_{\substack{|\etamu| = |\gammalambdamin| + 1 \\[1mm] \mu^-_1 \leq \, \dotsm \, \leq \mu^-_{n-k} }} C^{\gammalambdamin}_{\etamu} \PB{\etamu}, \]
we have $C^{\gammalambdamin}_{\etamu}\neq 0$ only if, for some $I_1=\{t_1, \dots, t_r\} \subseteq [1,k]$ with $t_1 < \, \dotsm \, < t_r$, and $\nu \in S_{n-k}(\lambda^-)$, the nonsymmetric part $\eta$ is
\begin{align*}
    &\eta_i = \gamma_i \hspace*{.55cm}\text{ if } i\notin I_1\\
    &\eta_{t_j}=\gamma_{t_{j+1}} \text{ for } 1\leq i \leq r \text{ where } \gamma_{t_{r+1}} = \nu_1,
\end{align*}
and the symmetric part $\mu^-$ satisfies
\begin{align*}
    (\mu_1^-, \dotsc, \mu_{n-k}^- ) \in S_{n-k}(\nu_2, \dotsc, \nu_{n-k}, \gammanu_{t_1}+1).
\end{align*}
In the case that $I_1 = \emptyset$, we have $\eta = \gamma$ and $\gammanu_{t_1} := \nu_1$.
\end{prop}
\begin{proof}

We first must analyze the actions of $H_i$ and $\Phi$ on interpolation nonsymmetric Macdonald polynomials. To simplify the formula for the $H_i$, define some functions, following \cite{Ba},
\[ a(x,y) := \frac{(t-1)x}{x-y}, \hspace*{.5cm} b(x,y) := \frac{x-ty}{x-y}.\]
Then write $H_i = a(z_i,z_{i+1}) + b(z_i,z_{i+1}) s_i$. To continue, note that in the full expression,
\begin{equation}\label{rightside}\sum_{\nu \sim \lambda} \fnu\left( \sum_{j=k+1}^n \left( H_j \dotsm H_{n-1} \Phi H_1 \dotsm H_{k} \, \Xi_{k+1}^{-1} \right)  \Estar{\gammanu}\right),\end{equation}
because $\Xi^{-1}_{k+1}$ is acting on an interpolation polynomial, it acts as a constant eigenvalue which depends only on $\gammanu$, and does not impact the support. So we are evaluating the resulting polynomials in the expansion of $H_j \dots H_{n-1} \Phi H_1 \dots H_k \Estarlong{\gammanu}$. 
Expanding all of the $H_i$ operators and $\Phi$, we obtain
\begin{align*} &  [a(z_j,z_{j+1})+b(z_j,z_{j+1}) s_j] \dotsm [a(z_{n-1},z_n)+b(z_{n-1},z_n) s_{n-1}]\,    (z_n-t^{-n+1})\Delta \, \\& \hspace*{1cm} \cdot [a(z_1, z_2)+b(z_1,z_2) s_1] \dotsm [a(z_k,z_{k+1}) + b(z_k,z_{k+1})s_k] \, \Estarlong{\gammanu}.
\end{align*}

Since $s_i$ permutes the variables $z_i$ and $z_{i+1}$, we have that the action of the simple reflections is $s_i\Estarlong{\gammanu} = \Estar{\gammanu}(s_i(z);q,t)$. While this does not change the index $\gammanu$, the argument $s_i(z)$ means this is no longer an interpolation nonsymmetric Macdonald polynomial. Finding a term after expanding the above product corresponds to choosing either $a(z_i,z_{i+1})$ or $b(z_i,z_{i+1}) s_i$ from each bracketed pair. In this way, choose two sets $I_1=\{t_1< \, \dotsm \, < t_r\} \subseteq [1,k]$ and $I_2=\{v_1< \, \dotsm \, < v_s\} \subseteq [j,n-1]$ so that if $i\in I_1$, we choose $a(z_i,z_{i+1})$, and if $i\notin I_1$, we choose $b(z_i,z_{i+1})s_i$, and the same for $I_2$. We will write
\[ w_{I_1} := \shatproduct \]
to mean the full product $s_1 \dotsm s_k$ with the simple reflections with a hat removed. Let $p_{I_1}(z)$ be the coefficient obtained from the terms in
\[ \Delta [a(z_1, z_2)+b(z_1,z_2) s_1] \dotsm [a(z_k,z_{k+1}) + b(z_k,z_{k+1})s_k] = \sum_{I_1\subseteq [1,k]} p_{I_1}(z) \Delta w_{I_1}.\]
Similarly, define the permutation
\[w_{I_2} := s_j \dotsm \widehat{s_{v_1}} \dotsm \widehat{s_{v_s}} \dotsm s_{n-1} \]
Let $p_{2}(z)$  be the coefficient obtained from terms in
\[ [a(z_j,z_{j+1})+b(z_j,z_{j+1}) s_j] \dotsm [a(z_{n-1},z_n)+b(z_{n-1},z_n) s_{n-1}] (z_n-t^{-n+1})=\sum_{I_2\subseteq[j,n-1]}p_2(z)w_{I_2}.\]
Notice that the $(z_n-t^{-n+1})$ term is part of $p_2$. Additionally, we call this $p_2$ rather than $p_{I_2}$ because in later computations, we will find that the Pieri coefficients do not depend on $I_2$. Using these definitions, we can write an expansion of $\displaystyle{H_j \dotsm H_{n-1} \Delta H_1 \dotsm H_k  \Estarlong{\gammanu}}$ as
\begin{equation}\label{Expandedright}
      \sum\limits_{\substack{I_1 \subseteq [1,k] \\ I_2 \subseteq [j, n-1]}} p_{I_1}(z) \cdot p_2(z)  \cdot \left( s_j \dotsm \widehat{s}_{v_1} \dotsm \widehat{s}_{v_s} \dotsm s_{n-1}\Delta \shatproduct \right) \Estar{\gammanu}(z), \end{equation}
or more succinctly as
   \[ \sum\limits_{\substack{I_1 \subseteq [1,k] \\ I_2 \subseteq [j, n-1]}} p_{I_1}(z) \cdot p_2(z)  \cdot \Estar{\gammanu}(I_2(I_1(z))),
\]
where $I_1(z)$ and $I_2(z)$ are defined by
\begin{align} \label{Ione}&(I_1(z))_m := \begin{cases} z_m  &\text{ if } m\notin I_1 \text{ and } m \leq k \\
z_{t_{\ell-1}} \hspace*{.5cm} &\text{ if } m=t_\ell \text{ and } \ell \neq 1\\
q^{-1} z_n &\text{ if } m=t_1, \text{ or } r=0 \text{ and }m=k+1 \\
z_{t_r} &\text{ if } m=k+1 \text{ and } r\neq 0\\
z_{m-1} &\text{ if } m > k+1 \end{cases}\\[5mm]
\label{Itwo}& (I_2(z))_m := \begin{cases} z_m  &\text{ if } m\leq k \\
z_{m+1} \hspace*{.5cm} &\text{ if } m-1 \notin I_2, \text{ and } k+1 \leq m \leq n-1\\
z_{j} &\text{ if } m=v_1+1\\
z_{v_{i-1}+1} &\text{ if } m=v_i+1 \text{ and } i \geq 2\\
z_{v_s +1} &\text{ if } m=n, \text{ and if } s=0, \text{ then } v_s+1=j \end{cases}.\end{align}

The actions of $I_1$ and $I_2$ on the indices of the variables $z_i$ can be described as products of disjoint cycles, along with $\Delta$, by considering the whole permutation as follows in cyclic notation:
\[ s_j \dotsm \widehat{s}_{v_1} \dotsm \widehat{s}_{v_s} \dotsm s_{n-1}\Delta \shatproduct \]
\[ = (j \dots v_1) (v_1+1 \dots v_2) \dotsm (v_s+1 \dots n) \Delta (1 \dots t_1) (t_1+1 \dots t_2) \dotsm (t_r+1 \dots k+1).\]
This whole product can almost be combined into one long cycle,
\[ (v_s+1, v_{s-1}+1, \dots, v_1+1, j, j-1, \dots, k+2, k+1, t_r, t_{r-1}, \dots, t_1).\]
However, the cycling of $t_1$ to $v_s+1$ is special, since $I_2(I_1(z_{t_1})) = q^{-1} z_{v_s+1}$.

It is easier to work with the compositions $\gammanu$ than the lists of eigenvalues $\overline{(\eta \mid \mu)}$, so we would like to find an action $\dI$ for which 
\[\dI \gammanu=(\eta \mid \mu) \quad \text{ if and only if } \quad  I_2\left(I_1\left(\overline{(\eta \mid \mu)}\right) \right) = \overline{\gammanu}.\] The $I$ in $\dI$ should be thought of as $I = I_1 \cup I_2$. The following definition of $d_I$ is derived by repeated application of \cref{eigenvaluepermutations} and one application of \cref{eigentwo}. Note that the hypothesis in \cref{eigenvaluepermutations} that $\mu_i \neq \mu_{i+1}$ means the property above only holds if $I_1$ and $I_2$ do not include the permutation $s_i$ if it would permute two columns of the same height. For now we take this for granted, but we will see in \cref{comaximality} that we are only interested in $I_1$ which have this property, and the same will be shown explicitly for $I_2$ in the proof of \cref{Pieriformula} when uniqueness of $I_2$ is established.

Assuming $I_1 \neq \emptyset$, the following definition of $\dI$ satisfies this property, and the horizontal bar separates the nonsymmetric and symmetric parts of $\dI\gammanu$:

\[ (\dI\gammanu)_m := \left\{\begin{tabular}{l l} $ \left.\begin{array}{l l} \gammanu_m  &\text{ if } m\leq k \text{ and } m\notin I_1 \\
\gammanu_{k+1} \hspace*{.7cm} &\text{ if } m=t_r\\
\gammanu_{t_{i+1}} &\text{ if } m=t_i \text{ and } i\neq r\end{array} \hspace{2.85cm} \right\} \, m \leq k$ \\
\underline{\hspace{10cm}}\\[2mm] $\left.\begin{array}{l l} \gammanu_{m+1} &\text{ if } k+1 \leq m < j\\
\gammanu_m &\text{ if } m>j \text{ and } m-1\notin I_2\\
\gammanu_{v_1+1} &\text{ if } m=j \text{ and } s\neq 0\\
\gammanu_{v_{i+1}+1} &\text{ if } m=v_i+1 \text{ and } 1 \leq i < s\\
(\gammanu_{t_1}) + 1&\text{ if } m=v_s+1 , \text{ or }  s=0 \text{ and } m = j \hspace*{.3cm} \end{array} \right\} \, m \geq k+1 $ \end{tabular}  \right.\\[5mm]\]

If $I_1 = \emptyset$, the entry $(\gammanu_{t_1})+1$ in the last row should be replaced by $(\gammanu_{k+1})+1$. The constructions so far allow us to compute the partially-symmetric Macdonald polynomials which appear in the Pieri expansion, by considering evaluations of \cref{rightside}. First, we rewrite that equation using the operator expansions we have found:
\begin{equation}\label{newrightside}
    \sum_{\nu \sim \lambda} \fnu\left( \sum_{j=k+1}^n    \sum\limits_{\substack{I_1 \subseteq [1,k] \\ I_2 \subseteq [j, n-1]}} p_{I_1}(z) \cdot p_2(z)  \cdot \Estar{\gammanu}(I_2(I_1(z))) \right).
\end{equation} If we evaluate the expression at $z=\overline{\etamu}$, by the vanishing properties of the interpolation polynomials, we only get a nonzero term if $I_2\left( I_1 \left( \overline{\etamu} \right) \right) = \overline{\gammanu}$. So equivalently, this vanishes unless $\etamu = d_I\gammanu$ for some sets $I_1$ and $I_2$. From the definition of $d_I$, this happens when $\eta$ satisfies
\begin{align}
    &\eta_i = \gamma_i \hspace*{.55cm}\text{ if } i\notin I_1 \label{part1}\\
    &\eta_{t_j}=\gamma_{t_{j+1}} \text{ for } 1\leq i \leq r \text{ where } \gamma_{t_{r+1}} = \nu_1\label{part2}
\end{align}
and $\mu^-$ has
\begin{align}
    (\mu_1^-, \dotsc, \mu_{n-k}^- ) \in S_{n-k} (\nu_2, \dotsc, \nu_{n-k}, \gammanu_{t_1}+1).\label{part3}
\end{align}
Again we have that if $I_1 = \emptyset$, then $t_1$ should be considered as $k+1$. \\

Then take the formula found earlier,
\[    \sum_{\substack{|\etamu| = |\lambdagamma| + 1 \\ \mu_1 \geq \, \dotsm \, \geq \mu_{n-k}}}  C^{\gammalambdamin}_{\etamu} e^+ \Estar{\etamu}  =  \sum_{\nu \sim \lambda} \fnu  \sum_{j=k+1}^n  H_j \dotsm H_{n-1} \Phi H_1 \dotsm H_{k} \, \Xi_{k+1}^{-1}  \Estar{\gammanu}.\]

Due to \cref{newrightside}, the right side vanishes when evaluated at $\overline{\etamu}$ unless $\etamu$ meets the above criteria. We know from \cref{Tiactthesame} that $e^+ \Estar{\etamu}$ expands in the same way as $ \PB{\etamu}$. So the left side must also vanish for $\etamu$ which don't meet those conditions. Because both sides are $\K$-linear combinations of interpolation polynomials,  using \cref{vanishing}, the only possible $\Estar{\etamu}$ which can appear with nonzero coefficients on the left are indexed by those same nonvanishing $\etamu$.

The order of $\mu^-$ is unimportant, since any permutation of $\mu^-$ contributes to the same $\PB{\etamu}$. Altogether, we conclude that the $\PB{\etamu}$ which appear in the expansion of $(z_{k+1} + \dotsm + z_{n})\PB{\gammanu}$ are those which satisfy \cref{part1}, \cref{part2}, and \cref{part3}. \qedhere \end{proof}

 With this all constructed, we make the following notes:
\begin{itemize}
    \item The action of $\dI$ increases the height of exactly one column, which is found in $\mu^-$.
    \item Because $I_2$ permutes only the entries with positions in $[k+1,n]$, which will all symmetrize to the same partially-symmetric Macdonald polynomial, the action of $I_2$ does not change which terms appear as basis elements in the Pieri expansion.
\end{itemize}

We can now restrict the possible subsets $I_1$ which contribute to the Pieri-type coefficient formula.

\begin{lem}\label{comaximality}Let $(\eta \mid \mu)$ be a composition such that $\overline{\gammanu} = I_2\left( I_1 \left(\overline{(\eta \mid \mu)} \right)\right)$ for some composition $\gammanu$ and sets $I_1\subseteq [1,k]$ and $I_2 \subseteq [j,n-1]$. Then the $I_1$ and $I_2$ which satisfy that equation are unique.
\end{lem}
\begin{proof}
This is a consequence of \cref{eigenvaluepermutations}, and is proved in \cite[proof of Lemma 4.3] {KnopSymmetric}.
\end{proof}

We will work more with $I_2$ in the proof of \cref{Pieriformula}, but for now we define $I_1$ which is the unique set in the lemma above.

\begin{dfn}\label{maximality}
We call $I_1$ maximal with respect to $(\gamma \mid \nu)$ and $I_2$ if there is some $(\eta \mid \mu)$ for which $\overline{(\gamma \mid \nu)} = I_2 \left( I_1 \left( \overline{(\eta \mid \mu)} \right)\right)$.
\end{dfn}
In explicit terms, $I_1$ is maximal with respect to $\gammanu$ and if it satisfies:
\begin{enumerate}
    \item $\gamma_i \neq \gamma_{t_1}$ for all $i< t_1$, and if $I_1=\emptyset$, consider $\gammanu_{t_1} := \nu_1$ \item $\gamma_i \neq \gamma_{t_{u+1}}$ for all $t_u+1 \leq i \leq t_{u+1}-1$
\end{enumerate}

Note that the set $I_1$ which satisfies the above crieria is called \textit{comaximal} with respect to $\gammanu$ in \cite{Ba}. We also omit the reference to $I_2$ in maximality where it is not needed. We will show after computing the coefficients $\pIone$ and $\pItwo$ that for a particular choice of $I_2$ and maximal $I_1$, the coefficient $C^{\gammalambdamin}_{\etamu}$ is nonzero. Finally, we give notation for the support set.

\begin{dfn}
Define $\MM$ to be the set of compositions which appear in \cref{expansion},
\begin{align*} \MM := \{\, \etamu \hspace*{.2cm} \Big| \hspace*{.2cm} \dI &\gammanu = \etamu \text{ for some } \nu \in S_{n-k}(\lambda^-), \text{ some maximal set } I_1 \\ & \text{ and some } \mu^- \in S_{n-k}(\nu_2, \dotsc, \nu_{n-k}, \gammanu_{t_1}+1)\} .\end{align*}
\end{dfn}

    \subsection{Pieri coefficient computation}\label{Piericomputation}
    Let us return to the equation which we will use to determine the Pieri coefficients:
\begin{align}\label{evaluatethis}\sum_{|\etamu| = |\gammalambdamin| + 1} &C^{\gammalambdamin}_{\etamu} \ep \Estar{\etamu}(z)\nonumber \\ = &\sum_{\nu \sim \lambda} \fnu \left( \sum_{j=k+1}^n \left( H_j \dotsm H_{n-1} \Phi H_1 \dotsm H_{k} \, \Xi_{k+1}^{-1} \right)  \Estar{\gammanu}(z)\right)\end{align}

Our approach will be to evaluate both sides at a particular value for $z$, and solve the equation for $C_{\etamu}$. First, use the expansion of $\ep \Estar{\etamu}(z)$ into nonsymmetric interpolation polynomials,
\begin{align}\label{evaluationpoint}\sum_{\etamu}  C^{\gammalambdamin}_{\etamu} &\sum_{\mu' \sim \mu} \gmuprime \Estar{\etamuprime}(z)\nonumber\\
&= \sum_{\nu \sim \lambda} \fnu \left( \sum_{j=k+1}^n \left( H_j \dotsm H_{n-1} \Phi H_1 \dotsm H_{k} \, \Xi_{k+1}^{-1} \right)  \Estar{\gammanu}(z)\right)\end{align}

By the vanishing property of interpolation polynomials, evaluating at a vector of eigenvalues $\overline{\etamuprime}$ will cause all but one term to vanish on the left side. We will choose a particular value of $\mu'$ that we call $\mutilde$ which simplifies our remaining computations. Starting with a vector $\gammalambdamin$, we begin by choosing some $\etamu$ in $\MM$, which determines a unique $I_1$ that is maximal with respect to $\gammalambdamin$. Notice that going from $\gammalambdamin$ to $\etamu$, exactly one column height from $\lambda^-$ is removed, and one column has its height increased in the creation of $\mu^-$. This can be read directly from $\lambda^-$ and $\mu^-$ unless $\lambda^-=\mu^-$. If this is the case, both the removed $\lambda^-$ column and the increased $\mu^-$ column have a height $\gamma_i$ where $i$ is the largest value for which $\gamma_i \neq \eta_i$.

Then for the choice of which $\mutilde$ to use in the evaluation, we require all columns with the same height as the increased column in $\mu^-$ to be at the beginning of the vector. The remaining entries are chosen to be weakly decreasing. As a consequence of the choice of $\mutilde$, we will see later that there is only one contributing $\nu$ in the sum indexed by $\nu$. We call this special composition $\lambdatilde$, which is given by
\[(\lambdatilde_1,  \dotsc, \lambdatilde_{n-k}) = (\eta_{t_r}, \mutilde_2, \dotsc, \mutilde_{n-k}) \]
if $I\neq \emptyset$, and if $I = \emptyset$,
\[ (\lambdatilde_1, \dotsc, \lambdatilde_{n-k}) = (\mutilde_1-1, \mutilde_2, \dotsc, \mutilde_{n-k}).\]
It is useful to define a value,
 \[\mIone := |\{ 1\leq i \leq n-k \mid \, \mutilde_i = \mutilde_1  \}|-1. \]
This is the same as the number of entries in $(\lambdatilde_2,\dotsc, \lambdatilde_{n-k})$ which are the same height as the newly increased column. We now have enough to write the coefficients in the Pieri expansion.

\begin{thm}\label{Pieriformula}
The product $e_1[z_{k+1},\dotsc, z_n] \PB{\gammalambdamin}$ can be written as a sum of partially-symmetric Macdonald polynomials with the formula,
\[ (z_{k+1} + \dotsm + z_n)\PB{\gammalambdamin} = \sum_{\etamu \in \MM} C^{\gammalambdamin}_{\etamu} \PB{\etamu},\]
where
\[C^{\gammalambdamin}_{\etamu} = \frac{f_{\lambdatilde,\lambdagamma}}{\fmutilde} \cdot \overline{\gammalambdatilde}_{k+1} \cdot p_{I_1} \cdot p_{2}\cdot \frac{\Estar{\gammalambdatilde}\left(\overline{\gammalambdatilde} \right)}{\Estar{\etamutilde} \left( \overline{\etamutilde}\right)}, \]
 the symmetric parts $\mutilde$ and $\lambdatilde$ are those described above, the set $I_1=\{t_1 < \, \dotsm \, < t_r\}$ is maximal with respect to $\gammalambdatilde$, $m := \mIone$, and $p_{I_1}$ and $p_{2}$ are given by
\begin{align*}
    p_{I_1}     &= \frac{(t-1)q^{-1} \overline{\mutilde}_{m+1}}{q^{-1} \overline{\mutilde}_{m+1}-\overline{\eta}_{t_1}} \prod_{u=1}^{r-1} \left(\frac{(t-1) \overline{\eta}_{t_u}}{ \overline{\eta}_{t_u}-\overline{\eta}_{t_{u+1}}}\right)  \prod_{j=1}^{t_1-1} \left( \frac{q^{-1} \overline{\mutilde}_{m+1} - t \overline{\eta}_j}{q^{-1} \overline{\mutilde}_{m+1} -  \overline{\eta}_j} \right) \prod_{u=1}^{r} \prod_{j=t_u+1}^{t_{u+1}-1}  \left(\frac{\overline{\eta}_{t_u} - t \overline{\eta}_j}{ \overline{\eta}_{t_u} -  \overline{\eta}_j} \right)\\[3mm]
    p_{2} &= \frac{1-t^{m+1}}{1-t}\cdot (\overline{\mutilde}_{m+1} - t^{-n+1}) \cdot \prod_{j=m+2}^{n-k} \frac{\overline{\mutilde}_{m+1} - t \cdot \overline{\mutilde}_j}{\overline{\mutilde}_{m+1} - \overline{\mutilde}_j}.
\end{align*}
\end{thm}
\begin{proof}The remainder of the derivation proceeds by evaluating both sides of \cref{evaluatethis} at $\overline{\etamutilde}$. Because the sums in the left side were constructed to have no overlapping terms, evaluation of the left side of the equation gives
\[ \sum_{\etamu} C^{\gammalambdamin}_{\etamu} \sum_{\mu' \sim \mu} f_{\mu',\etamu} \cdot \Estar{\etamuprime}\left(\overline{(\eta \mid \widetilde{\mu})}\right)= C^{\gammalambdamin}_{\etamu} \, \gmutilde \,\Estar{\etamutilde} \left( \overline{\etamutilde}\right). \]

Before evaluating the right side of \cref{evaluationpoint}, using the fact that $\Xi^{-1}_{k+1}$ acts as an eigenvalue, we make the simplification
\begin{align*}
    H_j \dotsm H_{n-1} \Phi H_1 \dotsm H_{k} \, \Xi_{k+1}^{-1}   \Estar{\gammanu}(z) 
    &=\overline{\gammanu}_{k+1} \left( H_j \dotsm H_{n-1} \Phi H_1 \dotsm H_{k} \right)  \Estar{\gammanu}(z).
\end{align*}

It remains to simplify the evaluation,
\begin{equation}\label{unsimplifiedsum} \sum_{\nu \sim \lambda} \fnu \cdot \overline{\gammanu}_{k+1} \left( \sum_{j=k+1}^n   \left( H_j \dotsm H_{n-1} \Phi H_1 \dotsm H_{k} \Estar{\gammanu} \right)\left(\overline{\etamutilde}\right)\right).\end{equation}

Expanding the inner sum indexed by $j$ with \cref{Expandedright}, we must then simplify
\[  \sum_{j=k+1}^n \,   \sum_{I_2\subseteq [j, n-1]} \sum_{I_1 \subseteq [1,k]}\pIoneevaluated \cdot \pItwoevaluated \cdot \Estar{\gammanu}\left(I_2\left(I_1\left(\overline{\etamutilde}\right)\right)\right).\]

The goal at this point is to find which compositions $\gammanu$ and sets $I_1$ and $I_2$ can be used to satisfy
\[\overline{\gammanu} = I_2\left(I_1\left(\overline{\etamutilde}\right)\right). \]
In other words, we must find which compositions and sets satisfy
\[\dI\gammanu = \etamutilde. \]
The nonsymmetric part $\eta$ has already been found in the support computations. Recall that there is only one $I_1$ such that $p_{I_1}\left(\overline{\etamutilde}\right) \neq 0$ by \cref{comaximality}. That leaves us with the task of finding possible values for $\nu$ and $I_2$.
 Note that $\mutilde$ has entries equal to $\mutilde_1$ in its first $\mIone+1$ positions. From this point on, we will use $m$ and $\mIone$ interchangeably. We will now make several observations which reduce the sums to only one composition $\nu=\lambdatilde$ and set $I_2=[j, \dotsc,k+m]$. 
\begin{itemize}
    \item $\nu_1$ must be $\eta_{t_r}$, which was found when constructing $\etamu$, if $I \neq \emptyset$. And if $I=\emptyset$, then $\nu_1$ is the increased column.
    \item Since $\mutilde_1 = \gammanu_{t_1}+1$, by the definition of $\dI$, we get $(\dI \gammanu)_{v_s+1} = \mutilde_1$. So since $\dI \gammanu = \etamutilde$ only has $\mutilde_1$ in positions $k+1, \dots, k+m+ 1$, we can conclude $v_s+1 \leq k+ m + 1$. And because $v_s$ is the highest value in $I_2$, that means any value in $\{k+m+1, \dotsc, n-1\}$ cannot appear in $I_2$. So $I_2 \subseteq [j, \dotsc, k+m]$.
    \item Again using the definition of $d_I$, we have that $(\dI \gammanu)_i = \etamutilde_i = \gammanu_i$ for $i\geq k+m+2$, since this means $i-1 \notin I_2$. Equivalently, $\nu_i = \mutilde_i$ for $i\geq m+2$. These are exactly the entries of $\mutilde$ which are not equal to $\mutilde_1$. The only entries not accounted for in $\nu$ are the ones equal to $\mutilde_1$, so $\nu_i = \mutilde_1$ for all $2\leq i \leq m + 1$.  Altogether,
    \[ (\nu_2, \dotsc, \nu_{n-k}) = (\mutilde_2, \dotsc, \mutilde_{n-k}).\]
    \item If $j>k+m + 1$, then from the definition of $\dI$, on one hand, \[(\dI \gammanu)_{k+m+1}=\gammanu_{k+m+2} \neq \mutilde_1.\]
    On the other hand, we have
    \[ (\dI \gammanu)_{k+m+1} = \etamutilde_{k+m + 1} = \mutilde_{m+1} = \mutilde_1.\]
    So $j\leq k+m+1$.
    \item Suppose $\ell \notin I_2$ and $j \leq \ell \leq k+m$, and choose the smallest such $\ell$. That would mean we take the part of $H_\ell$ that acts as $b(z_\ell, z_{\ell+1}) s_\ell$. Since $\etamutilde_\ell = \etamutilde_{\ell+1}=\mutilde_1$, by the definition of the eigenvalues, we have $\overline{\etamutilde}_\ell = t\cdot \overline{\etamutilde}_{\ell+1}$. Therefore the coefficient $p_2$ will include the term 
    \[b\left(\overline{\etamutilde}_{\ell}, \overline{\etamutilde}_{\ell+1} \right) =\frac{\overline{\etamutilde}_\ell- t\cdot \overline{\etamutilde}_{\ell+1}}{\overline{\etamutilde}_\ell- \overline{\etamutilde}_{\ell+1}} = 0.\] Thus for the coefficient to be nonzero, $H_\ell$ must act as $a(z_\ell,z_{\ell+1})$, which means $I_2 = [j, \dotsc, k+m]$. Note that if $j=k+m+1$, then $I_2 = \emptyset$.
\end{itemize}

The above observations allow us to conclude that for it to be possible that $\dI \gammanu = \etamutilde$, we must have $\nu = \lambdatilde$ and $I_2 = [j, \dotsc, k+m]$. \\

$H_\ell$ must act as $a(z_\ell,z_{\ell+1})$ for each $j \leq \ell \leq k+m$, and the columns $\ell$ and $\ell+1$ have the same height, so we have $\overline{z}_\ell=t\cdot \overline{z}_{\ell+1}$. After evaluation, the actions simplify to
\[ H_\ell = a(\overline{z}_\ell, \overline{z}_{\ell+1}) = \frac{(t-1) \overline{z}_\ell}{\overline{z}_\ell-\overline{z}_{\ell+1}} = t.\]
Taking all possible values of $j$ together,
\[ \sum_{j={k+1}}^{k+m+1} H_j \dotsm H_{k+m}\quad  \text{ acts by }\quad \sum_{i=0}^{m} t^i.\]

We will consider the above sum to be part of $\pItwo$. With all the simplifications up to this point, as well as rewriting the coefficient $f_{\nu, \gammalambdamin}$ with the corresponding specific value in our notation $f_{\lambdatilde, \lambdagamma}$, we can now rewrite \cref{unsimplifiedsum} as
\[ \flambdatilde \cdot \overline{\gammalambdatilde}_{k+1} \cdot    \pIoneevaluated \cdot \pItwoevaluated \cdot \Estar{\gammanu}\left(I_2\left(I_1\left(\overline{\etamutilde}\right)\right)\right).\]

The argument $\left(I_2\left(I_1\left(\overline{\etamutilde}\right)\right) \right)$ can just be written as $\overline{\gammalambdatilde}$ at this point. Now all that remains is computing explicit formulas for $\pIone$ and $\pItwo$. We have already done most of the work to compute $\pItwo$, since we know how each of $H_j,\dots, H_{n-1}$ act. So we can make some simplifications, \begin{align*}\sum_{j=k+1}^n &H_j \dotsm H_{n-1}(z_n-t^{-n+1})\\
  &= \left( \sum_{i=0}^{m} t^i \right) H_{k+m+1} \dotsm H_{n-1} (z_n-t^{-n+1}) \\
&= \left( \sum_{i=0}^{m} t^i \right)[b(z_{k+m+1},z_{k+m+2}) s_{k+m+1}] \dotsm [b(z_{n-1},z_n) s_{n-1}](z_n-t^{-n+1}) \\
&= \left( \sum_{i=0}^{m} t^i \right)\left(b(z_{k+m+1},z_{k+m+2}) \dotsm b(z_{k+m+1}, z_{n})\right) (z_{k+m+1}-t^{-n+1}) s_{k+m+1} \dotsm s_{n-1}\end{align*}

This means that the formula for $\pItwo$ is:
\begin{align*} \pItwo &= \left( \sum_{i=0}^{m} t^i \right) (\overline{z}_{m+k+1}-t^{-n+1}) (b(\overline{z}_{k+m+1},\overline{z}_{k+m+2}) \dotsm b(\overline{z}_{k+m+1}, \overline{z}_{n}))\\
\pItwoevaluated&=\frac{1-t^{m+1}}{1-t}\cdot (\overline{\mutilde}_{m+1} - t^{-n+1}) \cdot \prod_{j=m+2}^{n-k} \frac{\overline{\mutilde}_{m+1} - t \cdot \overline{\mutilde}_j}{\overline{\mutilde}_{m+1} - \overline{\mutilde}_j}\end{align*}

We shorten $\pItwoevaluated$ to $p_{2}$. And finally, for $\pIone$, the coefficient is almost identical to the one computed in \cite{Ba}. Let $I_1 = \{ t_1, \dotsc, t_r \}$ and $t_{r+1}:= k+1$, and note that $v_s+1 = k+m+1$. Then the coefficient $\pIone$ is:
\begin{align*}
     & a\left(q^{-1} \overline{z}_{v_s+1}, \overline{z}_{t_1} \right) \prod_{u=1}^{r-1} a \left(\overline{z}_{t_u},\overline{z}_{t_{u+1}} \right) \cdot \prod_{j=1}^{t_1-1} b \left(q^{-1}\overline{z}_{v_s+1}, \overline{z}_j \right) \prod_{u=1}^{r} \prod_{j=t_u+1}^{t_{u+1}-1} b \left( \overline{z}_{t_u},\overline{z}_j\right)\\
    &=\left(\frac{(t-1)q^{-1} \overline{z}_{v_s+1}}{q^{-1} \overline{z}_{v_s+1}-\overline{z}_{t_1}}\right) \cdot \prod_{u=1}^{r-1} \left(\frac{(t-1) \overline{z}_{t_u}}{ \overline{z}_{t_u}-\overline{z}_{t_{u+1}}} \right) \cdot \prod_{j=1}^{t_1-1} \left( \frac{q^{-1} \overline{z}_{v_s+1} - t \overline{z}_j}{q^{-1} \overline{z}_{v_s+1} -  \overline{z}_j}\right) \cdot  \prod_{u=1}^{r} \prod_{j=t_u+1}^{t_{u+1}-1} \left( \frac{\overline{z}_{t_u} - t \overline{z}_j}{ \overline{z}_{t_u} -  \overline{z}_j} \right)
\end{align*}

After evaluation at $\overline{\etamutilde}$, we get the formula for $\pIoneevaluated$, which we shorten to $p_{I_1}$:
\[p_{I_1} =  \left( \frac{(t-1)q^{-1} \overline{\mutilde}_{m+1}}{q^{-1} \overline{\mutilde}_{m+1}-\overline{\eta}_{t_1}} \right) \prod_{u=1}^{r-1} \left(\frac{(t-1) \overline{\eta}_{t_u}}{ \overline{\eta}_{t_u}-\overline{\eta}_{t_{u+1}}}\right)  \prod_{j=1}^{t_1-1} \left( \frac{q^{-1} \overline{\mutilde}_{m+1} - t \overline{\eta}_j}{q^{-1} \overline{\mutilde}_{m+1} -  \overline{\eta}_j} \right) \prod_{u=1}^{r} \prod_{j=t_u+1}^{t_{u+1}-1}  \left(\frac{\overline{\eta}_{t_u} - t \overline{\eta}_j}{ \overline{\eta}_{t_u} -  \overline{\eta}_j} \right)\]

With all the terms now computed, we solve the following equation to find the Pieri coefficients:
\[
    C^{\gammalambdamin}_{\etamu}\, \fmutilde \,\Estar{\etamutilde} \left( \overline{\etamutilde}\right) = \flambdatilde \cdot \overline{\gammalambdatilde}_{k+1} \cdot p_{I_1} \cdot p_{2} \cdot \Estar{\gammalambdatilde}\left(\overline{\gammalambdatilde} \right)\]
   \[ C^{\gammalambdamin}_{\etamu} = \frac{\flambdatilde}{\fmutilde} \cdot \overline{\gammalambdatilde}_{k+1} \cdot p_{I_1} \cdot p_{2} \cdot \frac{\Estar{\gammalambdatilde}\left(\overline{\gammalambdatilde} \right)}{\Estar{\etamutilde} \left( \overline{\etamutilde}\right)}\qedhere
\]
\end{proof}
\vspace*{.3cm}
    \subsection{Further coefficient results}
    Now that we have a formula for the Pieri coefficients, we will show first that the coefficient for any partially-symmetric Macdonald polynomial that appears in the support. Then we will convert the formulas back to our original conventions, and give the corresponding formula for the expansion of $e_1[x_1,\dots, x_{n-k}] \JJ$.

\begin{lem}\label{nonzero}
If $I_1$ is maximal with respect to $\gammalambdatilde$, then $\pIoneevaluated \neq 0$.
\end{lem}
\begin{proof}
    See appendices.
\end{proof}

\begin{cor}\label[corollary]{support}
For every $\etamu \in \MM$ , the coefficient $C^{\gammalambdamin}_{\etamu}$ is nonzero.
\end{cor}
\begin{proof}
In the proof of uniqueness of $I_2$ and $\nu$, we found that $\pItwoevaluated \neq 0$. The evaluation $\Estar{\gammalambdatilde} \left(\overline{\gammalambdatilde} \right)$ is nonzero by the vanishing properties of the interpolation polynomials. And $\flambdatilde \neq 0$ by \cref{Pexpansion}. So choosing the maximal set $I_1$ such that $\dI{\gammalambdatilde} = \etamutilde$ guarantees the coefficient $C^\gammalambdamin_{\etamu}$ is nonzero.
\end{proof}

We now convert \cref{Pieriformula} back to our original notation.

Given a weight $\lambdagamma$, begin by choosing some entry of $\lambda$ to be $\lambdatilde_{n-k}$ and set $I_1 = \{t_1, \dotsc, t_r\} \subseteq [1,k]$ which satisfies, 
\begin{itemize}
    \item $\eta_j \neq \eta_{t_u}$ for any $u \in \{1,\dotsc, r\}$ and $j \in \{t_{u-1}+1,\dotsc, t_{u}-1\}$, and
    \item $\eta_j \neq \mutilde_{n-k}-1$ for any $j \in \{t_r+1, \dotsc, k\}$.
\end{itemize}
The increased column height has height $\gamma_{t_r}+1$. Next we choose $\mutilde$ to have all the columns $\mutilde_{h} ,\dotsc, \mutilde_{n-k}$ of height $\gamma_{t_r}+1$, and $\mutilde_1 < \, \dotsm \, < \mutilde_{h-1}$. Also, we choose $(\lambdatilde_1, \dotsc, \lambdatilde_{n-k-1}) = (\mutilde_1, \dotsc, \mutilde_{n-k-1})$. Finally, we define $\eta$ as before,
 \begin{align}
     &\eta_i = \gamma_i \hspace*{.55cm}\text{ if } i\notin I_1\nonumber \\
     &\eta_{t_j}=\gamma_{t_{j+1}} \text{ for } 1\leq i \leq r \text{ where } \gamma_{t_{r+1}} = \lambdatilde_1.\label{eta}
 \end{align}

\begin{cor}\label{ourconventions}
 The product $e_1[x_1,\dotsc, x_{n-k}] \PP$ can be written as a sum of partially-symmetric Macdonald polynomials with the formula,
\[ (x_1 + \dotsm + x_{n-k})\PP{\lambdagamma} = \sum_{\mueta \in \MMM} C^{\lambdagamma}_{\mueta} \PP{\mueta},\]
where
\[C^{\lambdagamma}_{\mueta} = \frac{f_{\lambdatilde,\lambdagamma}}{f_{\mutilde,\mueta}} \cdot \overline{\lambdatildegamma}_{n-k} \cdot p_{I_1} \cdot p_{2}\cdot \frac{\Estar{\lambdatildegamma}\overline{\lambdatildegamma}}{\Estar{\mutildeeta}\overline{\mutildeeta}}, \]
$I_1$ is maximal with respect to $\lambdatildegamma$, the symmetric parts $\mutilde$ and $\lambdatilde$ are those described above, $p_{I_1}$ and $p_{2}$ are given by
\begin{align*}
p_{I_1} &= \\
  &\left(\frac{(t-1)q^{-1} \overline{\mutilde}_{\hIone}}{q^{-1} \overline{\mutilde}_{\hIone}-\overline{\eta}_{t_r}}\right) \cdot \prod_{u=1}^{r-1} \left(\frac{(t-1) \overline{\eta}_{t_{u+1}}}{ \overline{\eta}_{t_{u+1}}-\overline{\eta}_{t_{u}}} \right) \cdot \prod_{j=t_r+1}^{k} \left( \frac{q^{-1} \overline{\mutilde}_{\hIone} - t \overline{\eta}_j}{q^{-1} \overline{\mutilde}_{\hIone} -  \overline{\eta}_j}\right) \cdot  \prod_{u=1}^{r} \prod_{j=t_{u-1}+1}^{t_{u}-1} \left( \frac{\overline{\eta}_{t_u} - t \overline{\eta}_j}{ \overline{\eta}_{t_{u}} -  \overline{\eta}_j} \right)\\
p_{2}&=\left( \sum_{i=0}^{m} t^i \right)\cdot (\overline{\mutilde}_{\hIone} - t^{-n+1}) \cdot \prod_{j=1}^{\hIone-1} \frac{\overline{\mutilde}_{\hIone} - t \cdot \overline{\mutilde}_j}{\overline{\mutilde}_{\hIone} - \overline{\mutilde}_j}.
\end{align*}
\end{cor}

We can also renormalize the Pieri formula to get a corresponding equation for the integral form $\JJ$ polynomials.

\begin{cor}\label{Jpieri}
 There is a Pieri expansion of the integral form of partially-symmetric Macdonald polynomials,
\[ (x_1 + \dots + x_{n-k})\JJ = \sum_{\mueta \in \MMM} \left(  \frac{f_{\lambdatilde,\lambdagamma}}{f_{\mutilde,\mueta}} \cdot \overline{\lambdatildegamma}_{n-k} \cdot p_{I_1} \cdot p_{2}\cdot \frac{\Estar{\lambdatildegamma}\overline{\lambdatildegamma}}{\Estar{\mutildeeta}\overline{\mutildeeta}} \cdot \frac{ j_{\lambdagamma}}{j_{\mueta}} \right) \mathcal{J}_{\mueta}. \]
\end{cor}

\subsection{Multiplication by a nonsymmetric variable}

Much like in the previous section, we would like to find a rule for the expansion of the product $x_j P_\lambdagamma$, where $j>n-k$, in the partially-symmetric Macdonald basis. Fortunately, the approach is virtually identical to what we've done to compute $e_1[x_1, \dotsc, x_n] \Pn$. As there is little extra insight to be gained from the derivation, those computations are left to the appendices. The primary difference is that the nonsymmetric variables are split into two groups, those in $[1,j-1]$ and those in $[j,k]$, which act differently since the operators $H_j, \dotsc, H_k$ act after $\Phi$, and $H_1, \dotsc, H_{j-1}$ act before $\Phi$ in the interpolation polynomial computations.

\begin{thm}\label{nonsymmetricformula} Let $j \in [1,k]$. Then the expansion of the product $x_j \PB{\gammalambdamin}$ is,

\begin{equation} \label{monomialexpansion} z_j \PB{\gammalambdamin} =\sum_{\substack{|\etamu| = |\lambdagamma| + 1 \\ \mu_1 \geq \, \dotsm \, \geq \mu_{n-k}}}  \DBaratta \cdot \PB{\etamumin}, \end{equation}

where $\mu$, $\lambdatilde$, and $\mutilde$ are as found in \cref{Pieriformula}, and $\eta$ is
\begin{align}
    \eta_\ell = \begin{cases}
    \gamma_\ell  &\text{ if } 1\leq \ell\leq j-1 \text{ and } \ell \notin I'_1\\
    \gamma_{t_{i+1}} &\text{ if } \ell = t_i \text{ and } i\neq r\\
    \gamma_j &\text{ if } \ell = t_r\\
    \gamma_{y_1+1} &\text{ if } \ell = j\\
    \gamma_\ell &\text{ if } j+1 \leq \ell \leq k \text{ and } \ell-1 \notin I'_3\\
    \gamma_{y_{i+1}+1} &\text{ if } \ell=y_i+1 \text{ and } 1\leq i < c\\
    \lambdatilde_1 &\text{ if } \ell = y_c+1
    \end{cases}
\end{align}
for some maximal $I'_1 = \{t_1, \dotsc, t_r\} \subseteq [1,j-1]$ and $I'_3 = \{y_1, \dotsc, y_c\} \subseteq [j,k]$. The coefficients $\DBaratta$ have the formula,

\begin{equation} \DBaratta =C_{\etamumin}^{\gammalambdamin} \cdot \frac{ \overline{\gammanu}_j}{ \overline{\gammanu}_{k+1}} \cdot t^m \cdot \frac{1-t}{1-t^{m+1}} \cdot  \frac{\overline{\eta}_j(\overline{\eta}_{t_r} - \overline{\eta}_{y_1})}{\overline{\eta}_{t_r}(\overline{\eta}_j-\overline{\eta}_{y_1})}\cdot \prod_{u=j+1}^{y_1-1} \frac{\overline{\eta}_j-t\overline{\eta}_u}{\overline{\eta}_j-\overline{\eta}_u} \cdot \prod_{u=t_r+1}^{y_1-1} \frac{\overline{\eta}_{t_r}-\overline{\eta}_u}{\overline{\eta}_{t_r}-t\overline{\eta}_u},\end{equation}

\end{thm}

\section{Cancellation} \label{Cancellation}
The formula in \cref{Jpieri} is highly unsimplified. We conclude our work by cancelling that formula as far as possible, with the aim of a condensed form which will prove useful in upcoming work with D. Orr.
\subsection{Internal cancellations in $\flambdatilde/\fmutilde$.} Recall that $\lambdatilde$ and $\mutilde$ are almost antidominant, but both have columns of height $\mutilde_{n-k}$ as far right as possible, and $\lambdatilde_{n-k}$ is some distinguished entry from $\lambda$. If $\lambdatilde \neq \mutilde$, the only entry which differs between the two diagrams is the $(n-k)$th position. Consider $\flambdatilde$ and $\fmutilde$ using the recursive formula. Each symmetric column of height $\mutilde_{n-k}$ is permuted past taller columns in both $\mutilde$ and $\lambdatilde$, with this happening one more time in $\mutilde$ than $\lambdatilde$. The term which contributes multiplicatively to $\fmutilde$ and $\flambdatilde$ is identical for all the columns of height $\mutilde_{n-k}$ that are in both diagrams, as the boxes in row $\mutilde_{n-k}+1$ of the symmetric diagrams have the same number of arms and legs before and after columns are cycled.

There is one column which is permuted to the right from each of $\mutilde$ and $\lambdatilde$ which contributes to $\flambdatilde/\fmutilde$ and is not cancelled. For $\flambdatilde$, this is the column $\lambdatilde_{n-k}$, and for $\fmutilde$, this is the leftmost column of height $\mutilde_{n-k}$. Again using the recursive formula, the contributions are as follows:
\begin{itemize}
    \item For $\flambdatilde$, permuting the entry $\lambdatilde_{n-k}$ all the way to the right contributes\\ $\displaystyle{\prod_{ \square \in d_{(\lambdatilde_{n-k}+1)}(\lambda^-)}\frac{t-q^{\ell_{\lambdamingamma}(\square)+1}t^{a_{\lambdamingamma}(\square)+1}}{1-q^{\ell_{\lambdamingamma}(\square)+1}t^{a_{\lambdamingamma}(\square)+1}}}$\\
    \item Similarly, in $\fmutilde$, permuting the leftmost entry $\mutilde_{n-k}$ as far right as possible (left of the rest of the columns of that height) contributes\\
    $\displaystyle{\prod_{\square \in d_{(\mutilde_{n-k}+1)}(\mutilde)}\frac{t-q^{\ell_{\mutildeeta}(\square)+1}t^{a_{\mutildeeta}(\square)}}{1-q^{\ell_{\mutildeeta}(\square)+1}t^{a_{\mutildeeta}(\square)}}}$
\end{itemize}
Therefore we obtain the cancelled form,
\[ \frac{\flambdatilde}{\fmutilde} = \prod_{ \square \in d_{(\lambdatilde_{n-k}+1)}(\lambda^-)}\frac{t-q^{\ell_{\lambdamingamma}(\square)+1}t^{a_{\lambdamingamma}(\square)+1}}{1-q^{\ell_{\lambdamingamma}(\square)+1}t^{a_{\lambdamingamma}(\square)+1}} \left( \prod_{ \square \in d_{(\mutilde_{n-k}+1)}(\mutilde)}\frac{t-q^{\ell_{\mutildeeta}(\square)+1}t^{a_{\mutildeeta}(\square)}}{1-q^{\ell_{\mutildeeta}(\square)+1}t^{a_{\mutildeeta}(\square)}}\right)^{-1} \]

Note that if $\lambdatilde=\mutilde$, the formula still holds, but the argument works by moving just the $\lambdatilde_{n-k}$ column to the $(n-k)$th position, and moving the leftmost column of height $\mutilde_{n-k}$ to its final position, and the rows of boxes in $d_{(\lambdatilde_{n-k}+1)}(\lambda^-) $ and $d_{(\mutilde_{n-k}+1)}(\mutilde)$ are the same row.

\subsection{Internal cancellations in $j_\lambdagamma/j_\mueta$} To simplify this constant, we need to identify and sort all boxes in the diagrams $\lambdamingamma$ and $\mumineta$ whose contributions to $j_\lambdagamma$ and $j_\mueta$ are different. The diagram at the end of this section gives a visual representation of these groups. The groups that are in $\lambdamingamma$ are labelled with numbers and a subscript $j$, and the corresponding groups in $\mumineta$, namely the same boxes after columns are cycled, are labelled with the same group number, $j$, and a prime.

First recall that $j_\lambdagamma$ and $j_\mueta$ are calculated using the diagrams $\lambdamingamma$ and $\mumineta$, with the symmetric parts weakly increasing (antidominant). 
Beginning with columns which are in the symmetric part of both diagrams, the only symmetric boxes whose statistics change after columns are cyclically permuted are those in the same row as the newly added box. Each of these gains one (modified) arm after cycling. So if we let
\[ \group{1}_j'=\left\{\square \in d_{(\mutilde_{n-k})}(\mutilde) \, \mid \, \square \text{ is not the newly added box}\right\},\]
then the contribution of these boxes to the product $j_\lambdagamma/j_\mueta$ is
\begin{equation}\label{groupAproduct} \jone := \prod_{\square \in \group{1}_j'} \frac{1-q^{\leg{\mumineta}}t^{\armone{\mumineta}}}{1-q^{\leg{\mumineta}}t^{\armone{\mumineta}+1}}.\end{equation}

Next, we address the nonsymmetric columns, including the column which is moved from $\lambda$. First suppose $I_1=[k+1,n]$, i.e., each $T_i$ for $k+1\leq i \leq n-1$ acts as a constant, not permuting columns. Then the only boxes whose statistics change are those that moved from the rightmost symmetric column of height $\lambdatilde_{n-k}$ (call this $\group{2}_j$) to the first column of $\eta$ (call this $\group{2}_j'$).
In this case, the contribution to $j_\lambdagamma/j_\mueta$ is
\begin{equation}\jtwo := \frac{\prod\limits_{\square \in \group{2}_j}1-q^{\leg{\lambdamingamma}}t^{\armone{\lambdamingamma}+1} }{\prod\limits_{\square \in \group{2}_j'} 1-q^{\leg{\mumineta}+1}t^{\armtwo{\mumineta}+1}}.\end{equation}

Notice that for a box $\square=(\lambdatilde_{n-k},j)\in \group{2}_j$ and its corresponding box $ \square = (\eta_{t_1},j) \in \group{2}'_j$, the boxes have the same number of legs, and the arms only differ by the number of nonsymmetric columns of height $j-1$, which are not counted in $\armone{\lambdamingamma}$. So we can rewrite $\jtwo$ as a product of rational functions in only one diagram,
\begin{equation} \label{groupBproduct} \jtwo  = \prod\limits_{\square \in \group{2}_j}\frac{1-q^{\leg{\lambdamingamma}}t^{\armone{\lambdamingamma}+1} }{ 1-q^{\leg{\lambdamingamma}+1}t^{\armone{\lambdamingamma}+1+m_{j-1}(\eta)}}, \end{equation}
where $m_{j-1}(\eta)$ is the number of entries in $\eta$ that are equal to $j-1$.

Now we extend this to where $I_1 \neq [k+1,n]$ by considering what happens to the diagram if the action of $T_i$ includes the permutation $s_i$. From the columns in position $i$ and $i+1$, the only box whose statistic is changed from $\lambdamingamma$ is
\[ \square = \begin{cases} (i+1,\gamma_i+1) \text{ if } \gamma_i < \gamma_{i+1}\\ (i,\gamma_{i+1}+1) \text{ if } \gamma_i > \gamma_{i+1} \end{cases}\]
In the first case, the box has one more arm in $\mumineta$ than in $\lambdamingamma$, and in the second case, the box has one less arm in $\mumineta$ than $\lambdamingamma$. We will call these boxes $\group{3}_j$ when considered in $\lambdamingamma$, and $\group{3}_j'$ when considered in $\mumineta$. The boxes in the column which become $\mutilde_h$, as well as boxes in columns $\{t_r+n-k, \dotsc, n\}$ will be excluded from this group and dealt with separately. Then the contribution for these boxes is
\begin{equation}\label{groupCproduct} \jthree := \frac{\prod\limits_{\square \in \group{3}_j}1-q^{\leg{\lambdamingamma}+1}t^{\armtwo{\lambdamingamma}+1} }{\prod\limits_{\square \in \group{3}_j'} 1-q^{\leg{\mumineta}+1}t^{\armtwo{\mumineta}+1}}.\end{equation}

Note that $\group{3}_j$ might include a box from the column $\group{2}_j$, if the action of $T_{n-k}$ includes the permutation $s_{n-k}$. While this could be cancelled further with terms in $\jtwo$, it is more useful to leave the product uncancelled in this form for now.

The final set of boxes to consider are in the column which wraps around and gains a box, and any boxes in columns $\{t_r+n-k, \dotsc, n\}$ whose statistics were affected by the column that wrapped around. Like in $\group{3}_j$, each time the column $\gamma_{t_r}$ passes another column, one box in either the $\gamma_{t_r}$ column or the column it passes gains or loses an arm. If $i>t_r$ and $\eta_i \geq \mutilde_h$, then the box $(i+n-k,\mutilde_h)$ gains one arm. And if $\eta_i < \mutilde_h$, the box is in the column which becomes $\mutilde_h$. So our next group is
\[ \group{4}_j' :=  \{ \square = (i,\mutilde_h)\, \mid \, t_r+n-k < i\}\]
and the contribution from the boxes in these columns to $j$ is
\[ \jfour := \prod_{\square \in \group{4}'_j} \frac{1-q^{\leg{\mutildeeta}+1}t^{\armtwo{\mutildeeta}}}{1-q^{\leg{\mutildeeta}+1}t^{\armtwo{\mutildeeta}+1}}\]

We finally account for the boxes which land in $\mutilde_h$. The newly-added box has no arms or legs, so its contribution to $j_\mueta$ is $1-t$. Every other box in the wrapped-around column has one more leg in $\mueta$, which means the power of $q$ in the contribution of the box in $j_\lambdagamma$ is the same as in $j_\mueta$, since in $j_\lambdagamma$ the box is in the nonsymmetric part of the diagram, and in $j_\mueta$ it is in the symmetric part. Let $\square=(i,j)$ be a box not on top, so $j<\mutilde_h$. Then we make three observations:
\begin{enumerate}
    \item Boxes in the newly-increased column of $\mumineta$ have the same number of arms as their counterparts in $\mutildeeta$, so for the sake of future cancellation, we will write in terms of $\mutildeeta$.\\
    \item $a_\lambdatildegamma(t_r+n-k,j) = a_\mutildeeta(h,j)$.\\
    \item $\widetilde{a}_\mutildeeta(h,j) = a_\mutildeeta(h,j)-g_{j-1}$, where $g_\ell := \# \{i>t_r \, \mid \, \eta_i = \ell \}$.\\
\end{enumerate}
Define $\jfive$ to be the rational function contribution of boxes which move to $\mutilde_h$, and its corresponding group of boxes,
\[ \group{5}_j' := \{ \square = (h,i) \, \mid \, 1\leq i \leq \mutilde_h\}.\] Using the above observations, we find
\[ \jfive = \frac{1}{1-t} \prod_{\substack{\square \in \group{5}_j' \\ \square = (h,i) \\
i\neq \mutilde_h}} \frac{1-q^{\leg{\mutildeeta}}t^{\armone{\mutildeeta}+1+g_{i-1}}}{1-q^{\leg{\mutildeeta}}t^{\armone{\mutildeeta}+1}}.\]

With all the groups now indexed, we can rewrite $j_\lambdagamma/j_\mueta$:
\[j_\lambdagamma/j_\mueta = \jone \cdot \jtwo \cdot \jthree \cdot \jfour \cdot \jfive. \]
\begin{figure}
\centering

\begin{tikzpicture}[scale=0.9,fill opacity=0.25,every node/.style={scale=0.9}]
\draw (0,0) -- (7,0);
\draw (0,1) -- (7,1);
\draw (1,2) -- (7,2);
\draw (2,3) -- (7,3);
\draw (3,4) -- (7,4);
\draw (4,5) -- (6,5);
\draw (5,6) -- (6,6);
\draw (0,0) -- (0,1);
\draw (1,0) -- (1,2);
\draw (2,0) -- (2,3);
\draw (3,0) -- (3,4);
\draw (4,0) -- (4,5);
\draw (5,0) -- (5,6);
\draw (6,0) -- (6,6);
\draw (7,0) -- (7,4);
\draw (8,0) -- (8,1) -- (9,1) -- (9,0) -- (8,0);
\draw (10,0) -- (14,0);
\draw (10,1) -- (14,1);
\draw (10,2) -- (14,2);
\draw (11,3) -- (12,3);
\draw (13,3) -- (14,3);
\draw (13,4) -- (14,4);
\draw (13,5) -- (14,5);
\draw (10,0) -- (10,2);
\draw (11,0) -- (11,3);
\draw (12,0) -- (12,3);
\draw (13,0) -- (13,5);
\draw (14,0) -- (14,5);
\draw[dashed] (6,-1) -- (6,7);
\fill[gray] (0,0) rectangle (1,1);
\node[fill opacity = 1] at (.5,.5) {\small $ 0,0$};

\fill[yellow] (1,0) rectangle (2,0.5);
\fill[blue] (1,0.5) rectangle (2,1);
\node[fill opacity = 1] at (1.5,.5) {\small $ 1,2 $};
\fill[yellow] (1,1) rectangle (2,2);
\node[fill opacity = 1] at (1.5,1.5) {\small $ 0,0 $};
\fill[gray] (2,0) rectangle (3,1);
\node[fill opacity = 1] at (2.5,.5) {\small $ 2,5 $};
\fill[gray] (2,1) rectangle (3,2);
\node[fill opacity = 1] at (2.5,1.5) {\small $ 1,3 $};
\fill[gray] (2,2) rectangle (3,3);
\node[fill opacity = 1] at (2.5,2.5) {\small $ 0,0$};
\fill[gray] (3,0) rectangle (4,1);
\node[fill opacity = 1] at (3.5,.5) {\small $ 3,7 $};
\fill[gray] (3,1) rectangle (4,2);
\node[fill opacity = 1] at (3.5,1.5) {\small $ 2,5 $};
\fill[gray] (3,2) rectangle (4,3);
\node[fill opacity = 1] at (3.5,2.5) {\small $ 1,2 $};
\fill[green] (3,3) rectangle (4,4);
\node[fill opacity = 1] at (3.5,3.5) {\small $ 0,0 $};
\fill[gray] (4,0) rectangle (5,1);
\node[fill opacity = 1] at (4.5,.5) {\small $ 4,9 $};
\fill[gray] (4,1) rectangle (5,2);
\node[fill opacity = 1] at (4.5,1.5) {\small $ 3,7 $};
\fill[gray] (4,2) rectangle (5,3);
\node[fill opacity = 1] at (4.5,2.5) {\small $ 2,4 $};
\fill[green] (4,3) rectangle (5,4);
\node[fill opacity = 1] at (4.5,3.5) {\small $ 1,2 $};
\fill[gray] (4,4) rectangle (5,5);
\node[fill opacity = 1] at (4.5,4.5) {\small $ 0,0 $};
\fill[gray] (5,0) rectangle (6,1);
\node[fill opacity = 1] at (5.5,.5) {\small $ 5,11 $};
\fill[gray] (5,1) rectangle (6,2);
\node[fill opacity = 1] at (5.5,1.5) {\small $ 4,9 $};
\fill[gray] (5,2) rectangle (6,3);
\node[fill opacity = 1] at (5.5,2.5) {\small $ 3,6 $};
\fill[green] (5,3) rectangle (6,4);
\node[fill opacity = 1] at (5.5,3.5) {\small $ 2,4 $};
\fill[gray] (5,4) rectangle (6,5);
\node[fill opacity = 1] at (5.5,4.5) {\small $ 1,2 $};
\fill[gray] (5,5) rectangle (6,6);
\node[fill opacity = 1] at (5.5,5.5) {\small $ 0,0 $};
\fill[gray] (6,0) rectangle (7,1);
\node[fill opacity = 1] at (6.5,.5) {\small $ 3,10 $};
\fill[gray] (6,1) rectangle (7,2);
\node[fill opacity = 1] at (6.5,1.5) {\small $ 2,7 $};
\fill[blue] (6,2) rectangle (7,3);
\node[fill opacity = 1] at (6.5,2.5) {\small $ 1,5 $};
\fill[gray] (6,3) rectangle (7,4);
\node[fill opacity = 1] at (6.5,3.5) {\small $ 0,2 $};
\fill[blue] (8,0) rectangle (9,1);
\node[fill opacity = 1] at (8.5,.5) {\small $ 0,2 $};
\node[fill opacity = 1] at (8.5,-.5) {$\gamma_{t_1}$};
\fill[gray] (10,0) rectangle (11,1);
\node[fill opacity = 1] at (10.5,.5) {\small $ 1,3 $};
\fill[blue] (10,1) rectangle (11,2);
\node[fill opacity = 1] at (10.5,1.5) {\small $ 0,1 $};
\node[fill opacity = 1] at (11.5,.5) {\small $ 2,6 $};
\node[fill opacity = 1] at (11.5,1.5) {\small $ 1,4 $};
\node[fill opacity = 1] at (11.5,2.5) {\small $ 0,2 $};
\node[fill opacity = 1] at (11.5,-.5) {$\gamma_{t_2}$};
\fill[gray] (12,0) rectangle (13,1);
\node[fill opacity = 1] at (12.5,.5) {\small $ 1,4 $};
\fill[gray] (12,1) rectangle (13,2);
\node[fill opacity = 1] at (12.5,1.5) {\small $ 0,2 $};
\fill[gray] (13,0) rectangle (14,1);
\node[fill opacity = 1] at (13.5,.5) {\small $ 4,10 $};
\fill[gray] (13,1) rectangle (14,2);
\node[fill opacity = 1] at (13.5,1.5) {\small $ 3,8 $};
\fill[gray] (13,2) rectangle (14,3);
\node[fill opacity = 1] at (13.5,2.5) {\small $ 2,5 $};
\fill[purple] (13,3) rectangle (14,4);
\node[fill opacity = 1] at (13.5,3.5) {\small $ 1,3 $};
\fill[gray] (13,4) rectangle (14,5);
\node[fill opacity = 1] at (13.5,4.5) {\small $ 0,1 $};

\draw (1,-2) -- (1,-1.5) -- (1.5,-1.5) -- (1.5, -2) -- (1, -2);
\fill[green] (1,-2) rectangle (1.5,-1.5);
\node[fill opacity = 1] at (2.7, -1.8) {$\in \group{1}_j$};

\draw (5,-2) -- (5,-1.5) -- (5.5,-1.5) -- (5.5, -2) -- (5, -2);
\fill[yellow] (5,-2) rectangle (5.5,-1.5);
\node[fill opacity = 1] at (6.7, -1.8) {$\in \group{2}_j$};

\draw (9,-2) -- (9,-1.5) -- (9.5,-1.5) -- (9.5, -2) -- (9, -2);
\fill[blue] (9,-2) rectangle (9.5,-1.5);
\node[fill opacity = 1] at (10.7, -1.8) {$\in \group{3}_j$};

\draw (1,-4) -- (1,-3.5) -- (1.5,-3.5) -- (1.5, -4) -- (1, -4);
\fill[purple] (1,-4) rectangle (1.5,-3.5);
\node[fill opacity = 1] at (2.7, -3.8) {$\in \group{4}_j$};

\draw (5,-4) -- (5,-3.5) -- (5.5,-3.5) -- (5.5, -4) -- (5, -4);
\node[fill opacity = 1] at (6.7, -3.8) {$\in \group{5}_j$};

\draw (9,-4) -- (9,-3.5) -- (9.5,-3.5) -- (9.5, -4) -- (9, -4);
\fill[gray] (9,-4) rectangle (9.5,-3.5);
\node[fill opacity = 1] at (10.7, -3.8) {Cancels};

\end{tikzpicture}

\begin{tikzpicture}[scale=0.9,fill opacity=.25,every node/.style={scale=0.9}]

\draw (0,0) -- (7,0);
\draw (0,1) -- (7,1);
\draw (1,2) -- (7,2);
\draw (1,3) -- (7,3);
\draw (2,4) -- (7,4);
\draw (4,5) -- (6,5);
\draw (5,6) -- (6,6);
\draw (0,0) -- (0,1);
\draw (1,0) -- (1,3);
\draw (2,0) -- (2,4);
\draw (3,0) -- (3,4);
\draw (4,0) -- (4,5);
\draw (5,0) -- (5,6);
\draw (6,0) -- (6,6);
\draw (7,0) -- (7,4);
\draw (8,0) -- (8,1) -- (9,1) -- (9,0) -- (8,0);
\draw (8,1) -- (8,2) -- (9,2) -- (9,1);
\draw (10,0) -- (14,0);
\draw (10,1) -- (14,1);
\draw (10,2) -- (11,2);
\draw (10,0) -- (10,2);
\draw (11,0) -- (11,2);
\draw (12,0) -- (12,2);
\draw (13,0) -- (13,5);
\draw (14,0) -- (14,5);
\draw (12,2) -- (14,2);
\draw (13,3) -- (14,3);
\draw (13,4) -- (14,4);
\draw (13,5) -- (14,5);
\draw[dashed] (6,-1) -- (6,7);

\node[fill opacity = 1] at (.5,.5) {\small $ 0,0$};
\fill[gray] (0,0) rectangle (1,1);
\fill[gray] (1,0) rectangle (2,1);
\node[fill opacity = 1] at (1.5,.5) {\small $ 2,5 $};
\fill[gray] (1,1) rectangle (2,2);
\node[fill opacity = 1] at (1.5,1.5) {\small $ 1,3 $};
\fill[gray] (1,2) rectangle (2,3);
\node[fill opacity = 1] at (1.5,2.5) {\small $ 0,0 $};
\node[fill opacity = 1] at (2.5,.5) {\small $ 3,6 $};
\node[fill opacity = 1] at (2.5,1.5) {\small $ 2,4 $};
\node[fill opacity = 1] at (2.5,2.5) {\small $ 1,1 $};
\node[fill opacity = 1] at (2.5,3.5) {\small $ 0,0 $};
\fill[gray] (3,0) rectangle (4,1);
\node[fill opacity = 1] at (3.5,.5) {\small $ 3,7 $};
\fill[gray] (3,1) rectangle (4,2);
\node[fill opacity = 1] at (3.5,1.5) {\small $ 2,5 $};
\fill[gray] (3,2) rectangle (4,3);
\node[fill opacity = 1] at (3.5,2.5) {\small $ 1,2 $};
\fill[green] (3,3) rectangle (4,4);
\node[fill opacity = 1] at (3.5,3.5) {\small $ 0,1 $};
\fill[gray] (4,0) rectangle (5,1);
\node[fill opacity = 1] at (4.5,.5) {\small $ 4,9 $};
\fill[gray] (4,1) rectangle (5,2);
\node[fill opacity = 1] at (4.5,1.5) {\small $ 3,7 $};
\fill[gray] (4,2) rectangle (5,3);
\node[fill opacity = 1] at (4.5,2.5) {\small $ 2,4 $};
\fill[green] (4,3) rectangle (5,4);
\node[fill opacity = 1] at (4.5,3.5) {\small $ 1,3 $};
\fill[gray] (4,4) rectangle (5,5);
\node[fill opacity = 1] at (4.5,4.5) {\small $ 0,0 $};
\fill[gray] (5,0) rectangle (6,1);
\node[fill opacity = 1] at (5.5,.5) {\small $ 5,11 $};
\fill[gray] (5,1) rectangle (6,2);
\node[fill opacity = 1] at (5.5,1.5) {\small $ 4,9 $};
\fill[gray] (5,2) rectangle (6,3);
\node[fill opacity = 1] at (5.5,2.5) {\small $ 3,6 $};
\fill[green] (5,3) rectangle (6,4);
\node[fill opacity = 1] at (5.5,3.5) {\small $ 2,5 $};
\fill[gray] (5,4) rectangle (6,5);
\node[fill opacity = 1] at (5.5,4.5) {\small $ 1,2 $};
\fill[gray] (5,5) rectangle (6,6);
\node[fill opacity = 1] at (5.5,5.5) {\small $ 0,0 $};
\fill[gray] (6,0) rectangle (7,1);
\node[fill opacity = 1] at (6.5,.5) {\small $ 3,10 $};
\fill[gray] (6,1) rectangle (7,2);
\node[fill opacity = 1] at (6.5,1.5) {\small $ 2,7 $};
\fill[blue] (6,2) rectangle (7,3);
\node[fill opacity = 1] at (6.5,2.5) {\small $ 1,6 $};
\fill[gray] (6,3) rectangle (7,4);
\node[fill opacity = 1] at (6.5,3.5) {\small $ 0,2 $};
\fill[yellow] (8,0) rectangle (9,.5);
\fill[blue] (8,0.5) rectangle (9,1);
\node[fill opacity = 1] at (8.5,.5) {\small $ 1,3 $};
\fill[yellow] (8,1) rectangle (9,2);
\node[fill opacity = 1] at (8.5,1.5) {\small $ 0,1 $};
\node[fill opacity = 1] at (8.5,-.5) {$\eta_{t_1}$};
\fill[gray] (10,0) rectangle (11,1);
\node[fill opacity = 1] at (10.5,.5) {\small $ 1,3 $};
\fill[blue] (10,1) rectangle (11,2);
\node[fill opacity = 1] at (10.5,1.5) {\small $ 0,2 $};
\fill[blue] (11,0) rectangle (12,1);
\node[fill opacity = 1] at (11.5,.5) {\small $ 0,1 $};
\node[fill opacity = 1] at (11.5,-.5) {$\eta_{t_2}$};
\fill[gray] (12,0) rectangle (13,1);
\node[fill opacity = 1] at (12.5,.5) {\small $ 1,4 $};
\fill[gray] (12,1) rectangle (13,2);
\node[fill opacity = 1] at (12.5,1.5) {\small $ 0,2 $};
\fill[gray] (13,0) rectangle (14,1);
\node[fill opacity = 1] at (13.5,.5) {\small $ 4,10 $};
\fill[gray] (13,1) rectangle (14,2);
\node[fill opacity = 1] at (13.5,1.5) {\small $ 3,8 $};
\fill[gray] (13,2) rectangle (14,3);
\node[fill opacity = 1] at (13.5,2.5) {\small $ 2,5 $};
\fill[purple] (13,3) rectangle (14,4);
\node[fill opacity = 1] at (13.5,3.5) {\small $ 1,4 $};
\fill[gray] (13,4) rectangle (14,5);
\node[fill opacity = 1] at (13.5,4.5) {\small $ 0,1 $};

\draw (1,-2) -- (1,-1.5) -- (1.5,-1.5) -- (1.5, -2) -- (1, -2);
\fill[green] (1,-2) rectangle (1.5,-1.5);
\node[fill opacity = 1] at (2.7, -1.8) {$\in \group{1}'_j$};

\draw (5,-2) -- (5,-1.5) -- (5.5,-1.5) -- (5.5, -2) -- (5, -2);
\fill[yellow] (5,-2) rectangle (5.5,-1.5);
\node[fill opacity = 1] at (6.7, -1.8) {$\in \group{2}'_j$};

\draw (9,-2) -- (9,-1.5) -- (9.5,-1.5) -- (9.5, -2) -- (9, -2);
\fill[blue] (9,-2) rectangle (9.5,-1.5);
\node[fill opacity = 1] at (10.7, -1.8) {$\in \group{3}'_j$};

\draw (1,-4) -- (1,-3.5) -- (1.5,-3.5) -- (1.5, -4) -- (1, -4);
\fill[purple] (1,-4) rectangle (1.5,-3.5);
\node[fill opacity = 1] at (2.7, -3.8) {$\in \group{4}'_j$};

\draw (5,-4) -- (5,-3.5) -- (5.5,-3.5) -- (5.5, -4) -- (5, -4);
\node[fill opacity = 1] at (6.7, -3.8) {$\in \group{5}'_j$};

\draw (9,-4) -- (9,-3.5) -- (9.5,-3.5) -- (9.5, -4) -- (9, -4);
\fill[gray] (9,-4) rectangle (9.5,-3.5);
\node[fill opacity = 1] at (10.7, -3.8) {Cancels};

\end{tikzpicture}

\caption{Diagrams for $j_{(6,5,4,3,2,1 \mid 4,0,1,0,2,3,2,5)}$ (top) and $j_{(6,5,4,4,3,1 \mid 4,0,2,0,2,1,2,5)}$ (bottom). Each box is filled with (legs, arms).}
\end{figure}
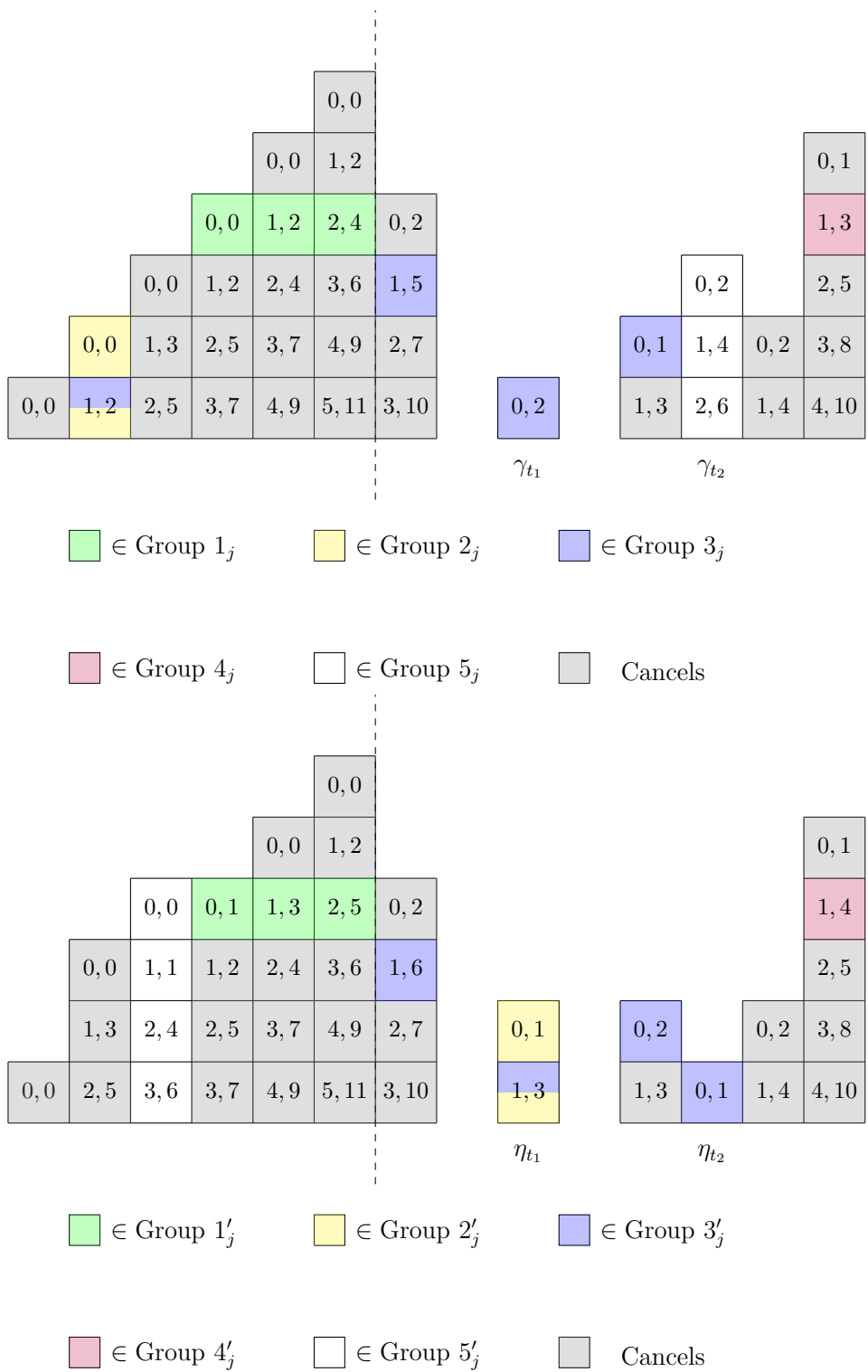

\subsection{Internal cancellations in $\Estar{\lambdatildegamma}\overline{\lambdatildegamma}/\Estar{\mutildeeta}\overline{\mutildeeta}$.} This will be very similar to the cancellations in $j_\lambdagamma/j_\mueta$. We will group things in the same ways, but since we are looking at the diagrams $\lambdatildegamma$ and $\mutildeeta$ now, we will label the groups with the subscript $E$. Recall the formula,
\[ \frac{\Estar{\lambdatildegamma}\overline{\lambdatildegamma}}{\Estar{\mutildeeta}\overline{\mutildeeta}} = \frac{\left( \prod\limits_{i=1}^n \overline{\lambdatildegamma}_i^{\lambdatildegamma_i} \right) \prod\limits_{\square \in \lambdatildegamma} \left(1-q^{-(\ell_\lambdatildegamma(\square)+1)} t^{-a_\lambdatildegamma(\square)}\right)}{\left( \prod\limits_{i=1}^n \overline{\mutildeeta}_i^{\mutildeeta_i} \right) \prod\limits_{\square \in \mutildeeta} \left(1-q^{-(\ell_\mutildeeta(\square)+1)} t^{-a_\mutildeeta(\square)}\right)}.\]

To simplify the products of eigenvalues on the left, recall from \cref{comaximality} and \cref{maximality} that since $I_1$ is maximal, the permutation of columns from $\lambdatildegamma$ to $\mutildeeta$ also permutes the associated eigenvalues, except in the column which gains a box. Combining this with the facts that $\mutilde_h = \gamma_{t_r}+1$ and $\ell'_{\mutildeeta}(h) = \ell'_{\lambdatildegamma}(t_r+n-k)$, we get
\begin{align*}
    \frac{\left( \prod\limits_{i=1}^n \overline{\lambdatildegamma}_i^{\lambdatildegamma_i} \right) }{\left( \prod\limits_{i=1}^n \overline{\mutildeeta}_i^{\mutildeeta_i} \right) } = \frac{\left(\overline{\lambdatildegamma}_{t_r+n-k}\right)^{\gamma_{t_r}}}{\left(\overline{\lambdatildegamma}_{h}\right)^{\mutilde_{h}}} = \frac{\left(q^{\gamma_{t_r}}t^{-\ell'_{\lambdatildegamma}(t_r+n-k)}\right)^{\gamma_{t_r}}}{\left(q^{\mutilde_h}t^{-\ell'_{\mutildeeta}(h)}\right)^{\mutilde_h}} = q^{-2\mutilde_h+1}t^{\ell'_{\mutildeeta}(h).}
\end{align*} 

Once again, if we compare a corresponding box in $\lambdatildegamma$ and $\mutildeeta$ and it has the same arms and legs in both diagrams, those terms cancel. Beginning like before with boxes which are in the symmetric parts of both diagrams, the boxes which have different statistics are in the $(\mutilde_{h}+1)$st row counting from the bottom, so define groups,
\begin{align*}
    \group{1}_E &:= d_{(\mutilde_{n-k}+1)}(\lambdatilde)\\
    \group{1}_E' &:= d_{(\mutilde_{n-k}+1)}(\mutilde)
\end{align*}
The boxes in group $1'_E$ have the same legs and one more arm, namely the column $\mutilde_h$, and so we can write
\begin{align*} \Estarone &:= \frac{ \prod\limits_{\square \in \group{1}_E} \left(1-q^{-(\ell_\lambdatildegamma(\square)+1)} t^{-a_\lambdatildegamma(\square)}\right)}{ \prod\limits_{\square \in \group{1}_E'} \left(1-q^{-(\ell_\mutildeeta(\square)+1)} t^{-a_\mutildeeta(\square)}\right)}\\
&= \prod\limits_{\square \in \group{1}_E'}\frac{ \left(1-q^{-(\ell_\mutildeeta(\square)+1)} t^{-(a_\mutildeeta(\square)-1)}\right)}{  \left(1-q^{-(\ell_\mutildeeta(\square)+1)} t^{-a_\mutildeeta(\square)}\right)}\\
&=  \prod\limits_{\square \in \group{1}_E'}\frac{ \left(t-q^{\ell_\mutildeeta(\square)+1} t^{a_\mutildeeta(\square)}\right)}{  \left(1-q^{\ell_\mutildeeta(\square)+1} t^{a_\mutildeeta(\square)}\right)}\end{align*}
Notice that this matches exactly with the remaining denominator from $\flambdatilde/\fmutilde$, so those will cancel.

Unlike with the $j$'s, because the $E^*$ formula is all in nonsymmetric terms, the statistics of the boxes which cycle from column $\lambda_{n-k}$ to $\eta_1$ (prior to the $T_i$ actions) do not change, so there is in effect no $\group{2}_E$ to account for.

The analogs to $\group{3}_j$ and $\group{3}_j'$ will work exactly the same with $E^*$ as in the $j$'s, since both use the same diagram statistics (nonsymmetric arms). The resulting contribution from these boxes is then
\[\Estarthree := \frac{\prod\limits_{\square \in \group{3}_j}1-q^{-(\leg{\lambdatildegamma}+1)}t^{-\armtwo{\lambdatildegamma}} }{\prod\limits_{\square \in \group{3}_j'} 1-q^{-(\leg{\mutildeeta}+1)}t^{-\armtwo{\mutildeeta}}}.\]

The boxes in $\group{4}_j'$ also behave the same in $E^*$ and $j$, so the contribution for those boxes to the right of $\eta_{t_r}$ which change is
\[ \Estarfour := \prod_{\square \in \group{4}'_j} \frac{1-q^{-(\leg{\mutildeeta}+1)}t^{-(\armtwo{\mutildeeta}-1)}}{1-q^{-(\leg{\mutildeeta}+1)}t^{-\armtwo{\mutildeeta}}}\]

Finally, we address the group of boxes which go from column $\eta_{t_r}$ to $\mutilde_h$. Compare the boxes $(t_r+n-k,i)$ in $\lambdatildegamma$ and $(h,i)$ in $\mutildeeta$. The box in $\mutilde_h$ has the added box as an extra leg, and comparing $a_\lambdatildegamma(t_r+n-k,i)$ to the modified $\widetilde{a}_\mutildeeta(h,i)$, these both count the same columns except that $a_\lambdatildegamma(t_r+n-k,i)$ also counts boxes in columns to the right of $t_r+n-k$ with height $i-1$. We still need to account for the newly-added box, so altogether we have groups,
\begin{align*} \group{5}_E &:= \{\square = (t_r+n-k,i) \mid  1 \leq i \leq \gamma_{t_r} \}\\
\group{5}'_E &:= \{ \square = (h,i) \, \mid \, 1 \leq i \leq \mutilde_h \},
\end{align*}
and contribution of those groups,
\begin{align*}
    \Estarfive &:= \frac{\prod\limits_{\square \in \group{5}_E} 1-q^{-(\leg{\lambdatildegamma}+1)}t^{-\armtwo{\lambdagamma}}}{\prod\limits_{\square \in \group{5}'_E} 1-q^{-(\leg{\mutildeeta}+1)}t^{-\armtwo{\mutildeeta}}}\\
    &= \frac{1}{1-q^{-1}t^{-a_\mutildeeta(h,\mutilde_h)}}\cdot \prod\limits_{\substack{\square \in \group{5}'_E \\ \square = (h,i) \\ i\neq \mutilde_h}} \frac{1-q^{-\leg{\mutildeeta}}t^{-(\armone{\mutildeeta}+g_{i-1})}}{1-q^{-(\leg{\mutildeeta}+1)}t^{-\armtwo{\mutildeeta}}}
\end{align*}

With all boxes enumerated, we have the simplification,

\[\frac{\Estar{\lambdatildegamma}\overline{\lambdatildegamma}}{\Estar{\mutildeeta}\overline{\mutildeeta}} = q^{-2\mutilde_h+1}t^{\ell'_{\mutildeeta}(h)} \cdot \Estarone \cdot \Estarthree \cdot \Estarfour \cdot \Estarfive.\]

\begin{figure}
\centering
\begin{tikzpicture}[scale=0.9,fill opacity=0.25,every node/.style={scale=0.9}]
\draw (0,0) -- (7,0);
\draw (0,1) -- (7,1);
\draw (1,2) -- (7,2);
\draw (1,3) -- (5,3);
\draw (2,4) -- (5,4);
\draw (2,5) -- (4,5);
\draw (3,6) -- (4,6);
\draw (0,0) -- (0,1);
\draw (1,0) -- (1,3);
\draw (2,0) -- (2,5);
\draw (3,0) -- (3,6);
\draw (4,0) -- (4,6);
\draw (5,0) -- (5,4);
\draw (6,0) -- (6,4);
\draw (6,3) -- (7,3);
\draw (6,4) -- (7,4);
\draw (7,0) -- (7,4);
\draw (8,0) -- (8,1) -- (9,1) -- (9,0) -- (8,0);
\draw (10,0) -- (14,0);
\draw (10,1) -- (14,1);
\draw (10,2) -- (14,2);
\draw (11,3) -- (12,3);
\draw (13,3) -- (14,3);
\draw (13,4) -- (14,4);
\draw (13,5) -- (14,5);
\draw (10,0) -- (10,2);
\draw (11,0) -- (11,3);
\draw (12,0) -- (12,3);
\draw (13,0) -- (13,5);
\draw (14,0) -- (14,5);
\draw[dashed] (6,-1) -- (6,7);
\fill[gray] (0,0) rectangle (1,1);
\node[fill opacity = 1] at (.5,.5) {\small $ 0,2$};
\fill[gray] (1,0) rectangle (2,1);
\node[fill opacity = 1] at (1.5,.5) {\small $ 2,7 $};
\fill[gray] (1,1) rectangle (2,2);
\node[fill opacity = 1] at (1.5,1.5) {\small $ 1,4 $};
\fill[gray] (1,2) rectangle (2,3);
\node[fill opacity = 1] at (1.5,2.5) {\small $ 0,3 $};
\fill[gray] (2,0) rectangle (3,1);
\node[fill opacity = 1] at (2.5,.5) {\small $ 4,11 $};
\fill[gray] (2,1) rectangle (3,2);
\node[fill opacity = 1] at (2.5,1.5) {\small $ 3,8 $};
\fill[gray] (2,2) rectangle (3,3);
\node[fill opacity = 1] at (2.5,2.5) {\small $ 2,7$};
\fill[gray] (2,3) rectangle (3,4);
\node[fill opacity = 1] at (2.5,3.5) {\small $ 1,3$};
\fill[green] (2,4) rectangle (3,5);
\node[fill opacity = 1] at (2.5,4.5) {\small $ 0,2$};
\fill[gray] (3,0) rectangle (4,1);
\node[fill opacity = 1] at (3.5,.5) {\small $ 5,13 $};
\fill[gray] (3,1) rectangle (4,2);
\node[fill opacity = 1] at (3.5,1.5) {\small $ 4,10 $};
\fill[gray] (3,2) rectangle (4,3);
\node[fill opacity = 1] at (3.5,2.5) {\small $ 3,9 $};
\fill[gray] (3,3) rectangle (4,4);
\node[fill opacity = 1] at (3.5,3.5) {\small $ 2,5 $};
\fill[green] (3,4) rectangle (4,5);
\node[fill opacity = 1] at (3.5,4.5) {\small $ 1,4 $};
\fill[gray] (3,5) rectangle (4,6);
\node[fill opacity = 1] at (3.5,5.5) {\small $ 0,1 $};
\fill[gray] (4,0) rectangle (5,1);
\node[fill opacity = 1] at (4.5,.5) {\small $ 3,9 $};
\fill[gray] (4,1) rectangle (5,2);
\node[fill opacity = 1] at (4.5,1.5) {\small $ 2,6 $};
\fill[gray] (4,2) rectangle (5,3);
\node[fill opacity = 1] at (4.5,2.5) {\small $ 1,5 $};
\fill[gray] (4,3) rectangle (5,4);
\node[fill opacity = 1] at (4.5,3.5) {\small $ 0,1$};

\fill[blue] (5,0) rectangle (6,1);
\node[fill opacity = 1] at (5.5,.5) {\small $ 1,4 $};
\fill[gray] (5,1) rectangle (6,2);
\node[fill opacity = 1] at (5.5,1.5) {\small $ 0,1 $};

\fill[gray] (6,0) rectangle (7,1);
\node[fill opacity = 1] at (6.5,.5) {\small $ 3,10 $};
\fill[gray] (6,1) rectangle (7,2);
\node[fill opacity = 1] at (6.5,1.5) {\small $ 2,7 $};
\fill[blue] (6,2) rectangle (7,3);
\node[fill opacity = 1] at (6.5,2.5) {\small $ 1,5 $};
\fill[gray] (6,3) rectangle (7,4);
\node[fill opacity = 1] at (6.5,3.5) {\small $ 0,2 $};
\fill[blue] (8,0) rectangle (9,1);
\node[fill opacity = 1] at (8.5,.5) {\small $ 0,2 $};
\node[fill opacity = 1] at (8.5,-.5) {$\gamma_{t_1}$};
\fill[gray] (10,0) rectangle (11,1);
\node[fill opacity = 1] at (10.5,.5) {\small $ 1,3 $};
\fill[blue] (10,1) rectangle (11,2);
\node[fill opacity = 1] at (10.5,1.5) {\small $ 0,1 $};
\node[fill opacity = 1] at (11.5,.5) {\small $ 2,6 $};
\node[fill opacity = 1] at (11.5,1.5) {\small $ 1,4 $};
\node[fill opacity = 1] at (11.5,2.5) {\small $ 0,2 $};
\node[fill opacity = 1] at (11.5,-.5) {$\gamma_{t_2}$};
\fill[gray] (12,0) rectangle (13,1);
\node[fill opacity = 1] at (12.5,.5) {\small $ 1,4 $};
\fill[gray] (12,1) rectangle (13,2);
\node[fill opacity = 1] at (12.5,1.5) {\small $ 0,2 $};
\fill[gray] (13,0) rectangle (14,1);
\node[fill opacity = 1] at (13.5,.5) {\small $ 4,10 $};
\fill[gray] (13,1) rectangle (14,2);
\node[fill opacity = 1] at (13.5,1.5) {\small $ 3,8 $};
\fill[gray] (13,2) rectangle (14,3);
\node[fill opacity = 1] at (13.5,2.5) {\small $ 2,5 $};
\fill[purple] (13,3) rectangle (14,4);
\node[fill opacity = 1] at (13.5,3.5) {\small $ 1,3 $};
\fill[gray] (13,4) rectangle (14,5);
\node[fill opacity = 1] at (13.5,4.5) {\small $ 0,1 $};

\draw (1,-2) -- (1,-1.5) -- (1.5,-1.5) -- (1.5, -2) -- (1, -2);
\fill[green] (1,-2) rectangle (1.5,-1.5);
\node[fill opacity = 1] at (2.7, -1.8) {$\in \group{1}_E$};

\draw (9,-2) -- (9,-1.5) -- (9.5,-1.5) -- (9.5, -2) -- (9, -2);
\fill[blue] (9,-2) rectangle (9.5,-1.5);
\node[fill opacity = 1] at (10.7, -1.8) {$\in \group{3}_E$};

\draw (1,-4) -- (1,-3.5) -- (1.5,-3.5) -- (1.5, -4) -- (1, -4);
\fill[purple] (1,-4) rectangle (1.5,-3.5);
\node[fill opacity = 1] at (2.7, -3.8) {$\in \group{4}_E$};

\draw (5,-4) -- (5,-3.5) -- (5.5,-3.5) -- (5.5, -4) -- (5, -4);
\node[fill opacity = 1] at (6.7, -3.8) {$\in \group{5}_E$};

\draw (9,-4) -- (9,-3.5) -- (9.5,-3.5) -- (9.5, -4) -- (9, -4);
\fill[gray] (9,-4) rectangle (9.5,-3.5);
\node[fill opacity = 1] at (10.7, -3.8) {Cancels};

\end{tikzpicture}

\begin{tikzpicture}[scale=0.9,fill opacity=0.25,every node/.style={scale=0.9}]
\draw (0,0) -- (7,0);
\draw (0,1) -- (7,1);
\draw (1,2) -- (7,2);
\draw (1,3) -- (6,3);
\draw (2,4) -- (6,4);
\draw (2,5) -- (4,5);
\draw (3,6) -- (4,6);
\draw (0,0) -- (0,1);
\draw (1,0) -- (1,3);
\draw (2,0) -- (2,5);
\draw (3,0) -- (3,6);
\draw (4,0) -- (4,6);
\draw (5,0) -- (5,4);
\draw (6,0) -- (6,4);
\draw (6,3) -- (7,3);
\draw (6,4) -- (7,4);
\draw (7,0) -- (7,4);
\draw (8,0) -- (8,1) -- (9,1) -- (9,0) -- (8,0);
\draw (8,1) -- (8,2) -- (9,2) -- (9,1);
\draw (10,0) -- (14,0);
\draw (10,1) -- (14,1);
\draw (10,2) -- (11,2);
\draw (10,0) -- (10,2);
\draw (11,0) -- (11,2);
\draw (12,0) -- (12,2);
\draw (13,0) -- (13,5);
\draw (14,0) -- (14,5);
\draw (12,2) -- (14,2);
\draw (13,3) -- (14,3);
\draw (13,4) -- (14,4);
\draw (13,5) -- (14,5);
\draw[dashed] (6,-1) -- (6,7);

\fill[gray] (0,0) rectangle (1,1);
\node[fill opacity = 1] at (.5,.5) {\small $ 0,2$};
\fill[gray] (1,0) rectangle (2,1);
\node[fill opacity = 1] at (1.5,.5) {\small $ 2,7 $};
\fill[gray] (1,1) rectangle (2,2);
\node[fill opacity = 1] at (1.5,1.5) {\small $ 1,4 $};
\fill[gray] (1,2) rectangle (2,3);
\node[fill opacity = 1] at (1.5,2.5) {\small $ 0,3 $};
\fill[gray] (2,0) rectangle (3,1);
\node[fill opacity = 1] at (2.5,.5) {\small $ 4,11 $};
\fill[gray] (2,1) rectangle (3,2);
\node[fill opacity = 1] at (2.5,1.5) {\small $ 3,8 $};
\fill[gray] (2,2) rectangle (3,3);
\node[fill opacity = 1] at (2.5,2.5) {\small $ 2,7$};
\fill[gray] (2,3) rectangle (3,4);
\node[fill opacity = 1] at (2.5,3.5) {\small $ 1,3$};
\fill[green] (2,4) rectangle (3,5);
\node[fill opacity = 1] at (2.5,4.5) {\small $ 0,3$};
\fill[gray] (3,0) rectangle (4,1);
\node[fill opacity = 1] at (3.5,.5) {\small $ 5,13 $};
\fill[gray] (3,1) rectangle (4,2);
\node[fill opacity = 1] at (3.5,1.5) {\small $ 4,10 $};
\fill[gray] (3,2) rectangle (4,3);
\node[fill opacity = 1] at (3.5,2.5) {\small $ 3,9 $};
\fill[gray] (3,3) rectangle (4,4);
\node[fill opacity = 1] at (3.5,3.5) {\small $ 2,5 $};
\fill[green] (3,4) rectangle (4,5);
\node[fill opacity = 1] at (3.5,4.5) {\small $ 1,5 $};
\fill[gray] (3,5) rectangle (4,6);
\node[fill opacity = 1] at (3.5,5.5) {\small $ 0,1 $};
\node[fill opacity = 1] at (4.5,.5) {\small $ 3,8 $};
\node[fill opacity = 1] at (4.5,1.5) {\small $ 2,5 $};
\node[fill opacity = 1] at (4.5,2.5) {\small $ 1,4 $};
\node[fill opacity = 1] at (4.5,3.5) {\small $ 0,0 $};
\fill[gray] (5,0) rectangle (6,1);
\node[fill opacity = 1] at (5.5,.5) {\small $ 3,9 $};
\fill[gray] (5,1) rectangle (6,2);
\node[fill opacity = 1] at (5.5,1.5) {\small $ 2,6 $};
\fill[gray] (5,2) rectangle (6,3);
\node[fill opacity = 1] at (5.5,2.5) {\small $ 1,5 $};
\fill[gray] (5,3) rectangle (6,4);
\node[fill opacity = 1] at (5.5,3.5) {\small $ 0,1 $};
\fill[gray] (6,0) rectangle (7,1);
\node[fill opacity = 1] at (6.5,.5) {\small $ 3,10 $};
\fill[gray] (6,1) rectangle (7,2);
\node[fill opacity = 1] at (6.5,1.5) {\small $ 2,7 $};
\fill[blue] (6,2) rectangle (7,3);
\node[fill opacity = 1] at (6.5,2.5) {\small $ 1,6 $};
\fill[gray] (6,3) rectangle (7,4);
\node[fill opacity = 1] at (6.5,3.5) {\small $ 0,2 $};
\fill[blue] (8,0) rectangle (9,1);
\node[fill opacity = 1] at (8.5,.5) {\small $ 1,3 $};
\fill[gray] (8,1) rectangle (9,2);
\node[fill opacity = 1] at (8.5,1.5) {\small $ 0,1 $};
\node[fill opacity = 1] at (8.5,-.5) {$\eta_{t_1}$};
\fill[gray] (10,0) rectangle (11,1);
\node[fill opacity = 1] at (10.5,.5) {\small $ 1,3 $};
\fill[blue] (10,1) rectangle (11,2);
\node[fill opacity = 1] at (10.5,1.5) {\small $ 0,2 $};
\fill[blue] (11,0) rectangle (12,1);
\node[fill opacity = 1] at (11.5,.5) {\small $ 0,1 $};
\node[fill opacity = 1] at (11.5,-.5) {$\eta_{t_2}$};
\fill[gray] (12,0) rectangle (13,1);
\node[fill opacity = 1] at (12.5,.5) {\small $ 1,4 $};
\fill[gray] (12,1) rectangle (13,2);
\node[fill opacity = 1] at (12.5,1.5) {\small $ 0,2 $};
\fill[gray] (13,0) rectangle (14,1);
\node[fill opacity = 1] at (13.5,.5) {\small $ 4,10 $};
\fill[gray] (13,1) rectangle (14,2);
\node[fill opacity = 1] at (13.5,1.5) {\small $ 3,8 $};
\fill[gray] (13,2) rectangle (14,3);
\node[fill opacity = 1] at (13.5,2.5) {\small $ 2,5 $};
\fill[purple] (13,3) rectangle (14,4);
\node[fill opacity = 1] at (13.5,3.5) {\small $ 1,4 $};
\fill[gray] (13,4) rectangle (14,5);
\node[fill opacity = 1] at (13.5,4.5) {\small $ 0,1 $};

\draw (1,-2) -- (1,-1.5) -- (1.5,-1.5) -- (1.5, -2) -- (1, -2);
\fill[green] (1,-2) rectangle (1.5,-1.5);
\node[fill opacity = 1] at (2.7, -1.8) {$\in \group{1}'_E$};

\draw (9,-2) -- (9,-1.5) -- (9.5,-1.5) -- (9.5, -2) -- (9, -2);
\fill[blue] (9,-2) rectangle (9.5,-1.5);
\node[fill opacity = 1] at (10.7, -1.8) {$\in \group{3}'_E$};

\draw (1,-4) -- (1,-3.5) -- (1.5,-3.5) -- (1.5, -4) -- (1, -4);
\fill[purple] (1,-4) rectangle (1.5,-3.5);
\node[fill opacity = 1] at (2.7, -3.8) {$\in \group{4}'_E$};

\draw (5,-4) -- (5,-3.5) -- (5.5,-3.5) -- (5.5, -4) -- (5, -4);
\node[fill opacity = 1] at (6.7, -3.8) {$\in \group{5}'_E$};

\draw (9,-4) -- (9,-3.5) -- (9.5,-3.5) -- (9.5, -4) -- (9, -4);
\fill[gray] (9,-4) rectangle (9.5,-3.5);
\node[fill opacity = 1] at (10.7, -3.8) {Cancels};

\end{tikzpicture}
\caption{Diagrams for $E^*_{(1,3,5,6,4,2 \mid 4,0,1,0,2,3,2,5)}$ (top) and $E^*_{(1,3,5,6,4,4 \mid 4,0,2,0,2,1,2,5)}$ (bottom). Each box is filled with (legs, arms).}
\end{figure}

\subsection{Conversion of $p_{I_1}$ terms.} The remaining cancellations happen by combining some diagram statistic expressions involving $j_i$ terms and $E*_i$ terms, and eigenvalue functions from $p_{I_1}$ and $p_2$. This will simplify the formulas substantially, but requires some lengthy computations, all using similar techniques.

\begin{prop}\label{prop1conversion} \begin{equation}\label{pI1conversion1} \jthree \cdot \Estarthree \cdot \prod_{u=1}^r \prod_{j=t_{u-1}+1}^{t_u-1} \frac{\overline{\eta}_{t_u}-t \overline{\eta}_j}{\overline{\eta}_{t_u}-\overline{\eta}_j} = \prod_{u=1}^r \prod_{j=t_{u-1}+1}^{t_u-1} \frac{t\overline{\eta}_{t_u}- \overline{\eta}_j}{\overline{\eta}_{t_u}-\overline{\eta}_j}  \end{equation} \end{prop} 
\begin{proof} In order to cancel parts of $p_{I_1}$ with other terms, we first need to convert from eigenvalue expressions to expressions involving diagram statistics. We split this into two cases with $t_{u-1} < j < t_u$, the first with $\eta_j>\eta_{t_u}$, and the second with $\eta_j < \eta_{t_u}$. It is not possible that $\eta_j = \eta_{t_u}$, as that would mean $I_1$ is not maximal. Recall from \cref{eigenvalueformula} the formula for the eigenvalues,
\[\overline{\nu}_i = q^{\nu_i} t^{-l'_\nu(i)}, \hspace*{.5cm} 1\leq i \leq n,\] 
\[l'_\nu(i) = \# \{j<i \mid \nu_j > \nu_i\} + \# \{j>i \mid \nu_j \geq \nu_i\}.\]

\underline{\textbf{Case 1:}} Consider some column $\eta_j$ with $t_{u-1} < j < t_u$, and suppose $\eta_j > \eta_{t_u}$. Our goal is to rewrite the rational function
\[ \frac{\overline{\eta}_{t_u}-t \overline{\eta}_j}{\overline{\eta}_{t_u}-\overline{\eta}_j}\]
in terms of diagram statistics of the box $\square = (j+n-k,\eta_{t_u}+1)$ in the diagram $ \mutildeeta $. First, the difference in column height, $\eta_{j}-\eta_{t_u}$, is one more than the number of legs of $\square$. So $\eta_j-\eta_{t_u} = \leg{\mutildeeta}+1$. For readability, we will write $\lj$ and  $\ltm$ to mean $\ell'_{\mutildeeta}(j+n-k)$ and $\ell'_{\mutildeeta}(t_u+n-k)$ respectively. Now, we find the differences between the $\ell'_\mutildeeta$ values for the columns $\eta_{t_u}$ and $\eta_j$. This will be organized by the location of each column relative to $\eta_j$ and $\eta_{t_u}$, and the height of that column.\\
\begin{center}\begin{tabular}{|c|c|c|c|}
     \hline
     & Counted in $\lj$ and $\ltm$ & Counted only in $\ltm$ & Counted in neither   \\ \hline
     Left of both& $> \eta_{j}$ &  $\eta_{t_u}+1, \dotsc, \eta_j $ & $\leq \eta_{t_u}$  \\ \hline
     Between $\eta_j$ and $\eta_{t_u}$ & $\geq \eta_j$ & $\eta_{t_u}+1, \dotsc, \eta_j-1$& $\leq \eta_{t_u}$\\ \hline 
     Right of both & $\geq \eta_j$ & $\eta_{t_u} , \dotsc, \eta_{j}-1$& $< \eta_{t_u}$ \\ \hline 
\end{tabular}
\end{center}

The final column that is counted and not in the chart is column $\eta_j$, which is counted in $\ltm$. We can compare the above to the arm length $a_\mutildeeta(j,\eta_{t_u}+1)$,
\begin{align*} a_\mutildeeta(j,\eta_{t_u}+1) =  &\#\{1 \leq i < j+n-k \, \mid \, \eta_{t_u}+1 \leq \mutildeeta_j \leq \eta_j  \}| +\\
& \# \{ j+n-k <  i \leq n \, \mid \, \eta_{t_u} \leq \mutildeeta_j \leq \eta_j-1 \}. \end{align*}
Notice that the same columns that are counted in this arm are the ones which contribute only to $\ltm$. While on the surface there seems to be a discrepancy between the right arm of the box and the columns between $\eta_j$ and $\eta_{t_u}$, where $\ltm$ does not count columns of height $\eta_{t_u}$, recall that such a column existing would contradict maximality of the set $I_1$. Additionally, there is one extra box counted in the arm length, which is again the column $\eta_{t_j}$. From this comparison, we obtain the formula,
\[ \ltm = \lj + a_\mutildeeta(j,\eta_{t_u}+1).\]
Further, combining this with the difference in column heights, we get a formula relating the eigenvalues and statistics of $\square = (j,\eta_{t_u}+1)$,
\[ \overline{\eta_j} = \overline{\eta_{t_u}} \cdot q^{\leg{\mutildeeta}+1} t^\armtwo{\mutildeeta}.\]
Using this, we can now simplify the rational function,
\begin{align*}
    \frac{\overline{\eta}_{t_u}-t \overline{\eta}_j}{\overline{\eta}_{t_u}-\overline{\eta}_j} &= \frac{\overline{\eta}_{t_u}-\overline{\eta_{t_u}} \cdot q^{\leg{\mutildeeta}+1} t^{\armtwo{\mutildeeta}+1}}{\overline{\eta}_{t_u}-\overline{\eta_{t_u}} \cdot q^{\leg{\mutildeeta}+1} t^\armtwo{\mutildeeta}}\\
    &= \frac{1-q^{\leg{\mutildeeta}+1} t^{\armtwo{\mutildeeta}+1}}{1-q^{\leg{\mutildeeta}+1} t^{\armtwo{\mutildeeta}}}.
\end{align*}

To continue with the cancellation, recall that $\square = (j,\eta_{t_u}+1) \in \group{3}'_j$. This box has one more arm, namely from the column $\eta_{t_u}$, than the corresponding box $(j,\eta_{t_u}+1) \in \group{3}_j$. This means the contribution of these boxes to $\jthree$ is
\[  \frac{1-q^{\leg{\lambdatildegamma}+1} t^{\armtwo{\lambdatildegamma}+1}}{1-q^{\leg{\mutildeeta}+1} t^{\armtwo{\mutildeeta}+1}} =\frac{1-q^{\leg{\mutildeeta}+1} t^{\armtwo{\mutildeeta}}}{1-q^{\leg{\mutildeeta}+1} t^{\armtwo{\mutildeeta}+1}}. \]
This cancels exactly with the eigenvalue rational function for column $\eta_j$. We can go one step farther and use the contribution of the same box to $\Estarthree$,
\begin{align*}
    \frac{1-q^{-(\leg{\lambdatildegamma}+1)}t^{-\armtwo{\lambdatildegamma}}}{1-q^{-(\leg{\mutildeeta}+1)}t^{-\armtwo{\mutildeeta}}} &= \frac{1-q^{-(\leg{\mutildeeta}+1)}t^{-(\armtwo{(\mutildeeta}-1)}}{1-q^{-(\leg{\mutildeeta}+1)}t^{-\armtwo{\mutildeeta}}}\\
    &=\frac{\overline{\eta_j}-\overline{\eta_j} \cdot q^{-(\leg{\mutildeeta}+1)}t^{-(\armtwo{(\mutildeeta}-1)}}{\overline{\eta_j}-\overline{\eta_j} \cdot q^{-(\leg{\mutildeeta}+1)}t^{-\armtwo{\mutildeeta}}}\\
    &= \frac{t\overline{\eta}_{t_u} - \overline{\eta}_j}{\overline{\eta}_{t_u} - \overline{\eta}_j.}
\end{align*}

\underline{\textbf{Case 2:}} The approach here will be similar to the first case, but now instead of just considering a single column $\eta_j$ with $\eta_j < \eta_{t_u}$, we consider all columns strictly between $\eta_{t_{u-1}}$ and $\eta_{t_u}$ of a given height that is lower than $\eta_{t_u}$. The product of the eigenvalue rational functions for all columns of height $s$ will cancel together with the contribution of $\square = (t_u+n-k,s+1)$ to $\Estarthree$, and the contribution of the same square to $\jthree$ will be equal to the corresponding products of eigenvalues on the right side of \cref{pI1conversion1}. Combined with case 1, this accounts for all of $\jthree$ and $\Estarthree$, as well as all pairs of $t_u$ and $j$ in the product of the eigenvalue rational functions.

Let $s \in \{0,1, \dotsc, \eta_{t_u}-1\}$, and suppose the columns of height $s$ between $\eta_{t_{u-1}}$ and $\eta_{t_u}$ of height $s$ are $\eta_{a_1}, \dotsc, \eta_{a_{g_s}}$, so $g_s = \# \{t_{u-1} < i < t_{u-1} \, \mid \, \eta_i = s \}$. Let $\square = (t_u+n-k,r+1)$. Suppose $\eta_j\in \{ \eta_{a_1}, \dotsc, \eta_{a_{g_s}}\}$ is one of these columns. Like in case 1, comparing column heights, we have $\eta_{t_u} = \eta_j + \leg{\mutildeeta}+1$. There is again a comparison of $\lj$ and $\ltm$: 
\begin{center}\begin{tabular}{|c|c|c|c|}
     \hline
     & Counted in $\lj$ and $\ltm$ & Counted only in $\lj$ & Counted in neither   \\ \hline
     Left of both& $> \eta_{t_u}$ & $s+1, \dotsc, \eta_{t_u}$ & $\leq s$ \\ \hline
     Between $\eta_j$ and $\eta_{t_u}$ & $>\eta_{t_u}$ & $s, \dotsc, \eta_{t_u}$ & $<s$\\ \hline 
     Right of both & $\geq \eta_{t_u}$ & $s, \dotsc, \eta_{t_u}-1$ & $<s$ \\ \hline 
\end{tabular}
\end{center}

There are two differences in boxes counted only in $\lj$ and $\armtwo{\mutildeeta}$:
\begin{enumerate}
    \item $\lj$ counts $\eta_{t_u}$, but there is no corresponding column which contributes to $\armtwo{\mutildeeta}$.
    \item $\lj$ counts boxes between $\eta_j$ and $\eta_{t_u}$ of height $s$, while those are not counted in $\armtwo{\mutildeeta}$.
\end{enumerate}

With these differences accounted for, we get the equation,
\[ \lj = \ltm + (\armtwo{\mutildeeta}+1) + \# \{ j < i < t_u \, \mid \, \eta_i = s\}.\]

We can use this to extract a series of formulas for the eigenvalues, moving from the rightmost column of height $s$ to the left,
\begin{align*}
    \overline{\eta_{a_{g_s}}} &= \overline{\eta_{t_u}} \cdot q^{-(\leg{\mutildeeta}+1)}t^{-(\armtwo{\mutildeeta}+1)}\\
    \overline{\eta_{a_{(g_s-1})}}&= \overline{\eta_{t_u}}\cdot q^{-(\leg{\mutildeeta}+1)}t^{-((\armtwo{\mutildeeta}+1)+1)}\\
    &\vdots \\
    \overline{\eta_{a_1}}&= \overline{\eta_{t_u}}\cdot q^{-(\leg{\mutildeeta}+1)}t^{-((\armtwo{\mutildeeta}+1)+g_s)}
\end{align*}

Taken together, the product of the respective eigenvalue rational functions before substitution is
\[ \frac{\overline{\eta}_{t_u} - t \overline{\eta}_{a_{g_s}}}{\overline{\eta}_{t_u} -  \overline{\eta}_{a_{g_s}}}\cdot \frac{\overline{\eta}_{t_u} - t \overline{\eta}_{a_{g_s-1}}}{\overline{\eta}_{t_u} -  \overline{\eta}_{a_{g_s-1}}} \dotsm \frac{\overline{\eta}_{t_u} - t \overline{\eta}_{a_{1}}}{\overline{\eta}_{t_u} -  \overline{\eta}_{a_{1}}}\]
and after substitution (and suppressing the subscripts, since all diagram statistics are relative to $\mutildeeta$), this becomes
\[ \frac{1-q^{-(\ell(\square)+1)}t^{-a(\square)}}{1-q^{-(\ell(\square)+1)}t^{-(a(\square)+1)}} \cdot \frac{1-q^{-(\ell(\square)+1)}t^{-(a(\square)+1)}}{1-q^{-(\ell(\square)+1)}t^{-(a(\square)+2)}} \dotsm \frac{1-q^{-(\leg{\mutildeeta}+1)}t^{-(a(\square)+g_s-1)}}{1-q^{-(\ell(\square)+1))}t^{-(a(\square)+g_s)}},\]
which simplifies so that
\[  \frac{\overline{\eta}_{t_u} - t \overline{\eta}_{a_{g_s}}}{\overline{\eta}_{t_u} -  \overline{\eta}_{a_{g_s}}}\cdot \frac{\overline{\eta}_{t_u} - t \overline{\eta}_{a_{g_s-1}}}{\overline{\eta}_{t_u} -  \overline{\eta}_{a_{g_s-1}}} \dotsm \frac{\overline{\eta}_{t_u} - t \overline{\eta}_{a_{1}}}{\overline{\eta}_{t_u} -  \overline{\eta}_{a_{1}}} = \frac{1-q^{-(\ell(\square)+1)}t^{-a(\square)}}{1-q^{-(\ell(\square)+1)}t^{-(a(\square)+g_s)}}\]
To complete the cancellations, notice that the columns which are counted in $a_\lambdatildegamma(t_{u-1},s+1)$ but not $a_\mutildeeta(t_u,s+1)$ are exactly those between $\eta_{t_{u-1}}$ and $\eta_{t_u}$ of height $s$, which means the contribution of the boxes $\square' =(t_{u-1},s+1)$ and $\square =(t_{u},s+1)$ to $\Estarthree$ is
\begin{align*}
    \frac{1-q^{-(\ell_{\lambdatildegamma+1)}(\square')}t^{-a_\lambdatildegamma(\square')}}{1-q^{-(\ell_{\mutildeeta+1)}(\square)}t^{-a_\mutildeeta(\square)}} &= \frac{1-q^{-(\ell_{\mutildeeta+1)}(\square)}t^{-(a_\mutildeeta(\square)+g_s)}}{1-q^{-(\ell_{\mutildeeta+1)}(\square)}t^{-a_\mutildeeta(\square)}}
\end{align*}

So the product of rational functions cancels exactly with the box's contribution to $\Estarthree$. Looking at the corresponding contribution of the box to $\jthree$, we have
\begin{align*}
    \frac{1-q^{\ell_\lambdatildegamma(\square')+1}t^{a_\lambdatildegamma(\square')+1}}{1-q^{\ell_\mutildeeta(\square)+1}t^{a_\mutildeeta(\square)+1}} &= \frac{1-q^{\ell(\square)+1}t^{a(\square)+1+g_s}}{1-q^{\ell(\square)+1}t^{a(\square)+1}} \\
    &= \frac{1-q^{\ell(\square)+1}t^{a(\square)+1+g_s}}{1-q^{\ell(\square)+1}t^{a(\square)+1+(g_s-1)}} \dotsm \frac{1-q^{\ell(\square)+1}t^{(a(\square)+1)+1}}{1-q^{\ell(\square)+1}t^{a(\square)+1}}\\
    &= \frac{\overline{\eta_{a_1}}-\overline{\eta_{a_1}} \cdot q^{\ell(\square)+1}t^{a(\square)+1+g_s}}{\overline{\eta_{a_1}}-\overline{\eta_{a_1}} \cdot q^{\ell(\square)+1}t^{a(\square)+(g_s-1)}} \dotsm \frac{\overline{\eta_{a_{g_s}}}-\overline{\eta_{a_{g_s}}} \cdot q^{\ell(\square)+1}t^{(a(\square)+1)+1}}{\overline{\eta_{a_{g_s}}}-\overline{\eta_{a_{g_s}}}\cdot q^{\ell(\square)+1}t^{a(\square)+1}}\\
    &= \frac{\overline{\eta_{a_1}}-t \cdot \overline{\eta_{t_u}}}{\overline{\eta_{a_1}}-\overline{\eta_{t_u}}} \dotsm \frac{\overline{\eta_{a_{g_s}}}-t \cdot \overline{\eta_{t_u}}}{\overline{\eta_{a_{g_s}}}-\overline{\eta_{t_u}}}.
\end{align*}

Altogether, between the two cases, every term from $\jthree$ and $\Estarthree$, and every eigenvalue rational function in the left side of \cref{pI1conversion1} has been either cancelled or converted to the form in the right side of the equation.\qedhere

\end{proof}

\begin{prop}\label{prop2conversion}
\begin{equation} \begin{split} \label{PI1conversion2}\jfour \cdot \jfive \cdot  \Estarfour &\cdot \Estarfive \cdot \prod_{j=t_r+1}^k \frac{q^{-1} \overline{\mutilde_h} -t\overline{\eta}_j}{q^{-1} \overline{\mutilde_h} -\overline{\eta}_j} = \prod_{j=t_r+1}^k \frac{(q^{-1}t) \overline{\mutilde_h} -\overline{\eta}_j}{(q^{-1}) \overline{\mutilde_h} -\overline{\eta}_j}  \\ &\cdot \prod\limits_{\substack{\square \in \group{5}'_E \\ \square = (h,i) \\ i\neq \mutilde_h}} \frac{1-q^{-\leg{\mutildeeta}}t^{-\armone{\mutildeeta}}}{1-q^{-(\leg{\mutildeeta}+1)}t^{-\armtwo{\mutildeeta}}} \cdot \left(\frac{1}{1-t}\right) \cdot \left(\frac{1}{1-q^{-1}t^{-a_\mutildeeta(h,\mutilde_h)}}\right)\end{split}\end{equation}
\end{prop}

\begin{proof} See \cref{SecondSimplification}. The proof uses the same techniques as \cref{prop1conversion}. \end{proof}

\subsection{Conversion of $p_2$ terms.} 

\begin{prop}\label{prop3conversion}
The eigenvalue function from $p_2$ can be rewritten,
\[ \prod_{j=1}^{h-1} \frac{\overline{\mutilde}_h-t\overline{\mutilde}_j}{\overline{\mutilde}_h-\overline{\mutilde}_j} = \prod_{\substack{\square = (h,s) \in \group{5}'_j \\ s\neq \mutilde_h}} \frac{1-q^{-\leg{\mutildeeta}}t^{-\armone{\mutildeeta}+g_s}}{1-q^{-\leg{\mutildeeta}}t^{-\armone{\mutildeeta}}} \cdot \prod_{d_{(\mutilde_h+1)}(\mutilde)} \frac{1-q^{\leg{\mutildeeta}+1}t^{\armtwo{\mutildeeta}+1}}{1-q^{\leg{\mutildeeta}+1}t^{\armtwo{\mutildeeta}}} \]
where $g_s = \# \{ i \, \mid \, \mutilde_i = s\}$
\end{prop}

\begin{proof}
Detail omitted, the process is identical to the previous 2 proofs.
\end{proof}

Next, we can cancel most of the $\group{5}'_j$ product from \cref{prop2conversion} and \cref{prop3conversion}.

\begin{prop}\label{group5cance3llation}
\begin{align*} &\prod\limits_{\substack{\square=(h,i) \in \group{5}'_E \\  i\neq \mutilde_h}} \frac{  1-q^{-\leg{\mutildeeta}}t^{-\armone{\mutildeeta}}}{ 1-q^{-(\leg{\mutildeeta}+1)}t^{-\armtwo{\mutildeeta}}} \cdot \prod_{\substack{\square = (h,s) \in \group{5}'_j \\ s\neq \mutilde_h}} \frac{1-q^{-\leg{\mutildeeta}}t^{-\armone{\mutildeeta}+g_s}}{1-q^{-\leg{\mutildeeta}}t^{-\armone{\mutildeeta}}} \\ &\hspace*{2cm}\cdot \left(\frac{1}{1-q^{-1}t^{-a_\mutildeeta(h,\mutilde_h)}}\right)  = \frac{1}{1-q^{-(\ell_{\mutildeeta}(h,1)+1)}t^{-a_\mutildeeta(h,1)}}\end{align*}
\end{prop}
\begin{proof}
The left product numerator and right product denominator cancel directly. The rest comes immediately from the observation that for a box $\square = (h,i)$ in the left product and $\square' = (h,i+1)$ in the right product,
\begin{align*}
    &\armone{\mutildeeta} - g_s = a_\mutildeeta(\square')\\
    & \leg{\mutildeeta} = \ell_\mutildeeta(\square')+1.
\end{align*}
So the contributing terms for each pair of boxes $\square$ and $\square'$ in the left and right products cancel, and what's left is the contribution for the box $(h,1)$ in the left product.
\end{proof}

There is also a small cancellation with one of the terms left over from \cref{prop2conversion},
\begin{prop}
\[\frac{\overline{\mutilde}_h-t^{-n+1}}{1-q^{-(\ell_{\mutildeeta}(h,1)+1)}t^{-a_\mutildeeta(h,1)}} = q^{\mutilde_h}t^{-\ell'_\mutildeeta(h)}\]
\end{prop}
\begin{proof}
Again, the approach here is the same as the other eigenvalue conversions. The only important thing to note is this makes use of the identity,
\[ \ell'_\mutildeeta(h) = n-a_\mutildeeta(h,1) - 1,\]
which comes from comparing the columns counted in $a_\mutildeeta(h,1)$ and in $\ell'_\mutildeeta(h)$ in the same way we have done up to this point.
\end{proof}

We have one more major cancellation to compute.

\begin{prop}
\[ \jone \cdot \frac{1-t^{m+1}}{1-t} \cdot \prod_{d_{(\mutilde_h+1)}(\mutilde)} \frac{1-q^{\leg{\mutildeeta}+1}t^{\armtwo{\mutildeeta}+1}}{1-q^{\leg{\mutildeeta}+1}t^{\armtwo{\mutildeeta}}} = 1\]
\end{prop}
\begin{proof}
Recall the formula for $\jone$,
\[ \jone = \prod_{\square \in \group{1}_j'} \frac{1-q^{\leg{\mumineta}}t^{\armone{\mumineta}}}{1-q^{\leg{\mumineta}}t^{\armone{\mumineta}+1}},\]
where $\group{1}'_j$ is $d_{\mutilde_h}(\mutilde)$ but excluding the newly-added box. The cancellation occurs by breaking this product into two pieces. The first will be those boxes in $\group{1}'_j$ which have no boxes above them (i.e., boxes in columns of height $\mutilde_h$), and the second is those which have boxes above them.

If we consider first the boxes which have no boxes above them, notice each of them has 0 legs, and the number of modified arms in $\armone{\mumineta}$ are $1, 2, \dotsc, m$, where $m = \#\{ 1\leq i \leq n-k \mid \, \mutilde_i = \mutilde_1  \}-1$. The contribution of all those boxes to $\jone$ is then
\[ \frac{1-t}{1-t^2}\dotsm \frac{1-t^m}{1-t^{m+1}}. \]
So those contributions cancel with $\frac{1-t^{m+1}}{1-t}$. For the remaining boxes, we compare the contribution of one of the remaining boxes $\square = (i,j)$ in the $\group{1}'_j$ product with the corresponding box above it $\square'$ in the $d_{(\mutilde_h+1)}(\mutilde)$ product. The diagram $\mumineta$ has boxes of height $\mutilde_h$ to the left of $i$, which means they are counted in the modified arm of $\square$, while $\mutildeeta$ has boxes of height $\mutilde_h$ to the right of $\square'$, which are counted in the arm of $\square'$. All other columns are counted identically in $\armone{\mumineta}$ and $a_\mutildeeta(\square')$. The powers of $q$ match as well, with the extra leg accounted for in the $+1$ on the $q$-powers of the $d_{(\mutilde_h+1)}(\mutilde)$ product. Everything cancels as desired.
\end{proof}

Considering all the cancellations up to this point, we state the final cancelled form of the expansion of $e_1[x_1,\dotsc,x_{n-k}] \JJ$.

\begin{thm} The coefficients in the expansion,

\[ e_1[x_1,\dotsc,x_{n-k}]\JJ = \sum_{\mueta \in \MMM}  \CC \mathcal{J}_{\mueta}, \]

are given by the formula,

\begin{equation}\label{simplifiedpieri} \CC = \prod_{ \square \in d_{(\lambdatilde_{n-k}+1)}(\lambda^-)}\frac{t-q^{\ell_{\lambdamingamma}(\square)+1}t^{a_{\lambdamingamma}(\square)+1}}{1-q^{\ell_{\lambdamingamma}(\square)+1}t^{a_{\lambdamingamma}(\square)+1}} \cdot \jtwo \cdot p'_{I_1} \cdot \left(\frac{1}{1-t}\right) \cdot (q^{-\mutilde_h+1}) \cdot \overline{\lambdatildegamma}_{n-k}\end{equation}
where $p'_{I_1}$ is the modified $p_{I_1}$,
\[ p'_{I_1} := \left(\frac{(t-1)q^{-1} \overline{\mutilde}_{\hIone}}{q^{-1} \overline{\mutilde}_{\hIone}-\overline{\eta}_{t_r}}\right)  \prod_{u=1}^{r-1} \left(\frac{(t-1) \overline{\eta}_{t_{u+1}}}{ \overline{\eta}_{t_{u+1}}-\overline{\eta}_{t_{u}}} \right)  \prod_{j=t_r+1}^{k} \left( \frac{(q^{-1}t) \overline{\mutilde}_{\hIone} -  \overline{\eta}_j}{q^{-1} \overline{\mutilde}_{\hIone} -  \overline{\eta}_j}\right)\prod_{u=1}^{r} \prod_{j=t_{u-1}+1}^{t_{u}-1} \left( \frac{t\overline{\eta}_{t_u} -  \overline{\eta}_j}{ \overline{\eta}_{t_{u}} -  \overline{\eta}_j} \right),\]

with $0 = t_0 < t_1 < \dotsm < t_r < t_{r+1}  = k+1$, and we recall
\[\jtwo  = \prod\limits_{\square \in \group{2}_j}\frac{1-q^{\leg{\lambdamingamma}}t^{\armone{\lambdamingamma}+1} }{ 1-q^{\leg{\lambdamingamma}+1}t^{\armone{\lambdamingamma}+1+m_{j-1}(\eta)}},\]
where $\group{2}_j$ is the column that is moved from the symmetric part of the diagram to the nonsymmetric part.
\end{thm}

We end with the following example of a coefficient $\CC$ using \cref{Jpieri} and  \cref{simplifiedpieri}.

\begin{ex}
    We will compute $\mathcal{A}^{(3,1,1 \, \mid \, 1,0,1)}_{(3,2,1\, \mid \, 1,0,1)}$ in two ways. First, note $I_1 = \{1,3\}$, and
    \[ \lambdatildegamma = (1,3,1\, \mid \, 1,0,1), \qquad \mutildeeta = (1,3,2 \, \mid \, 1,0,1).\]
    The eigenvalues are
    \[ \overline{\lambdatildegamma} = (qt^{-4}, q^3, qt^{-3} \, \mid \, qt^{-2}, t^{-5}, qt^{-1}), \qquad \overline{\mutildeeta} = (qt^{-4}, q^3, \underline{q^2t^{-1}} \, \mid \, \underline{qt^{-3}}, t^{-5}, \underline{qt^{-2}})\]
    The underlined terms in $\overline{\mutildeeta}$ are $\overline{\mutilde}_h$, $\overline{\eta}_{t_1}$, and $\overline{\eta}_{t_2}=\overline{\eta}_{t_r}$ respectively. The following diagrams are used to compute $j_{\lambdagamma}$ and $j_\mueta$ respectively, filled with (leg,arm), and arms are modified in the symmetric columns and standard in the nonsymmetric columns:

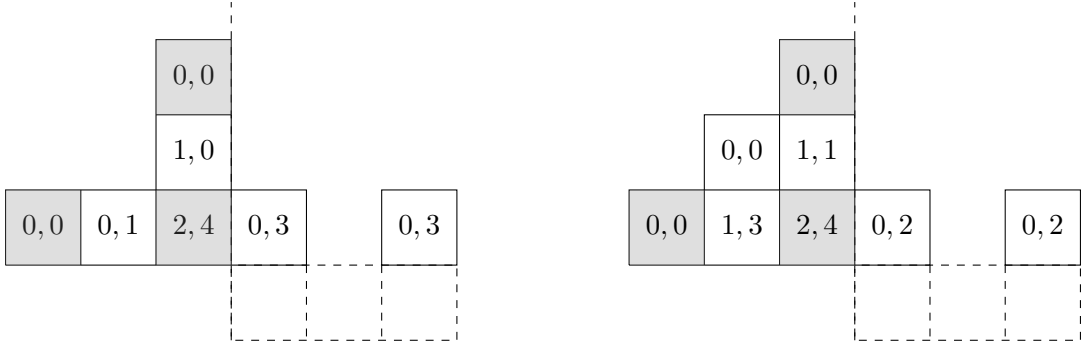
\begin{center}
    \begin{tikzpicture}[scale=1,fill opacity=0.25,every node/.style={scale=1}]

\draw (0,0) -- (0,1) -- (4,1) -- (4,0) -- (0,0);
\draw (1,0) -- (1,1);
\draw (2,0) -- (2,3) -- (3,3) -- (3,0);
\draw (2,2) -- (3,2);
\draw (5,0) -- (5,1) -- (6,1) -- (6,0) -- (5,0);

\draw[dashed] (3,0) -- (3,-1) -- (6,-1) -- (6,0) -- (3,0);
\draw[dashed] (4,0) -- (4,-1);
\draw[dashed] (5,0) -- (5,-1);

\draw[dashed] (3,-1) -- (3,3.5);

\node[fill opacity = 1] at (.5,.5) {\small $ 0,0$};
\node[fill opacity = 1] at (1.5,.5) {\small $ 0,1$};
\node[fill opacity = 1] at (2.5,.5) {\small $ 2,4$};
\node[fill opacity = 1] at (2.5,1.5) {\small $ 1,0$};
\node[fill opacity = 1] at (2.5,2.5) {\small $ 0,0$};
\node[fill opacity = 1] at (3.5,.5) {\small $ 0,3$};
\node[fill opacity = 1] at (5.5,.5) {\small $ 0,3$};
\fill[gray] (0,0) rectangle (1,1);
\fill[gray] (2,2) rectangle (3,3);
\fill[gray] (2,0) rectangle (3,1);
\end{tikzpicture} \hspace*{2cm}
 \begin{tikzpicture}[scale=1,fill opacity=0.25,every node/.style={scale=1}]

\draw (0,0) -- (0,1) -- (4,1) -- (4,0) -- (0,0);
\draw (1,0) -- (1,2) -- (2,2);
\draw (2,0) -- (2,3) -- (3,3) -- (3,0);
\draw (2,2) -- (3,2);
\draw (5,0) -- (5,1) -- (6,1) -- (6,0) -- (5,0);
\draw[dashed] (3,0) -- (3,-1) -- (6,-1) -- (6,0) -- (3,0);
\draw[dashed] (4,0) -- (4,-1);
\draw[dashed] (5,0) -- (5,-1);

\draw[dashed] (3,-1) -- (3,3.5);

\fill[gray] (0,0) rectangle (1,1);
\fill[gray] (2,2) rectangle (3,3);
\fill[gray] (2,0) rectangle (3,1);

\node[fill opacity = 1] at (.5,.5) {\small $ 0,0$};
\node[fill opacity = 1] at (1.5,.5) {\small $ 1,3$};
\node[fill opacity = 1] at (2.5,.5) {\small $ 2,4$};
\node[fill opacity = 1] at (2.5,1.5) {\small $ 1,1$};
\node[fill opacity = 1] at (2.5,2.5) {\small $ 0,0$};
\node[fill opacity = 1] at (1.5, 1.5) {\small $ 0,0$};
\node[fill opacity = 1] at (3.5,.5) {\small $ 0,2$};
\node[fill opacity = 1] at (5.5,.5) {\small $ 0,2$};
\end{tikzpicture}

\end{center}

Cancelling the gray box terms which are the same in the numerator and denominator, we get
\[ \frac{j_\lambdagamma}{j_\mueta} = \frac{(1-qt)(1-t^2)(1-qt^4)(1-qt^4)}{(1-t)(1-qt^2)(1-qt^4)(1-qt^3)(1-qt^3)}.\]

Next we build the diagrams for $\lambdatildegamma$ and $\mutildeeta$ respectively, filled with (leg,arm) with all standard (non-modified) arms:

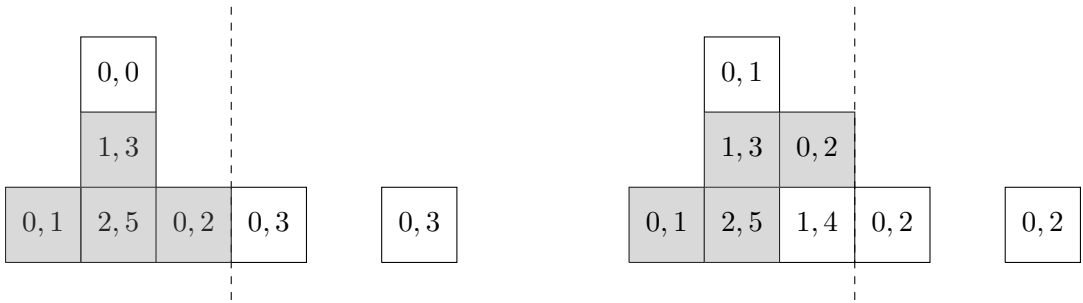
\begin{center}
    \begin{tikzpicture}[scale=1,fill opacity=.3,every node/.style={scale=1}]

\draw (0,0) -- (0,1) -- (4,1) -- (4,0) -- (0,0);
\draw (1,0) -- (1,1);
\draw (2,0) -- (2,3) -- (1,3) -- (1,0);
\draw (2,2) -- (1,2);
\draw (5,0) -- (5,1) -- (6,1) -- (6,0) -- (5,0);
\draw (3,0) -- (3,1);

\draw[dashed] (3,-0.5) -- (3,3.5);

\node[fill opacity = 1] at (.5,.5) {\small $ 0,1$};
\node[fill opacity = 1] at (1.5,.5) {\small $ 2,5$};
\node[fill opacity = 1] at (2.5,.5) {\small $ 0,2$};
\node[fill opacity = 1] at (1.5,1.5) {\small $ 1,3$};
\node[fill opacity = 1] at (1.5,2.5) {\small $ 0,0$};
\node[fill opacity = 1] at (3.5,.5) {\small $ 0,3$};
\node[fill opacity = 1] at (5.5,.5) {\small $ 0,3$};
\fill[gray] (0,0) rectangle (1,1);
\fill[gray] (1,0) rectangle (2,1);
\fill[gray] (1,1) rectangle (2,2);
\fill[gray] (2,0) rectangle (3,1);
\end{tikzpicture} \hspace*{2cm} \begin{tikzpicture}[scale=1,fill opacity=.3,every node/.style={scale=1}]

\draw (0,0) -- (0,1) -- (4,1) -- (4,0) -- (0,0);
\draw (1,0) -- (1,1);
\draw (2,0) -- (2,3) -- (1,3) -- (1,0);
\draw (2,2) -- (1,2);
\draw (5,0) -- (5,1) -- (6,1) -- (6,0) -- (5,0);
\draw (2,2) -- (3,2) -- (3,1);
\draw (3,0) -- (3,1);
\fill[gray] (0,0) rectangle (1,1);
\fill[gray] (1,0) rectangle (2,1);
\fill[gray] (1,1) rectangle (2,2);
\fill[gray] (2,1) rectangle (3,2);

\draw[dashed] (3,-0.5) -- (3,3.5);

\node[fill opacity = 1] at (.5,.5) {\small $ 0,1$};
\node[fill opacity = 1] at (1.5,.5) {\small $ 2,5$};
\node[fill opacity = 1] at (2.5,.5) {\small $ 1,4$};
\node[fill opacity = 1] at (2.5, 1.5) {\small $0,2$};
\node[fill opacity = 1] at (1.5,1.5) {\small $ 1,3$};
\node[fill opacity = 1] at (1.5,2.5) {\small $ 0,1$};
\node[fill opacity = 1] at (3.5,.5) {\small $ 0,2$};
\node[fill opacity = 1] at (5.5,.5) {\small $ 0,2$};
\end{tikzpicture}

\end{center}

Again cancelling boxes with the same arms/legs in gray, we find,
\begin{align*} \frac{\Estar{\lambdatildegamma}\overline{\lambdatildegamma}}{\Estar{\mutildeeta}\overline{\mutildeeta}} &= \frac{(qt^{-1})^1}{(q^2t^{-1})^2} \frac{(1-q^{-1}) (1-q^{-1}t^{-3})(1-q^{-1}t^{-3})}{(1-q^{-1}t^{-1})(1-q^{-1}t^{-2})(1-q^{-1}t^{-2})(1-q^{-2}t^{-4})}\\
&= q^{-3} \left( \frac{1-q}{1-qt}\right) \left( \frac{1-qt^3}{1-qt^2}\right) \left( \frac{1-qt^3}{1-qt^2} \right) \frac{1}{1-q^{-2}t^{-4}}\end{align*}

Then we compute $p_2$ and $p_{I_1}$,
\begin{align*}
    p_2 &= \frac{1-t}{1-t} (q^2t^{-1} - t^{-5}) \left( \frac{q^2t^{-1} - qt^{-3}}{q^2t^{-1} - qt^{-4}} \right) \left( \frac{q^2t^{-1}-q^3t}{q^2t^{-1} - q^3}\right)\\
    &= q^2t^{-1} (1-q^{-2}t^{-4})\cdot t \cdot \left( \frac{1-qt^2}{1-qt^3} \right) \left( \frac{1-qt^2}{1-qt} \right)\\
    p_{I_1} &= \left(\frac{(t-1)qt^{-1}}{qt^{-1} - qt^{-2}} \right) \left(\frac{(t-1) qt^{-2}}{qt^{-2}-qt^{-3}} \right) \left( \frac{ qt^{-2}-t^{-4}}{qt^{-2} - t^{-5}}\right)\\
    &= t^3 \frac{1-qt^2}{1-qt^3}
\end{align*}

Finally, we find $f_{\lambdatilde, \lambdagamma} / f_{\mutilde, \mueta} = \frac{(1-q^2t^2)(1-qt)}{(1-q^2t^3)(1-q)}$. Combining all these terms as well as $\overline{\lambdatildegamma}_{n-k} = qt^{-3}$ as per \cref{Jpieri}, and cancelling as far as possible, we obtain:

\[ \mathcal{A}^{(3,1,1 \, \mid \, 1,0,1)}_{(3,2,1\, \mid \, 1,0,1)} = \frac{(1-q^2t^2)(1-t^2)(1-qt^4)}{(1-q^2t^3)(1-t)(1-qt^3)(1-qt^3)}\]

The computation using \cref{simplifiedpieri} is substantially easier. Since $\lambdatilde_{n-k} = 1$, the product over $d_{(\lambdatilde_{n-k}+1)}(\lambdatilde) = d_2(\lambdatilde)$ only has one box, and
\[ \prod_{\square \in d_{(\lambdatilde_{n-k}+1)}(\lambda^-)}\frac{t-q^{\ell_{\lambdamingamma}(\square)+1}t^{a_{\lambdamingamma}(\square)+1}}{1-q^{\ell_{\lambdamingamma}(\square)+1}t^{a_{\lambdamingamma}(\square)+1}} =  \frac{t-q^2t^3}{1-q^2t^3} = t \frac{1-q^2t^2}{1-q^2t^3}. \]
Next, $\jtwo$ also only has a single box, namely the box $(2,1)$ in $\lambdamingamma$, and
\[ \jtwo =\prod\limits_{\square \in \group{2}_j}\frac{1-q^{\leg{\lambdamingamma}}t^{\armone{\lambdamingamma}+1} }{ 1-q^{\leg{\lambdamingamma}+1}t^{\armone{\lambdamingamma}+1+m_{j-1}(\eta)}} =  \frac{1-t^2}{1-qt^3}.\]
That leaves $p'_{I_1}$, which we compute:
\begin{align*}
    p'_{I_1} &= \left(\frac{(t-1)qt^{-1}}{qt^{-1}-qt^{-2}} \right) \left( \frac{(t-1)qt^{-2}}{qt^{-2}-qt^{-3}}\right) \left( \frac{qt^{-1}-t^{-5}}{qt^{-2}-t^{-5}}\right)\\
    &=t^2 \frac{1-qt^4}{1-qt^3}
\end{align*}
Now, including again $\overline{\lambdatilde}_{n-k} = qt^{-3}$, the extra $\frac{1}{1-t}$, and $q^{-\mutilde_h + 1} = q^{-1}$, we use $\cref{simplifiedpieri}$,
\[ \mathcal{A}^{(3,1,1 \, \mid \, 1,0,1)}_{(3,2,1\, \mid \, 1,0,1)} = \frac{(1-qt^4)(1-q^2t^2)(1-t^2)}{(1-qt^3)(1-q^2t^3)(1-qt^3)(1-t)}.\]

\end{ex}

	\bibliography{biblio}
	\bibliographystyle{abbrv}

\appendix

\section{Proof of \cref{decreasingT}} 
\cref{decreasingT}: If $s_i(\mu)<\mu$, and $u=(i,\mu_{i}+1)$, then
\[ T_iE_\mu = \frac{q^{\ell(u)+1}t^{a(u)}(1-t)}{1-q^{\ell(u)+1}t^{a(u)}}E_{\mu} - \frac{(t-q^{\ell(u)+1}t^{a(u)})(1-q^{\ell(u)+1}t^{a(u)+1})}{(1-q^{\ell(u)+1}t^{a(u)})^2} E_{s_i(\mu)},\]
where all arms and legs are in terms of the diagram of $s_i(\mu)$.
\begin{proof}
Suppose $\mu_i < \mu_{i+1}$, so  $s_i(\mu)<\mu$. Using \cref{quadratic}, note that $T_i^{-1} = \frac{1}{t}(T_i+1-t)$. Then since $s_i(\mu)<s_i(s_i(\mu))$, we can apply \cref{E:T-on-E-up} with $s_i(\mu)$ to obtain,

\[ T_i E_{s_i(\mu)} = E_\mu + \frac{t-1}{1-q^{\ell(u)+1}t^{a(u)}} E_{s_i(\mu)},\]
where the arms and legs are with respect to the diagram of $s_i(\mu)$. Multiplying both sides on the left by $T_i^{-1}$ and simplifying, we get,
\begin{align*}
E_{s_i(\mu)} &= \frac{1}{t}(T_i+1-t)(E_\mu + \frac{t-1}{1-q^{\ell(u)+1}t^{a(u)}}E_{s_i(\mu)})\\
tE_{s_i(\mu)} &= T_i E_\mu + (1-t) E_\mu + \frac{t-1}{1-q^{\ell(u)+1}t^{a(u)}} \left(T_i E_{s_i(\mu)} + (1-t) E_{s_i(\mu)} \right)\\
tE_{s_i(\mu)} &= T_i E_\mu + (1-t) E_\mu + \frac{t-1}{1-q^{\ell(u)+1}t^{a(u)}} \left( E_\mu + \frac{t-1}{q^{\ell(u)+1}t^{a(u)}} E_{s_i(\mu)} + (1-t) E_{s_i(\mu)} \right)\\
T_i E_\mu &= \left(t-1+\frac{1-t}{1-q^{\ell(u)+1}t^{a(u)}} \right) E_\mu + \left(t-\frac{(t-1)^2}{1-q^{\ell(u)+1}t^{a(u)})^2} - \frac{(t-1)(1-t)}{1-q^{\ell(u)+1}t^{a(u)}} \right) E_{s_i(\mu)}\\
T_iE_\mu &= \frac{q^{\ell(u)+1}t^{a(u)}(1-t)}{1-q^{\ell(u)+1}t^{a(u)}}E_{\mu} - \frac{(t-q^{\ell(u)+1}t^{a(u)})(1-q^{\ell(u)+1}t^{a(u)+1})}{(1-q^{\ell(u)+1}t^{a(u)})^2} E_{s_i(\mu)}\qedhere
\end{align*}
\end{proof}

\section{Proof of \cref{nonzero}}

\textbf{\cref{nonzero}} If $I_1$ is maximal with respect to $\gammalambdatilde$, then $\pIoneevaluated \neq 0$.

\begin{proof}
We will prove the contrapositive, so assume $p_{I_1}\left( \overline{\etamutilde}\right) = 0$. This implies either \begin{enumerate} \item $\overline{\eta}_{t_u} = t \overline{\eta}_j \text{ for some } u\in [1,r] \text{ and } j\in [t_u+1,t_{u+1}-1]$., \text{ or } \item $ q^{-1} \overline{\mutilde}_{m+1} = t \overline{\eta}_j \text{ for some } j \in [1,t_1-1]$
 \end{enumerate}
 In the first case, if $\overline{\eta}_{t_u} = t\overline{\eta}_j$, the powers of $q$ being equal means columns $t_u$ and $j$ have the same height, so $\eta_{t_u} = \eta_j$. Let $I_1' := I_1 \cup \{j\}$, and let $\dI'$ be the action associated with $I_1'$. We will show that $d_I\gammanu = d_I'\gammanu$. In that context, we have $\gamma_j = \gamma_{t_{u+1}}$. The only positions affected by the addition of $j$ to $I_1$ are the positions $j$ and $t_1$. And in these positions,
 \[ (\dI \gammanu)_{t_u} = \gammanu_{t_{u+1}} \, \text{ and }\,  (\dI' \gammanu)_{t_u} = \gammanu_j\]
 \[ (\dI \gammanu)_{j} = \gammanu_j  \, \text{ and } \, (\dI' \gammanu)_{j} = \gammanu_{t_{u+1}}.\]
 All entries of $d_I\gammanu$ and $d_I'\gammanu$ agree, so $I_1$ is not maximal.
 
 In the second case, since $ q^{-1} \overline{\mutilde}_{m+1} = t \overline{\eta}_j \text{ for some } j \in [1,t_1-1]$, again using column height, we have $\eta_j = \mutilde_{m+1}-1$. In terms of $\gammanu$, because $\mutilde_{m+1} = \mutilde_{v_s+1}$, this implies $\gamma_j = \gamma_{t_1}$ after applying the definition of $\dI$. By the same argument as in the first case, $d_I\gammanu$ and $d_I'\gammanu$ are identical, hence $I_1$ is not maximal with respect to $\gammanu$.
 \end{proof}

\section{Proof of \cref{nonsymmetricformula}}

\subsection{Setup}
We return to Baratta conventions for the computation. We are therefore looking to find coefficients $\DBaratta$ for the expansion,
\begin{equation} z_j \PB{\gammalambdamin} =\sum_{\substack{|\etamu| = |\lambdagamma| + 1 \\ \mu_1 \geq \, \dotsm \, \geq \mu_{n-k}}}  \DBaratta \cdot \PB{\etamumin}, \end{equation}
where $1\leq j \leq k$. Then passing to the interpolation polynomial computation as before,
\begin{align*}
    z_j \PB{\gammalambdamin} &= \Psi^{-1} \left(\widetilde{Z_j}  \Xi_j^{-1} e^+ \Estar{\gammalambdamin} \right)\\
    &=\Psi^{-1} \left( \sum_{\nu \sim \lambda} \fnu H_j\dotsm H_{n-1} \Phi H_1 \dotsm H_{j-1} \Xi^{-1}_{j} \Estar{\gammanu}\right).
\end{align*}
Taking $\Psi$ of both sides, using \cref{monomialexpansion} for the left side, and expanding $\PB{\etamumin}$ into the nonsymmetric Macdonald basis, this becomes,
\begin{equation}\label{Pre-evaluation nonsymmetric} \sum_{\etamumin}  \, \sum_{\mu' \sim \mu} \DBaratta \cdot f_{\mu',\etamumin} \cdot \Estar{\etamuprime} = \sum_{\nu \sim \lambda} \fnu H_j\dotsm H_{n-1} \Phi H_1 \dotsm H_{j-1} \Xi^{-1}_{j} \Estar{\gammanu}. \end{equation}
\subsection{Evaluation}
In spite of some of the $H_i$ operators being shuffled, we will still evaluate at the same $\etamutilde$ used in \cref{Piericomputation}. With the vanishing property of the interpolation polynomials, this reduces the left side of \cref{Pre-evaluation nonsymmetric} to
\[
    \DBaratta \cdot f_{\mutilde,\etamumin} \cdot \Estar{\etamutilde}\left( \overline{\etamutilde}\right),\]
    and the right side becomes
\[\sum_{\nu \sim \lambda} \fnu \cdot \left( \overline{\gammanu}_j\right) \cdot H_j \dotsm H_{n-1} \Phi H_1 \dotsm H_{j-1}  \Estar{\gammanu} \left( \overline{\etamutilde}\right).
\]
Prior to simplifying the product of operators and sum over $\nu$, the formula for $\DBaratta$ is then,
\begin{equation}\label{prevanishing}
    \DBaratta = \sum_{\nu \sim \lambda} \frac{\fnu}{f_{\mutilde,\etamumin}} \cdot \left( \overline{\gammanu}_j\right) \cdot \frac{H_j \dotsm H_{n-1} \Phi H_1 \dotsm H_{j-1}  \Estar{\gammanu} \left( \overline{\etamutilde}\right)}{\Estar{\etamutilde}{\left(\overline{\etamutilde}\right)}}
\end{equation}
The question that remains is, for which $\nu$ and specific actions of each $H_i$ does the term in this sum not vanish?
\subsection{Nonvanishing Polynomials}
The goal in this section is to identify the terms in the expansion,
\[ H_j \dotsm H_{n-1} \Phi H_1 \dotsm H_{j-1} \Estar{\gammanu}(z),\]
which do not vanish when the function is evaluated at $\overline{\gammanu}$. Begin by defining three indexing sets,
\[ I'_1 \subseteq [1,j-1], \quad I'_2 \subseteq [k+1,n-1], \quad I'_3 \subseteq [j,k]. \]
Also, label the sets as follows:
\[ I'_1 = \{t_1, \dotsc, t_r\}, \quad I'_2 = \{v_1, \dotsc, v_s\}, \quad I'_3 = \{y_1, \dotsc, y_c\}.\]
Note that like before, it is possible for some of these sets to be empty. Expanding out the product,
\begin{align*} &  [a(z_j,z_{j+1})+b(z_j,z_{j+1}) s_j] \dotsm [a(z_{n-1},z_n)+b(z_{n-1},z_n) s_{n-1}]\,    (z_n-t^{-n+1})\Delta \, \\& \hspace*{1cm} \cdot [a(z_1, z_2)+b(z_1,z_2) s_1] \dotsm [a(z_{j-1},z_{j}) + b(z_{j-1},z_{j})s_{j-1}] \, \Estar{\gammanu}(z),
\end{align*}
term by term results in some choice of either $a(z_i,z_{i+1})$ or $b(z_i,z_{i+1}) s_i$ for each $1\leq i \leq n-1$. Then the sets $I'_1$, $I'_2$, and $I'_3$ are those indices for which we choose $a(z_i,z_{i+1})$. Thus a single term in the expansion will be of the form,

\[ \pstarone \cdot \pstartwo \cdot \pstarthree \cdot \Estar{\gammanu}\left(I'_3(I'_2(I'_1(z))) \right).\]
We label these $p^*$ to distinguish from $p_{I'_1}$ and $p_{2}$ already used, as we will be comparing these new functions to $p_{I_1}$ and $p_2$ after further computation. Also, note once again the use of $\pstartwo$ rather than $p^*_{I'_2}$, as we will see there is an explicit form for $I'_2$. Rather than identify how each set permutes $z$ separately, we combine it into a single action:

\[ I'_3(I'_2(I'_1(z)))_\ell := \left\{\begin{tabular}{l l} $ \left.\begin{array}{l l} z_\ell  &\text{ if } \ell\notin I'_1 \text{ and } 1\leq \ell \leq j-1 \\
z_{t_{\ell-1}} \hspace*{.5cm} &\text{ if } \ell=t_i \text{ and } i \neq 1\\
q^{-1} z_{v_s+1} &\text{ if } \ell=t_1 \\\end{array} \hspace{1cm} \right\} \, \ell \leq j$ \\
\underline{\hspace{8.5cm}}\\[2mm] $\left.\begin{array}{l l}
z_{\ell} \hspace*{1.15cm} &\text{ if } k+m+2 \leq \ell\\
z_{\ell-1} &\text{ if } k+2 \leq \ell \leq k+m+1
\hspace*{1.35cm}\, \end{array} \right\} \, k+2 \leq \ell $
\\
\underline{\hspace{8.5cm}}\\[2mm]
$\left.\begin{array}{l l} z_\ell &\text{ if } \ell -1\notin I'_3 \text{ and } j+1 \leq \ell \leq k\\
z_{y_{i-1}+1} \hspace*{0.35cm}&\text{ if } \ell=y_i+1 \text{ and } i\geq 2\\
z_{y_c+1} &\text{ if } \ell=k+1
\hspace*{3.75cm} \end{array} \right\} \, j+1 \leq \ell \leq k+1  $
\end{tabular}  \right.\]

This is written with the assumption that $I'_2 = [k+1,k+m]$,  where $m$ is the number of columns of newly-increased height in $\mu$. This comes from the exact same argument as in the elementary symmetric polynomial multiplication, as $I_2$ is the same there, and we are evaluating at the same $\etamutilde$. In cyclic notation, the action $I'_3(I'_2(I'_1(z)))$ almost acts like the cycle,
\[(k+m+1,k+m,\dotsc, k+1,y_c+1, y_{c-1}+1, \dotsc, y_1+1, j, t_r, t_{r-1}, \dotsc, t_1),\]
except that $z_{t_1}$ is sent to $q^{-1}z_{k+m+1}$. If either of $I'_1$ or $I'_3$ is empty, the relevant entries are simply omitted from the list. Also, some of the terms may actually be the same, e.g., it is possible that $k\in I'_3$ so that $k+1=y_c+1$, so one of those terms would be omitted from the cycle in that case.

We would like to find an action $\dtwo$ for which
\[ \dtwo \gammanu=(\eta \mid \mu) \quad \text{ if and only if } \quad  I'_3 \left(I'_2\left(I'_1\left(\overline{(\eta \mid \mu)}\right) \right)  \right)= \overline{\gammanu}.\]

This action is in some sense the one that `undoes' the permutation above. The following is the action which achieves this aim:
\[ (\dtwo \gammanu)_\ell := \left\{\begin{tabular}{l l} $ \left.\begin{array}{l l} \gammanu_\ell  &\text{ if } 1\leq \ell\leq j-1 \text{ and } \ell\notin I'_1 \\
\gammanu_{t_{i+1}} &\text{ if } \ell=t_i \text{ and } i\neq r\\
\gammanu_{j} \hspace*{1.1cm} &\text{ if } \ell=t_r\\
\gammanu_{y_1+1} &\text{ if } \ell=j
\end{array} \hspace{.65cm} \right\} \, \ell \leq j$ \\
\underline{\hspace{9cm}}\\[2mm]
$\left.\begin{array}{l l} 
\gammanu_\ell &\text{ if } k+m+2 \leq \ell \\
\gammanu_{\ell+1} &\text{ if } k+1 \leq \ell \leq k+m \text{ and } \\
(\gammanu_{t_1}) + 1&\text{ if } \ell=k+m+1 
\hspace*{2.5cm} \end{array} \right\} \, k+1 \leq \ell $\\
\underline{\hspace{9cm}}\\[2mm] $\left.\begin{array}{l l} \gammanu_{\ell} &\text{ if } j+1 \leq \ell \leq k \text{ and } \ell-1 \notin I'_3\\
\gammanu_{y_{i+1}+1} \hspace*{.35cm}&\text{ if } \ell=y_i+1 \text{ and } 1 \leq i < c\\
\gammanu_{k+1}&\text{ if } \ell=y_c+1 \hspace*{3.25cm}\\ \end{array} \right\} \, j+1 \leq \ell \leq k $
 \end{tabular}  \right.\]
 
We are looking to evaluate $\Estar{\gammanu}\left(I'_3\left(I'_2\left(I'_1\left(\overline{\etamutilde}\right)\right)\right)\right)$, so we must identify the possible $I'_1$, $I'_2$, and $I'_3$, as well as $\nu$, such that
\[ \dtwo \gammanu =\etamutilde.\] Therefore we find that,
\[ (\nu_2, \dotsc, \nu_{n-k} )= (\mutilde_2,\dotsc, \mutilde_{n-k}).\]
Using the action $\dtwo$, we find that $\nu_1 = \etamutilde_{y_c+1}.$ We give this specially-chosen $\nu$ the name $\nutilde$, so
\[ \nutilde = (\eta_{y_c+1}, \mutilde_2, \dotsc, \mutilde_{n-k}).\]
It can be checked using exactly the same arguments as in \cref{ElementaryMultiplication} that $\nutilde$ is the only permutation of $\lambda$ which has nonzero contribution to \cref{prevanishing}, and that there are unique (maximal) $I'_1$, $I'_2$, and $I'_3$ for which $\dtwo \gammanutilde = \etamutilde$.

Here we note that this choice of indexing causes $\nutilde$ to be the same as $\lambdatilde$.
\subsection{Coefficient Calculations}
Just like in \cref{Piericomputation}, the remaining goal is to compute $\pstarone\left( \overline{\etamutilde}\right)$, $\pstartwo\left( \overline{\etamutilde}\right)$, and $\pstarthree\left( \overline{\etamutilde}\right)$, and our approach here will be the same as well. We start with $\pstartwo\left( \overline{\etamutilde}\right)$, which is the simplest.

The product of operators which contribute to $\pstartwo$ is $H_{k+1} \dotsm H_{n-1}(z_n-t^{-n+1})$. The only difference between this and the formula we found for $p_2$ is that rather than the sum $\displaystyle{\sum_{j={k+1}}^{k+m+1} H_j \dotsm H_{k+m}}$ which acted as $\displaystyle{\sum_{i=0}^m t^i}$, we only have the longest product, $H_{k+1} \dotsm H_{k+m}$, which acts as $t^m$. The rest is exactly identical, and so we obtain,
\[\pstartwo = t^m \cdot (\overline{\mutilde}_{m+1} - t^{-n+1}) \cdot \prod_{j=m+2}^{n-k} \frac{\overline{\mutilde}_{m+1} - t \cdot \overline{\mutilde}_j}{\overline{\mutilde}_{m+1} - \overline{\mutilde}_j} ,\]
or equivalently,
\[\pstartwo = t^m \cdot \frac{1-t}{1-t^{m+1}} \cdot p_2.\]

Next, we deal with $\pstarthree$. 
We are working with the following product:
\[ H_{j}\dotsm H_k = [a(z_j,z_{j+1})+b(z_j,z_{j+1}) s_j] \dotsm [a(z_{k},z_{k+1})+b(z_{k},z_{k+1}) s_{k}] \]
For the choice of $I'_3 = \{y_1, \dotsc, y_c \} \subseteq [j,k]$, consider the contribution from some $H_{y_i}$, which acts as $a(z_{y_i},z_{y_i+1})$. The relevant part of the full $H_j\dots H_k$ which modifies $a(z_{y_i},z_{y_i+1})$ is:
\[ s_{y_{i-1}+1} \dotsm s_{y_i-1} a(z_{y_i},z_{y_i+1}) = a(z_{y_{i-1}+1},z_{y_i+1}) s_{y_{i-1}+1} \dotsm s_{y_i-1},\]
if $i>1$, and if $i=1$, the term $y_0$ should be treated as $j-1$. Similarly, for some $\ell \in [j,k]$ for which $y_i < \ell < y_{i+1}$, we find $H_\ell$ acts as $b(z_\ell,z_{\ell+1})$, and the relevant portion of the full product is:
\[ s_{y_i+1}\dotsm s_{\ell-1} b(z_\ell,z_{\ell+1}) s_\ell = b(z_{y_i+1},z_{\ell+1}) s_{y_i+1}\dotsm s_\ell.\]
If $\ell<y_1$, then again, consider $y_0 < \ell < y_1$ with $y_0=j-1$. Additionally, if $y_c \neq k$, then the remaining operators on the right, $H_{y_c+1} \dotsm H_{k}$, act as
\[\left(\prod_{u=y_c+2}^{k+1} b(z_{y_c+1},z_u)\right) s_{y_c+1} \dotsm s_k. \]
Taking all the contributions from $H_j\dotsm H_k(z)$ together, we obtain:
\[ \pstarthree(z) = \prod_{i=1}^c a\left(z_{y_{i-1}+1},z_{y_i+1}\right) \cdot  \prod_{u=y_c+2}^{k+1} b(z_{y_c+1},z_u) \cdot  \prod_{i=1}^c \prod_{u=y_{i-1}+2}^{y_i} b\left(z_{y_{i-1}+1},z_u \right).\]

The computation for $\pstarone$ is similar to $p_{I_1}$. We give the formula without proof, as the process is like what we did for $\pstarthree$, and again refer to \cite{Ba} where the derivation is almost identical:
\[ \pstarone(z) = a(q^{-1} z_{k+m+1},z_{t_1}) \cdot \prod_{i=1}^{r-1} a(z_{t_i},z_{t_{i+1}}) \cdot \prod_{u=1}^{t_1-1} b(q^{-1} z_{k+m+1},z_u) \cdot \prod_{i=1}^r \prod_{u=t_i+1}^{t_{i+1}-1} b(z_{t_i},z_u),\]
where $t_{r+1}:=j+1$. To unify the formulas of $\pstarone$ and $\pstarthree$ and compare them to $p_{I_1}$, we use the following indexing:
\[ I_{13} := \{(ty)_1, \dotsc, (ty)_{c+r}\} \quad \text{ such that } (ty)_i = \begin{cases} t_i & \text{ if } 1\leq i \leq r\\
y_{i-r}+1 &\text{ if } r+1 \leq i \leq r+c\end{cases}.\]
Expressing $\pstarthree$ in these terms and rearranging the products turns this into,
\[a(z_{j},z_{(ty)_{r+1}})\cdot \prod_{i=r+1}^{r+c-1} a\left(z_{(ty)_{i}},z_{(ty)_{i+1}}\right)  \cdot \prod_{u=j+1}^{(ty)_{r+1}-1} b(z_j,z_u) \cdot  \prod_{i=r+1}^{r+c}\, \prod_{u=(ty)_{i}+1}^{(ty)_{i+1}-1} b\left(z_{(ty)_i},z_u \right),\]
with $(ty)_{r+c+1}=k+2$.

Now the product $\pstarone \cdot \pstarthree$ can be simplified to,
\begin{align*} \pstarone(z) \cdot \pstarthree(z)  = \,&a(q^{-1}z_{k+m+1},z_{t_1})\, \cdot \,  \frac{a(z_j,z_{(ty)_{r+1}})}{a(z_{(ty)_r},z_{(ty)_{r+1}})} \prod_{i=1}^{r+c-1} a(z_{(ty)_i},z_{(ty)_{i+1}})  \\ & \quad \cdot \prod_{u=1}^{t_1-1} b(q^{-1} z_{k+m+1},z_u) \cdot \frac{\prod\limits_{u=j+1}^{(ty)_{r+1}-1} b(z_j,z_u)}{\prod\limits_{u=(ty)_r+1}^{(ty)_{r+1}-1} b(z_{(ty)_r},z_u)} \cdot  \prod_{i=1}^{r+c}\, \prod_{u=(ty)_{i}+1}^{(ty)_{i+1}-1} b\left(z_{(ty)_i},z_u \right) \end{align*}
Written this way, it is easier to compare to $p_{I_1}$ where $I_1 = \{ t_1, \dotsc, t_r, y_1+1,\dotsc,y_c+1 \}$:
\[ \pstarone(z) \cdot \pstarthree(z) = p_{I_1}(z) \cdot\left( \frac{a(z_j,z_{(ty)_{r+1}})}{a(z_{(ty)_r},z_{(ty)_{r+1}})} \right)\cdot \left( \frac{\prod\limits_{u=j+1}^{(ty)_{r+1}-1} b(z_j,z_u)}{\prod\limits_{u=(ty)_r+1}^{(ty)_{r+1}-1} b(z_{(ty)_r},z_u)}\right). \]
Expanding out the $a$'s and $b$'s,
\[ \pstarone(z) \cdot \pstarthree(z) = p_{I_1}(z) \cdot \frac{z_j(z_{(ty)_r}-z_{(ty)_{r+1}})}{z_{(ty)_r} (z_j-z_{(ty)_{r+1}})} \cdot \prod_{u=j+1}^{(ty)_{r+1}-1} \frac{z_j-tz_u}{z_j-z_u} \cdot \prod_{u=(ty)_r+1}^{(ty)_{r+1}-1} \frac{z_{(ty)_r}-z_u}{z_{(ty)_r}-tz_u}.\]
And finally, evaluating $\pstarone\left( \overline{\etamutilde} \right) \cdot \pstarthree\left( \overline{\etamutilde} \right)$, we find,
\[\pstarone\left( \overline{\etamutilde} \right) \cdot \pstarthree\left( \overline{\etamutilde} \right) = p_{I_1}\left( \overline{\etamutilde} \right) \cdot \frac{\overline{\eta}_j(\overline{\eta}_{t_r} - \overline{\eta}_{y_1})}{\overline{\eta}_{t_r}(\overline{\eta}_j-\overline{\eta}_{y_1})}\cdot \prod_{u=j+1}^{y_1-1} \frac{\overline{\eta}_j-t\overline{\eta}_u}{\overline{\eta}_j-\overline{\eta}_u} \cdot \prod_{u=t_r+1}^{y_1-1} \frac{\overline{\eta}_{t_r}-\overline{\eta}_u}{\overline{\eta}_{t_r}-t\overline{\eta}_u}. \]

It is worth noting that if $t_r=j$, then this simplifies to just $p_{I_1}\left(\overline{\etamutilde} \right)$. Combining all these simplifications with \cref{prevanishing}, using the vanishing property of the interpolation polynomials, and comparing with the result from \cref{Pieriformula}, we obtain the expression:
\begin{align*}\label{nonsymmetricpieri}     \DBaratta &=  \frac{f_{\lambdatilde,\gammalambdamin}}{f_{\mutilde,\etamumin}} \cdot \left( \overline{\gammanu}_j\right) \cdot \pstarone \cdot \pstartwo \cdot \pstarthree \cdot \frac{  \Estar{\gammalambdatilde} \left( \overline{\gammalambdatilde}\right)}{\Estar{\etamutilde}{\left(\overline{\etamutilde}\right)}}\\
&= C_{\etamumin}^{\gammalambdamin} \cdot \frac{ \overline{\gammanu}_j}{ \overline{\gammanu}_{k+1}} \cdot t^m \cdot \frac{1-t}{1-t^{m+1}} \cdot  \frac{\overline{\eta}_j(\overline{\eta}_{t_r} - \overline{\eta}_{y_1})}{\overline{\eta}_{t_r}(\overline{\eta}_j-\overline{\eta}_{y_1})}\cdot \prod_{u=j+1}^{y_1-1} \frac{\overline{\eta}_j-t\overline{\eta}_u}{\overline{\eta}_j-\overline{\eta}_u} \cdot \prod_{u=t_r+1}^{y_1-1} \frac{\overline{\eta}_{t_r}-\overline{\eta}_u}{\overline{\eta}_{t_r}-t\overline{\eta}_u}.\end{align*}

\section{Proof of \cref{prop2conversion}}\label{SecondSimplification}
\raggedright \textbf{\cref{prop2conversion}}
\begin{align*} \begin{split}
\jfour \cdot \jfive \cdot  \Estarfour &\cdot \Estarfive \cdot \prod_{j=t_r+1}^k \frac{q^{-1} \overline{\mutilde_h} -t\overline{\eta}_j}{q^{-1} \overline{\mutilde_h} -\overline{\eta}_j} = \prod_{j=t_r+1}^k \frac{(q^{-1}t) \overline{\mutilde_h} -\overline{\eta}_j}{(q^{-1}) \overline{\mutilde_h} -\overline{\eta}_j}  \\ &\cdot \prod\limits_{\substack{\square \in \group{5}'_E \\ \square = (h,i) \\ i\neq \mutilde_h}} \frac{1-q^{-\leg{\mutildeeta}}t^{-\armone{\mutildeeta}}}{1-q^{-(\leg{\mutildeeta}+1)}t^{-\armtwo{\mutildeeta}}} \cdot \left(\frac{1}{1-t}\right) \cdot \left(\frac{1}{1-q^{-1}t^{-a_\mutildeeta(h,\mutilde_h)}}\right)\end{split}\end{align*}
\begin{proof} 
We approach the proof of this proposition much in the same way as \cref{prop1conversion}. The major difference will be that the boxes contributing to $\jfive$ now have the modified arm lengths. This naturally breaks into two cases again, the first case cancelling very similarly to those in \cref{prop1conversion}, and the second one which leaves us with the extra diagram statistic terms on the right side of the equation.

Before getting into the cases, we point out that the final two fractions in the right side of the equation, $\displaystyle{\frac{1}{1-t}}$ and $\displaystyle{\frac{1}{1-q^{-1}t^{-a_\mutildeeta(h,\mutilde_h)}}}$, are the contributions of the newly added box $(h,\mutilde_h)$ to $\jfive$ and $\Estarfive$ respectively. These do not cancel with the eigenvalue rational functions like the terms from other boxes do.\\

\underline{\textbf{Case 1:}} Let $j>t_r$, and suppose that $\eta_j \geq \mutilde_h$. We will show that the term coming from $\jfour$ for the box $(j+n-k,\mutilde_h)$ cancels with $\displaystyle{\frac{q^{-1} \overline{\mutilde_h} -t\overline{\eta}_j}{q^{-1} \overline{\mutilde_h} -\overline{\eta}_j}}$, and the corresponding term coming from $\Estarfour$ replaces it with $\displaystyle{\frac{(q^{-1}t) \overline{\mutilde_h} -\overline{\eta}_j}{(q^{-1}) \overline{\mutilde_h} -\overline{\eta}_j}}$. The first observation is that $\eta_j = \mutilde_h + \leg{\mutildeeta}$. As we did before, we write $\lj$ and $\lh$ to mean $\ell'_\mutildeeta(j+n-k)$ and $\ell'_\mutildeeta(h)$ respectively. The following table shows column heights where $\lh$ and $\lj$ differ:

\begin{center}\begin{tabular}{|c|c|c|c|}
     \hline
     & Counted in $\lj$ and $\lh$ & Counted only in $\lh$ & Counted in neither   \\ \hline
     Left of both& $>\eta_j$ & $\mutilde_h+1, \dotsc, \eta_j$& $\leq \mutilde_h$ \\ \hline
     Between $\mutilde_h$ and $\eta_j$ & $>\eta_j$ & $\mutilde_h, \dotsc, \eta_j$& $<\mutilde_h$\\ \hline 
     Right of both & $\geq \eta_j$ & $\mutilde_h, \dotsc, \eta_j-1$ & $<\mutilde_h$\\ \hline 
\end{tabular}
\end{center}

Like before, $\lh$ also counts the column $\eta_j$, and $\armtwo{\mutildeeta}$ counts the column $\mutilde_h$. Superficially, there appear to be two differences between the boxes only counted in $\lh$ and $\armtwo{\mutildeeta}$. The first is that $\armtwo{\mutildeeta}$ counts boxes left of $\mutilde_h$ of height $\mutilde_h$, and $\lh$ does not, but by construction, no such columns exist. The second is that only $\armtwo{\mutildeeta}$ counts columns to the right of $\eta_j$ of height $\mutilde_h-1$, but such a column does not exist by maximality of $I_1$. This means we have $\lh = \lj+\armtwo{\mutildeeta}$. Comparing eigenvalues, we also get $\overline{\eta}_j = \overline{\mutilde}_h \cdot q^\leg{\mutildeeta} t^\armtwo{\mutildeeta}$, again with $\square = (j,\mutilde_h)$. The comparison leads to the following conversion:
\begin{align*}
    \frac{q^{-1} \overline{\mutilde_h} -t\overline{\eta}_j}{q^{-1} \overline{\mutilde_h} -\overline{\eta}_j} &= \frac{q^{-1} \overline{\mutilde_h} - \overline{\mutilde}_h \cdot q^\leg{\mutildeeta} t^{\armtwo{\mutildeeta}+1}}{q^{-1} \overline{\mutilde_h} -\overline{\mutilde}_h \cdot q^\leg{\mutildeeta} t^\armtwo{\mutildeeta}}\\
    &= \frac{1-q^{\leg{\mutildeeta}+1} t^{\armtwo{\mutildeeta}+1}}{1-q^{\leg{\mutildeeta}+1} t^\armtwo{\mutildeeta}}
\end{align*}

Then the power of $t$ in the denominator can be modified, using the fact that $a_\mutildeeta(j,\mutilde_h)$ counts one more column than $a_\lambdatildegamma(j,\mutilde_h)$, which is the column which wraps around and becomes $\mutilde_h$. As a result,
\[ \frac{q^{-1} \overline{\mutilde_h} -t\overline{\eta}_j}{q^{-1} \overline{\mutilde_h} -\overline{\eta}_j} = \frac{1-q^{\leg{\mutildeeta}+1} t^{\armtwo{\mutildeeta}+1}}{1-q^{\leg{\mutildeeta}+1} t^{\armtwo{\lambdatildegamma}+1}},\]
which cancels identically with the related term for this box in $\jfour$. And the rational function coming from the box in $\Estarfour$ can be converted back into an eigenvalue function,
\begin{align*}
    \frac{1-q^{-(\leg{\lambdatildegamma}+1)} t^{-\armtwo{\lambdatildegamma}}}{1-q^{-(\leg{\mutildeeta}+1)} t^{-\armtwo{\mutildeeta}}} &= \frac{1-q^{-(\leg{\mutildeeta}+1)} t^{-(\armtwo{\mutildeeta}-1)}}{1-q^{-(\leg{\mutildeeta}+1)} t^{-\armtwo{\mutildeeta}}}\\
    &= \frac{\overline{\eta}_j-\overline{\eta}_j\cdot (q^{-1}t)q^{-\leg{\mutildeeta}} t^{-\armtwo{\mutildeeta}}}{\overline{\eta}_j-\overline{\eta}_j\cdot (q^{-1})q^{-\leg{\mutildeeta}} t^{-\armtwo{\mutildeeta}}}\\
    &=\frac{\overline{\eta}_j- (q^{-1}t)\overline{\mutilde}_h}{\overline{\eta}_j- (q^{-1})\overline{\mutilde}_h}.
\end{align*}
In summary, this case takes care of:
\begin{enumerate}
    \item All terms in $\jfour$ and $\Estarfour$
    \item The rational functions $\displaystyle{ \frac{q^{-1} \overline{\mutilde_h} -t\overline{\eta}_j}{q^{-1} \overline{\mutilde_h} -\overline{\eta}_j}}$ on the left and $\displaystyle{\frac{\overline{\eta}_j- (q^{-1}t)\overline{\mutilde}_h}{\overline{\eta}_j- (q^{-1})\overline{\mutilde}_h}}$ on the right of the equation for all values $j$ such that $\eta_j\geq \mutilde_h$.\\
\end{enumerate}

\underline{\textbf{Case 2:}} Like case 1, the work in this case mirrors case 2 of the proof of \cref{prop1conversion}. So let $(h,s)$ be a box in the diagram $\mutildeeta$ with $s<\mutilde_h$. Let $\eta_{a_1},\dotsc,\eta_{a_{g_{s-1}}}$ be the columns of height $s-1$, with $t_r < a_1 < \dotsm < a_{g_{s-1}}$, so $g_{s-1} = \# \{ i > t_r \, \mid \, \eta_i = s-1\}$. Suppose $\eta_j \in \{\eta_{a_1}, \dotsc, \eta_{a_{g_{s-1}}}\}$. Let $\square = (h,\eta_j+1)$, or equivalently $\square = (h,s)$. Note that $\mutilde_h = \eta_j + \leg{\mutildeeta}+1$. Comparing $\lj$ and $\lh$, we get the following:

\begin{center}\begin{tabular}{|c|c|c|c|}
     \hline
     & Counted in $\lj$ and $\lh$ & Counted only in $\lj$ & Counted in neither   \\ \hline
     Left of both & $>\mutilde_h$ & $\eta_j+1, \dotsc, \mutilde_h $ & $\leq \eta_j$\\ \hline
     Between $\mutilde_h$ and $\eta_j$ & $\geq \mutilde_h$ & $\eta_j+1,\dotsc, \mutilde_h-1$ & $\leq \eta_j$\\ \hline 
     Right of both & $\geq \mutilde_h$ & $\eta_j, \dotsc, \mutilde_h-1$ & $<\eta_j$ \\ \hline 
\end{tabular}
\end{center}

Additionally, $\lj$ counts $\mutilde_h$. Unlike before, we would like to compare $\lj$ and $\lh$ using the modified arm, $\armone{\mutildeeta}$. The formula for this is,
\[\armone{\nu} := \# \{1\leq r < i \mid j \leq \nu_r \leq \nu_i \} + \# \{i<r \leq n \mid j \leq \nu_r < v_i \}\]
Like in the previous proof's case 2, this does not count columns to the right of $\eta_j$ of the same height $s-1$. Taking into account this and the additional column $\mutilde_h$ counted in $\lj$, we have the relation,
\[ \lj =\lh + (\armone{\mutildeeta}+1)+\# \{ j < i \leq k \, \mid \, \eta_i = \eta_j \}.\]

Now we use this information to find eigenvalue relations for each column of height $s-1$,
\begin{align*}
    \overline{\eta}_{a_{g_{s-1}}} &= \overline{\mutilde}_h \cdot q^{-(\leg{\mutildeeta}+1)}t^{-(\armone{\mutildeeta}+1)}\\
     \overline{\eta}_{a_{g_{s-1}-1}} &= \overline{\mutilde}_h \cdot q^{-(\leg{\mutildeeta}+1)}t^{-(\armone{\mutildeeta}+2)}\\
     &\vdots \\
      \overline{\eta}_{a_1} &= \overline{\mutilde}_h \cdot q^{-(\leg{\mutildeeta}+1)}t^{-(\armone{\mutildeeta}+1+g_{s-1}-1)}\\
\end{align*}

Next, we simplify the product of rational functions,
\begin{align*}
   &\frac{q^{-1}\overline{\mutilde}_h-t \overline{\eta}_{a_{g_{s-1}}}}{q^{-1}\overline{\mutilde}_h- \overline{\eta}_{a_{g_{s-1}}}} \dotsm \frac{q^{-1}\overline{\mutilde}_h-t \overline{\eta}_{a_1}}{q^{-1}\overline{\mutilde}_h- \overline{\eta}_{a_1}} \\
   &= \frac{1-q^{-\leg{\mutildeeta}}t^{-\armone{\mutildeeta}}}{1-q^{-\leg{\mutildeeta}}t^{-(\armone{\mutildeeta}+1)}}\cdot  \frac{1-q^{-\leg{\mutildeeta}}t^{-(\armone{\mutildeeta}+1)}}{1-q^{-\leg{\mutildeeta}}t^{-(\armone{\mutildeeta}+2)}} \dotsm \frac{1-q^{-\leg{\mutildeeta}}t^{-(\armone{\mutildeeta}+g_{s-1}-1)}}{1-q^{-\leg{\mutildeeta}}t^{-(\armone{\mutildeeta}+g_{s-1})}} \\
   &= \frac{1-q^{-\leg{\mutildeeta}}t^{-\armone{\mutildeeta}}}{1-q^{-\leg{\mutildeeta}}t^{-(\armone{\mutildeeta}+g_{s-1})}} 
\end{align*}

Unfortunately, unlike the other cases, this does not fully cancel with the term coming from $\Estarfive$. Recall that the contribution of that box to $\Estarfive$ is
\[ \frac{1-q^{-\leg{\mutildeeta}}t^{-(\armone{\mutildeeta}+g_{i-1})}}{1-q^{-(\leg{\mutildeeta}+1)}t^{-\armtwo{\mutildeeta}}}.\]
So the $\Estarfive$ numerator cancels with the rational function denominator, and we are left with
\[ \frac{1-q^{-\leg{\mutildeeta}}t^{-\armone{\mutildeeta}}}{1-q^{-(\leg{\mutildeeta}+1)}t^{-\armtwo{\mutildeeta}}}\]

We can do a similar thing with the corresponding rational function on the right side of the desired  equation, \cref{PI1conversion2},
\begin{align*}
    &\frac{(q^{-1}t)\overline{\mutilde}_h- \overline{\eta}_{a_{g_{s-1}}}}{q^{-1}\overline{\mutilde}_h- \overline{\eta}_{a_{g_{s-1}}}} \dotsm \frac{(q^{-1}t)\overline{\mutilde}_h- \overline{\eta}_{a_1}}{q^{-1}\overline{\mutilde}_h- \overline{\eta}_{a_1}} \\
    &=\frac{(q^{-1}t)\overline{\eta}_{a_{g_{s-1}}} \cdot q^{\leg{\mutildeeta}+1}t^{\armone{\mutildeeta}+1}- \overline{\eta}_{a_{g_{s-1}}}}{q^{-1}\overline{\eta}_{a_{g_{s-1}}} \cdot q^{\leg{\mutildeeta}+1}t^{\armone{\mutildeeta}+1}- \overline{\eta}_{a_{g_{s-1}}}} \dotsm \frac{(q^{-1}t)\overline{\eta}_{a_1} \cdot q^{\leg{\mutildeeta}+1}t^{\armone{\mutildeeta}+g_{s-1}}- \overline{\eta}_{a_1}}{q^{-1}\overline{\eta}_{a_1} \cdot q^{\leg{\mutildeeta}+1}t^{\armone{\mutildeeta}+g_{s-1}}- \overline{\eta}_{a_1}} \\
    &= \frac{1-q^\leg{\mutildeeta}t^{\armone{\mutildeeta}+1+g_{s-1}}}{1-q^\leg{\mutildeeta}t^{\armone{\mutildeeta}+1}}.
\end{align*}

Observe that this matches identically with the contribution of the box to $\jfive$. Then to summarize our last case, we have accounted for the following parts of the equation \cref{PI1conversion2}:
\begin{enumerate}
    \item $j_5$ and $\Estarfive$
    \item The rational functions $\displaystyle{ \frac{q^{-1} \overline{\mutilde_h} -t\overline{\eta}_j}{q^{-1} \overline{\mutilde_h} -\overline{\eta}_j}}$ on the left and $\displaystyle{\frac{\overline{\eta}_j- (q^{-1}t)\overline{\mutilde}_h}{\overline{\eta}_j- (q^{-1})\overline{\mutilde}_h}}$ on the right of the equation for all values $j$ such that $\eta_j < \mutilde_h$
    \item The extra product with diagram statistics on the right, $\prod\limits_{\substack{\square \in \group{5}'_E \\ \square = (h,i) \\ i\neq \mutilde_h}} \frac{1-q^{-\leg{\mutildeeta}}t^{-\armone{\mutildeeta}}}{1-q^{-(\leg{\mutildeeta}+1)}t^{-\armtwo{\mutildeeta}}}$.
\end{enumerate}
That takes care of the remainder of the pieces of \cref{PI1conversion2}. \qedhere
\end{proof}

\end{document}